\documentclass[11pt,letterpaper]{article}

\usepackage{amssymb,amsmath,amsfonts}
\usepackage{graphicx,xcolor,enumitem}
\usepackage{epsfig}
\usepackage{amsthm}
\usepackage{bm}
\usepackage{subcaption}
\usepackage{nicematrix}
\usepackage[round]{natbib}
\usepackage{csquotes}

\usepackage{multirow}
\usepackage{algorithm, algorithmic}

\usepackage{diagbox}

\renewcommand{\mkbegdispquote}[2]{\itshape}
\usepackage[twoside, hmarginratio=1:1, vmarginratio=1:1, left=1in,top=1in]{geometry}

\RequirePackage[breaklinks=true, hidelinks]{hyperref}
\usepackage{breakcites}

\newcommand{\twoheadmapsto}{\mathrel{\mapstochar\twoheadrightarrow}}

\newcommand{\cA}{\mathcal{A}}

\newcommand{\cD}{\mathcal{D}}
\newcommand{\E}{\mathbb{E}}
\newcommand{\btau}{\bar{\tau}}
\newcommand{\cH}{\mathcal{H}}
\newcommand{\bF}{\mathbb{F}}
\newcommand{\R}{\mathbb{R}}
\newcommand{\p}{\mathbb{P}}

\newcommand{\cL}{{\mathcal L}}
\newcommand{\cS}{{\mathcal S}}

\newcommand{\barS}{\overline{\mathcal{S}}}
\newcommand{\cF}{{\mathcal F}}

\newcommand{\cV}{{\mathcal V}}
\newcommand{\cC}{{\mathcal C}}
\newcommand{\cM}{{\mathcal M}}

\newcommand{\cI}{{\mathcal I}}

\newcommand{\one}{\mathbf{1}}

\newtheorem{theorem}{Theorem}

\newtheorem{definition}[theorem]{Definition}

\newtheorem{lemma}[theorem]{Lemma}

\newtheorem{proposition}[theorem]{Proposition}

\theoremstyle{definition}

\numberwithin{equation}{section}
\numberwithin{theorem}{section}

\begin{document}

\title{Goal-based portfolio selection with fixed transaction costs\footnote{Erhan Bayraktar is partially supported by the National Science Foundation under grant DMS-2507940 and by the Susan M. Smith chair. Bingyan Han is partially supported by The Hong Kong University of Science and Technology (Guangzhou) Start-up Fund G0101000197 and the Guangzhou-HKUST(GZ) Joint Funding Scheme (No. 2025A03J3858). Jingjie Zhang is supported by the National Natural Science Foundation of China under Grant No.12201113.}}

\author{
	Erhan Bayraktar\thanks{Department of Mathematics, University of Michigan, Ann Arbor, Email: erhan@umich.edu.}
	\and Bingyan Han\thanks{Thrust of Financial Technology, The Hong Kong University of Science and Technology (Guangzhou), Email: bingyanhan@hkust-gz.edu.cn.}
	\and Jingjie Zhang\thanks{University of International Business and Economics, Email: jingjie.zhang@uibe.edu.cn.}
}

\date{\today}
\maketitle

\begin{abstract}
	We study a goal-based portfolio selection problem in which an investor aims to meet multiple financial goals, each with a specific deadline and target amount. Trading the stock incurs a strictly positive transaction cost. Using the stochastic Perron's method, we show that the value function is the unique viscosity solution to a system of quasi-variational inequalities. The existence of an optimal trading strategy and goal funding scheme is established. Numerical results reveal complex optimal trading regions and show that the optimal investment strategy differs substantially from the V-shaped strategy observed in the frictionless case.
	\\[2ex] 
	\\[2ex]
	%\noindent{\textbf {JEL Classification:}} 
\end{abstract}

\section{Introduction}
% Motivation

Portfolio selection has long been a central topic in financial research. Classical frameworks, including Merton's utility maximization and Markowitz's mean-variance model, are built upon several key assumptions. A critical assumption is that investors possess a precise understanding of their own risk aversion and can specify its value without ambiguity. In practice, however, retail investors often find it difficult to quantify their risk preferences. The well-known equity premium puzzle \citep{mehra2003equity} illustrates that it is challenging to identify a reasonable risk aversion coefficient consistent with observed equity premiums and broader economic considerations. Furthermore, a single coefficient is insufficient to capture the diverse investment objectives of individual investors.

Goal-based portfolio selection has emerged as an alternative paradigm for modeling and fulfilling investors' objectives. In this framework, an investor specifies the timing, required funding levels of financial goals and their relative importance. Compared with risk aversion, investors typically have a clearer understanding of their funding needs and the relative importance of different goals. For instance, an investor may know that purchasing a house within a certain price range before a given date is a priority, while a vacation is a less important objective.

The goal-based paradigm has been considered in both the wealth management industry and academia. Platforms such as Schwab and Betterment enable clients to specify goals including retirement plans and home down payments. \cite{gargano2024goal} used data from a FinTech application to demonstrate that setting savings goals increases individual savings rates. \cite{das2010portfolio} investigated separate portfolios for distinct goals and imposed different thresholds on the failure probability associated with each goal. \cite{das2022dynamic} extended this framework by allowing different deadlines and capturing competition among goals, although their model assumes a finite number of states for both strategy and wealth. \cite{capponi2024} introduced a continuous-time framework for multi-goal wealth management, solved using the Hamilton-Jacobi-Bellman (HJB) equation method. \cite{bayraktar2025goal} incorporated mental accounting behavior by assuming that investors construct separate portfolios for each goal, with penalties applied to fund transfers between goals.

% literature on transaction/fixed costs 
An essential aspect of portfolio selection is the inclusion of transaction costs in stock trading. A substantial body of literature has examined investment decisions under market frictions. Proportional transaction costs were first introduced by \cite{magill1976portfolio} in the context of Merton's problem. \cite{davis1990portfolio} demonstrated that the optimal strategies correspond to the local times of a two-dimensional process at the boundaries of a wedge-shaped region. \cite{shreve1994optimal} relaxed several assumptions in \cite{davis1990portfolio} and provided a comprehensive characterization of the value function and optimal strategies. Finite-horizon problems with proportional transaction costs have been investigated in \cite{davis1993european,dai2009finite,belak2019finite}, among others. In addition to the dynamic programming and HJB equation approaches, the duality method has been widely employed to derive structural results and candidate solutions; see, for example, \cite{cvitanic1996hedging,kabanov1999hedging,deelstra2001dual,klein2007duality,Kallsen2010,Czichowsky2016AAP}. Another line of research incorporates fixed transaction costs; see \cite{altarovici2017optimal,belak2019utility,belak2022optimal,bayraktar2022convergence} and references therein. Notably, when transaction costs are small, asymptotic expansions can be derived using homogenization methods \citep{soner2013homogenization,possamai2015homogenization,altarovici2015asymptotics}.

% our technical contributions and findings
A key finding in \cite{capponi2024} is the $V$-shaped investment strategy, which exhibits a non-monotonic relationship between the risk profile and wealth level (see Figure \ref{fig:invest_frictionless} for details). This pattern often results in substantial shifts in stock holdings. Since \cite{capponi2024} assumes a frictionless market, a natural question arises as to whether the $V$-shaped behavior persists when trading incurs costs. In this work, we adopt the goal-based framework of \cite{capponi2024} and consider a financial market with frictions as described in \cite{belak2022optimal}. The cost structure encompasses fixed costs, fixed-plus-proportional costs, and floored or capped costs, which commonly arise in retail investment settings.

Our main contributions and findings are summarized as follows. We employ the stochastic Perron's method to establish that the value function is the unique viscosity solution of a quasi-variational inequality (QVI) system. Early developments of the stochastic Perron's method can be found in \cite{bayraktar2012linear,bayraktar2013stochastic,bayraktar2014Dynkin,bayraktar2015stochastic}. Several essential differences distinguish our results from existing studies in \cite{capponi2024,belak2022optimal}:
\begin{enumerate}[label={(\arabic*)}]
	\item Unlike \cite{belak2022optimal}, demonstrating that the lower stochastic envelope $v_{-}$ is the unique viscosity solution to the QVI system is insufficient in our setting. This distinction stems from the specific structure of the goal-based objective functions.
	\item The expiration of goals at fixed deadlines complicates the proof of the viscosity solution properties. Further details are provided in Lemmas \ref{lem:k_vissub} and \ref{lem:k_vissup}.
	\item The construction of a strict classical subsolution in Lemma \ref{lem:classical_sub} is more delicate, with the difficulty again stemming from the goal-based objectives.
\end{enumerate}
Despite these challenges, one advantage of strictly positive costs is that the existence of an optimal strategy requires only continuity, rather than smoothness, of the value function, similar to the setting in \cite{belak2022optimal}. This property allows for an explicit construction of an optimal strategy, which is presented in Section \ref{sec:strategy}.

In the numerical study, we focus on fixed transaction costs and summarize the main findings as follows:
\begin{enumerate}[label={(\arabic*)}]
	\item The investor must consider both stock and bank account holdings, rather than total wealth alone, when determining the optimal stock exposure. The continuation regions exhibit complex geometries and lack symmetry with respect to the target positions. In particular, a straight continuation region arises when the wealth level is close to the amounts required for both goals, as discussed in Section \ref{sec:straight}.
	\item The optimal strategy in our setting may still allocate the entire wealth to the stock when the total wealth is close to the amount required by the first goal, as shown in Figure \ref{fig:c002t05}. This behavior contrasts with the $V$-shaped strategy observed in the frictionless case. 
	\item Within the continuation region, since no transfer occurs, the optimal funding ratio of the first goal is determined based on the bank account. Figure \ref{fig:fund_equal_w} shows that, under fixed costs, the optimal funding ratios exhibit greater variability for a given level of total wealth.
\end{enumerate}

In contrast to the present paper, \cite{bayraktar2025goal} study a frictionless financial market and introduce penalties for fund transfers between goals. The proof of the viscosity solution property in \cite{bayraktar2025goal} differs substantially in handling goal deadlines and establishing the comparison principle. Furthermore, the incorporation of mental costs in \cite{bayraktar2025goal} results in optimal trading regions that differ from those derived in the current study.

The remainder of this paper is structured as follows. Section \ref{sec:form} introduces the problem formulation and the financial market. Section \ref{sec:HJB} derives the QVI system and presents the first main result, Theorem \ref{thm:viscosity}, which establishes the viscosity solution property of the value function. Sections \ref{sec:sto_super}, \ref{sec:sto_sub}, and \ref{sec:compare} contain the proof of Theorem \ref{thm:viscosity}. Section \ref{sec:strategy} constructs the optimal strategy, and Section \ref{sec:numerics} reports the numerical results. All technical proofs are provided in the Appendix.

\section{Formulation}\label{sec:form}

Assume that an investor has $K$ goals. Each goal $k \in \{1, \ldots, K\}$ requires a target amount $G_k$ by a predetermined deadline $T_k$. For simplicity, assume that the deadlines are distinct and ordered as $T_1 < \ldots < T_k < \ldots < T_K$. For convenience, let $T_0 := 0$ and $T := T_K$. The investment problem therefore spans the time horizon $[0, T]$. The investor constructs a single portfolio to meet each target $G_k$.

Following the financial market framework in \cite{belak2022optimal}, we restate the setting here for completeness. Let $(\Omega, \cF, \p)$ be a probability space supporting a one-dimensional Brownian motion $W := \{W(t) : t \in [0, T] \}$. The filtration $\mathbb{F} := \{\cF_t : t \in [0, T] \}$ denotes the completion of the natural filtration generated by $W$ and satisfies the usual conditions. The financial market consists of a risk-free asset and a single risky asset (stock). Denote by $r$ the constant risk-free interest rate. The stock price process $\{ S(u): u \in [t, T]\}$ evolves according to
\begin{equation}\label{stock}
	d S(u)  = S(u) [\mu du + \sigma dW(u)], 
\end{equation}
where $\mu \in \R$ is the constant drift and $\sigma > 0$ is the constant volatility.

Following \cite{belak2022optimal}, a trading volume $\Delta$ in the stock is assumed to incur a strictly positive transaction cost denoted by $C(\Delta)$. Suppose the transaction cost function $C(\cdot)$ satisfies the following conditions:
\begin{enumerate}[label={(\arabic*)}]	
	\item The function $C(\Delta)$ is continuous and the mapping $|\Delta| \mapsto C(|\Delta|)$ is increasing, implying that transaction costs rise with trading volume. The minimum cost is attained at $\Delta = 0$, with $C_{\min} := C(0) > 0$. 
	\item Suppose the mapping
	\begin{equation}
		\Delta \mapsto \Delta + C(\Delta) 
	\end{equation}
	is strictly increasing on $\R$, and its range contains $[0, \infty)$.
	\item Transactions of size zero $(\Delta = 0)$ are permitted but still incur a positive cost $C_{\min} > 0$. This assumption is made for analytical convenience, as it guarantees the compactness of the feasible set of transactions.
\end{enumerate} 
Typical examples of $C(\Delta)$ include fixed costs, fixed-plus-proportional costs, and other specifications discussed in \cite{belak2022optimal}.

We now introduce the regions representing portfolio positions. Let $x_0$ and $x_1$ denote the dollar amounts invested in the money market and the stock, respectively. The two-dimensional variable $x := (x_0, x_1) \in \R^2$ represents the investor's portfolio position. Throughout this paper, short selling is not permitted in either the money market or the stock. The corresponding set of admissible portfolio positions is denoted by $\barS := [0, \infty)^2$. For later use, define $\cS: = [0, \infty)^2 \backslash \{(0, 0)\}$, which excludes the corner point $(0, 0)$. 

Following a transaction of size $\Delta \in \R$, the portfolio $x = (x_0, x_1)$ is updated according to
\begin{equation}
	(x_0 - \Delta - C(\Delta), x_1 + \Delta) =: \Gamma(x, \Delta),
\end{equation}
where $\Gamma(x, \Delta)$ is referred to as the rebalancing function in \cite{belak2022optimal}. 

Given a portfolio position $x \in \barS$, a transaction $\Delta$ is called feasible if it does not result in short positions in either asset. The set of all feasible transactions is defined by 
\begin{equation}
	D(x) := \{ \Delta \in \R : \Gamma(x, \Delta) \in \barS \}.
\end{equation}
Following \cite{belak2022optimal}, the feasible set $D(x)$ can be simplified. Recall that the mapping $ \Delta  \mapsto \Delta + C(\Delta)$ is strictly increasing, and its range covers $[0, \infty)$. Consequently, there exists a continuous and strictly increasing inverse function $\chi : [0, \infty) \rightarrow \R$. The rebalancing position $\Gamma(x, \Delta)$ belongs to $\barS$ if and only if 
\begin{equation}
	x_0 - \Delta - C(\Delta) \geq 0 \quad \text{and} \quad x_1 + \Delta \geq 0.
\end{equation} 
This condition is equivalent to $\chi(x_0) \geq \Delta$ and $\Delta \geq - x_1$. Hence, the set of feasible transactions can be written as
\begin{equation}\label{eq:D(x)}
	D(x) = [- x_1, \chi(x_0)], \quad x \in \barS.
\end{equation}
When $\chi(x_0) < - x_1$, no feasible transaction exists. The set of portfolio positions without feasible transactions is denoted by 
\begin{equation}
	\cS_\emptyset := \{ x \in \barS: \chi(x_0) < - x_1\}.
\end{equation}
The representation \eqref{eq:D(x)} implies that $D(x) \neq \emptyset$ if and only if $- x_1 \in D(x)$, which is equivalent to $x_0 + x_1 \geq C(-x_1)$; in other words, there is sufficient budget to liquidate the stock position. As noted by \cite{belak2022optimal}, this yields
\begin{equation}
	\cS_\emptyset = \{ x \in \barS : x_0 + x_1 < C(-x_1) \} \supseteq \{ x \in \barS: x_0 + x_1 < C_{\min} \}.
\end{equation}
Therefore, $\cS_\emptyset$ is open relative to $\barS$. The $\barS$-relative boundary of $\cS_\emptyset$ is
\begin{equation}
	\partial \cS_\emptyset = \{x \in \barS : x_0 + x_1 = C(-x_1)\}.
\end{equation}
The closure of $\cS_\emptyset$ is 
\begin{equation}
	\overline{\cS_\emptyset} = \{ x \in \barS : x_0 + x_1 \leq C(-x_1) \}.
\end{equation}

When transaction costs are bounded below by a strictly positive constant, the investor can only trade discretely, as continuous trading would lead to immediate bankruptcy. An investment strategy is represented by a sequence $\Lambda := \{(\tau_n, \Delta_n)\}^\infty_{n=1}$, where $\{ \tau_n \}^\infty_{n=1}$ is an increasing sequence of $\bF$-stopping times representing trading times, and $\Delta_n$ is an $\cF_{\tau_n}$-measurable random variable denoting the volume of the $n$-th trade. In addition to the investment strategy, the investor also needs to determine the dollar amounts allocated to each goal. Let $\theta_k \geq 0$ denote the $\cF_{T_k}$-measurable random variable representing the amount withdrawn from the money account to finance goal $k$.

Starting at time $0$ with an initial portfolio position $x = (x_0, x_1) \in \barS$, the portfolio dynamics $(X_0(s), X_1(s))_{s \in [0, T]}$ are given by
\begin{equation}
	\begin{aligned}
		X_0(s) &= x_0 + \int^s_0 r X_0(u) du - \sum^\infty_{n=1} [ \Delta_n + C(\Delta_n)] \one_{\{\tau_n \leq s \}} - \sum^K_{l=1} \theta_l \one_{\{T_l \leq s\}},  \\
		X_1(s) &= x_1 + \int^s_0 \mu X_1(u) du + \int^s_0 \sigma X_1(u) dW(u) + \sum^\infty_{n=1} \Delta_n \one_{\{ \tau_n \leq s\}}, \quad s \in [0, T]. 
	\end{aligned}
\end{equation} 
For notational simplicity, let $X(s) := (X_0(s), X_1(s))$. Since trading at time $0$ is allowed, the initial condition is interpreted as $X(0-) = x$.

In the general case where the initial time is $t \in [0, T]$ and $X(t-) = x$, the dynamics are given by
\begin{equation}\label{eq:Xt}
	\begin{aligned}
		X_0(s) &= x_0 + \int^s_t r X_0(u) du - \sum^\infty_{n=1} [ \Delta_n + C(\Delta_n)] \one_{\{t \leq \tau_n \leq s \}} - \sum^K_{l=1} \theta_l \one_{\{t \leq T_l \leq s\}},  \\
		X_1(s) &= x_1 + \int^s_t \mu X_1(u) du + \int^s_t \sigma X_1(u) dW(u) + \sum^\infty_{n=1} \Delta_n \one_{\{t \leq \tau_n \leq s\}}, \quad s \in [t, T]. 
	\end{aligned}
\end{equation} 
In particular, at each goal deadline $T_k$ for $k = 1, \ldots, K$, the portfolio dynamics satisfy
\begin{equation}
	\begin{aligned}
		X_0(T_k) &= X_0(T_k -) - \sum^\infty_{n=1} [ \Delta_n + C(\Delta_n)] \one_{\{\tau_n = T_k \}} - \theta_k,  \\
		X_1(T_k) &= X_1(T_k -) + \sum^\infty_{n=1} \Delta_n \one_{\{\tau_n = T_k\}}. 
	\end{aligned}
\end{equation} 
The wealth processes jump due to the withdrawal $\theta_k$ and transfers between the money account and the stock. Depending on the cost structure, executing several smaller trades may be less costly than making a single large trade.

For the final goal $K$, it is assumed that the investor liquidates her stock position whenever doing so does not incur a net loss. The liquidation value of a portfolio $x \in \barS$ is defined as
\begin{equation}
	L(x) := x_0 + (x_1 - C(-x_1))^+.
\end{equation}
Accordingly, the investor is assumed to meet the last goal using the liquidation value.

\begin{definition}[Admissible strategies]\label{def:adm}
	Consider the initial time $t \in [T_{k-1}, T_k]$ for some $k = 1, \ldots, K$ and the initial portfolio position $x = (x_0, x_1) \in \barS$. A trading strategy consists of the withdrawal sequence $\theta_{k:K} = \{\theta_l\}^K_{l=k}$, where $\theta_K$ equals to the liquidation value, and the investment strategy $\Lambda = \{(\tau_n, \Delta_n)\}^\infty_{n=1}$ with $\tau_1 \geq t$. The strategy is called admissible if it does not involve short positions in either the money account or the stock. The set of admissible strategies is denoted by $\cA(t, x; k)$.
\end{definition}
In Definition \ref{def:adm}, when $ k \leq K-1$, the set $\cA(T_k, x; k)$ corresponds to the problem immediately before the expiration of goal $k$ and therefore includes $\theta_k$. In contrast, $\cA(T_k, x; k+1)$ only contains $\theta_{k+1:K}$ and applies to the problem immediately after the expiration of goal $k$. This distinction is crucial for defining the value functions.

For each $k=1, \ldots, K$, a pair $(\btau, \xi)$ is called a {\it random initial condition} for the portfolio process \eqref{eq:Xt} if $\btau \in [T_{k-1}, T]$ is an $\mathbb{F}$-stopping time and $\xi$ is an $\cF_{\btau}$-measurable random variable satisfying $\p(\xi \in \barS) = 1$. For an admissible strategy $(\theta_{k:K}, \Lambda) := (\theta_{k:K}, \{ (\tau_n, \Delta_n) \}^\infty_{n=1})$ with $\tau_1 \geq \btau$, let $\{X(t; \btau, \xi, \theta_{k:K}, \Lambda) \}_{t \in [\btau, T]}$ denote the solution of the portfolio process \eqref{eq:Xt}. The random initial condition $(\btau, \xi)$ is said to be satisfied if
\begin{equation*}
	X(\btau-; \btau, \xi, \theta_{k:K}, \Lambda) = \xi.
\end{equation*}
The strategy $(\theta_{k:K}, \Lambda)$ is called $(\btau, \xi)$-admissible if 
\begin{equation*}
	\p(X(t; \btau, \xi, \theta_{k:K}, \Lambda) \in \barS, \; \btau \leq t \leq T ) = 1.
\end{equation*}

When there is no transfer between the money account and the stock, and only withdrawals $\theta_{k:K}$ are permitted, denote by $\{X(t; \btau, \xi, \theta_{k:K}, \emptyset) \}_{t \in [\btau, T]}$ the corresponding solution of the portfolio process \eqref{eq:Xt}. For later reference, we consider the solution on the interval $[\btau, T_k]$ with $\btau \leq T_k$. Denote by $\{X(t; \btau, \xi, \emptyset, \Lambda) \}_{t \in [\btau, T_k]}$ the solution when the withdrawal $\theta_k$ has not yet been determined. Similarly, the process $\{X(t; \btau, \xi, \emptyset, \emptyset) \}_{t \in [\btau, T_k]}$ represents the uncontrolled state process.

For clarity, we distinguish between processes initialized at time $T_k$. In the process $\{X(t; T_k, x$, $\theta_{k:K}, \Lambda) \}_{t \in [T_k, T]}$, the control variable $\theta_k$ remains active, and the initial position $x$ represents the state before the withdrawal of $\theta_k$. In contrast, in the process $\{X(t; T_k, x, \theta_{k+1:K}, \Lambda) \}_{t \in [T_k, T]}$, the initial position $x$ corresponds to the state after the withdrawal of $\theta_k$. Other analogous notations with $T_k$ as the initial time are interpreted in the same manner.  

Under the admissibility and no-arbitrage conditions, \citet[Lemma A.4]{belak2019utility} shows that the investor trades only finitely many times almost surely within a finite time interval. Moment estimates for $X(\cdot; t, x, \theta_{k:K}, \Lambda)$ can be obtained similarly to \citet[Equation (10)]{belak2022optimal}.

The investor seeks to minimize the shortfalls between the target levels $G_k$ and the funding amounts $\theta_k$, weighted by the importance parameters $w_k > 0$: 
\begin{align}\label{obj0}
	\inf_{(\theta_{1:K}, \Lambda) \in \cA(0, x; 1)}\E \Big[ &\sum^{K}_{k=1} w_k (G_k - \theta_k)^+  \Big].
\end{align}  
As a benchmark, the weight for goal $1$ is set as $w_{1} = 1.0$. To avoid trivial cases, we assume $w_k > 0$ and $G_k > 0$ for all $k = 1, \ldots, K$.

For time $t \in [T_{k-1}, T_{k}]$ with $k = 1, \ldots, K$, the value function is defined as
\begin{equation}\label{eq:ValFun_k}
	\begin{aligned}
		V_k(t, x) := \inf_{(\theta_{k:K}, \Lambda) \in \cA(t, x; k)} \E \Big[& \sum^{K}_{i = k} w_i (G_i - \theta_i)^+ \Big| X(t-; t, x, \theta_{k:K}, \Lambda) = x \Big].
	\end{aligned}
\end{equation}
The value function $V_k(t, x)$ applies when the goals $k, \ldots, K$ are active. At the deadline $T_k$ with $k \leq K-1$, both $V_k(T_k, x)$ and $V_{k+1}(T_k, x)$ are defined, representing the optimal objective values immediately before and after the deadline $T_k$, respectively. Specifically, $V_k(T_k, x)$ optimizes over $(\theta_{k:K}, \Lambda) \in \cA(T_k, x; k)$, while $V_{k+1}(T_k, x)$ optimizes over $(\theta_{k+1:K}, \Lambda) \in \cA(T_k, x; k+1)$.

\section{The QVI system}\label{sec:HJB}
In contrast to \cite{capponi2024}, we define the value function as an array of \eqref{eq:ValFun_k}:
\begin{equation}\label{value}
	(\{ V_1(t, x) \}_{t \in [0, T_1]}, \ldots, \{ V_k(t, x) \}_{t \in [T_{k-1}, T_{k}]}, \ldots, \{ V_{K}(t, x) \}_{t \in [T_{K-1}, T]} ),
\end{equation}
which facilitates the analysis of terminal conditions at $T_k$, $k=1, \ldots, K-1$. Under the framework of \cite{capponi2024}, our $ V_k(T_k, x)$ corresponds to $V(T_k -, x)$ in their notation.

To introduce the QVI system, we define the infinitesimal generator as
\begin{equation}
	\cL[V_k](t, x) := - \frac{\partial V_k}{\partial t} - r x_0 \frac{\partial V_k}{\partial x_0} - \mu x_1 \frac{\partial V_k}{\partial x_1} - \frac{1}{2} \sigma^2 x^2_1 \frac{\partial^2 V_k}{\partial x^2_1}.
\end{equation}
For a locally bounded function $V_k(t, x)$, the intervention operator is defined by
\begin{eqnarray}
	\cM[V_k](t, x) = \left\{\begin{array}{rcl}
		& \inf_{\Delta \in D(x)} V_k (t, \Gamma(x, \Delta)), & \text{ if } D(x) \neq \emptyset, \\
		& + \infty, & \text{ if } D(x) = \emptyset.
	\end{array}\right.
\end{eqnarray}

Through a heuristic derivation, the QVI system is given as follows:
\begin{enumerate}[label={(\arabic*)}]
	\item For time $t \in [T_{k-1}, T_{k})$ with $k = 1, \ldots, K$, the goals $k, \ldots, K$ are active. The corresponding QVI is
	\begin{equation}\label{QVI_t}
		\max \Big\{  \cL[V_k](t, x), V_k(t, x) - \cM[V_k](t, x) \Big\} = 0, \quad (t, x) \in [T_{k-1}, T_k) \times \cS.
	\end{equation}

	\item At time $T_k$ with $k = 1, \ldots, K-1$, the boundary condition connecting $V_k(T_k, x)$ and $V_{k+1}(T_k, x)$ is
	\begin{equation}\label{Tk_bc}
		\begin{aligned}
			\max \Big\{ & V_k(T_k, x) - \inf_{0 \leq \theta_k \leq x_0} \left[ w_k (G_k - \theta_k)^+ + V_{k+1}(T_k, x_0 - \theta_k, x_1) \right], \\
			& V_k(T_k, x) - \cM[V_k](T_k, x) \Big\} = 0, \quad x \in \cS.
		\end{aligned}
	\end{equation}
	
	\item At time $T_K$, the terminal condition is
	\begin{equation}\label{TK_bc}
		\begin{aligned}
			\max \Big\{ & V_K(T_K, x) - w_K \left[ G_K - x_0 - (x_1 - C(-x_1))^+ \right]^+ , \\
			& V_K(T_K, x) - \cM[V_K](T_K, x) \Big\} = 0, \quad x \in \cS.
		\end{aligned}
	\end{equation}
	
	\item At the portfolio position $x = (0, 0)$, the boundary condition is
	\begin{equation}\label{eq:cond_0}
		V_k(t, 0) = \sum^{K}_{i = k} w_i G_i, \quad t \in [T_{k-1}, T_k], \quad k = 1, \ldots, K.
	\end{equation}
\end{enumerate}
Since this is the only QVI system considered in the paper, we refer to it simply as {\it the QVI system}.

The first main result of this paper characterizes the value function defined in \eqref{value} with \eqref{eq:ValFun_k} as the unique viscosity solution of the QVI system. We adopt standard notation from the theory of viscosity solutions. For a locally bounded function $v_k$, denote $v^*_k$ as its upper semicontinuous (USC) envelope and $v_{k,*}$ as its lower semicontinuous (LSC) envelope. See \citet[Equation 4.1]{crandall1992user} for the precise definition.

\begin{definition}[Viscosity subsolution]\label{def:vis_sub} 
	Consider an array of functions
	\begin{equation}\label{vis_sub}
		(\{ v_1(t, x) \}_{t \in [0, T_1]}, \ldots, \{ v_k(t, x) \}_{t \in [T_{k-1}, T_{k}]}, \ldots, \{ v_{K}(t, x) \}_{t \in [T_{K-1}, T]} ),
	\end{equation}
	where $v_k(t, x): [T_{k-1}, T_k] \times \barS \rightarrow \R$ is locally bounded for each $k=1, \ldots, K$. The array \eqref{vis_sub} is a viscosity subsolution of the QVI system if the following conditions hold: 
	\begin{enumerate}[label={(\arabic*)}]	
		\item For each $k=1,\ldots, K$,
		\begin{equation}\label{eq:interior_vissub}
			\max \Big\{  \cL[\varphi](\bar{t}, \bar{x}), v^*_k(\bar{t}, \bar{x}) - \cM[v^*_k]^*(\bar{t}, \bar{x}) \Big\} \leq 0,
		\end{equation}
		for all $(\bar{t},\bar{x}) \in [T_{k-1}, T_k) \times \cS$ and for all $\varphi \in C^{1, 2}([T_{k-1}, T_k) \times \cS)$ such that $(\bar{t},\bar{x})$ is a maximum point of $v^*_k - \varphi$.
		
		\item For each $T_k$ with $k = 1, \ldots, K-1$,
		\begin{equation}\label{eq:k_vissub}
			\begin{aligned}
				\max \Big\{ & v^*_k(T_k, x) - \inf_{0 \leq \theta_k \leq x_0} \left[ w_k (G_k - \theta_k)^+ + v^*_{k+1}(T_k, x_0 - \theta_k, x_1) \right], \\
				& v^*_k(T_k, x) - \cM[v^*_k]^*(T_k, x) \Big\} \leq 0,
			\end{aligned}
		\end{equation}
		for all $ x \in \cS$.
		\item At the terminal time $T_K$,
		\begin{equation}\label{eq:TK_vissub}
			\begin{aligned}
				\max \Big\{ & v^*_K(T_K, x) - w_K \left[ G_K - x_0 - (x_1 - C(-x_1))^+ \right]^+ , \\
				& v^*_K(T_K, x) - \cM[v^*_K]^*(T_K, x) \Big\} \leq 0,
			\end{aligned}
		\end{equation}
		for all $ x \in \cS$.
		\item At the boundary $x = (0, 0)$,
		\begin{equation}\label{vissub_bd0} 
			v^*_k(t, 0) \leq \sum^{K}_{i = k} w_i G_i, \quad t \in [T_{k-1}, T_k], \quad k = 1, \ldots, K. 
		\end{equation}
	\end{enumerate}
\end{definition}

\begin{definition}[Viscosity supersolution]\label{def:vis_super} 
Consider an array of functions
	\begin{equation}\label{vis_super}
		(\{ v_1(t, x) \}_{t \in [0, T_1]}, \ldots, \{ v_k(t, x) \}_{t \in [T_{k-1}, T_{k}]}, \ldots, \{ v_{K}(t, x) \}_{t \in [T_{K-1}, T]} ),
	\end{equation}
	where $v_k(t, x): [T_{k-1}, T_k] \times \barS \rightarrow \R$ is locally bounded for each $k=1,\ldots, K$. The array \eqref{vis_super} is a viscosity supersolution of the QVI system if the following conditions hold:
	\begin{enumerate}[label={(\arabic*)}]	
		\item For each $k=1,\ldots, K$,
		\begin{equation}\label{t_vissup}
			\begin{aligned}
				\max \Big\{  \cL[\varphi](\bar{t}, \bar{x}), v_{k, *}(\bar{t}, \bar{x}) - \cM[v_{k, *}]_*(\bar{t}, \bar{x}) \Big\} \geq 0,
			\end{aligned} 
		\end{equation}
		for all $(\bar{t},\bar{x}) \in [T_{k-1}, T_k) \times \cS$ and for all $\varphi \in C^{1, 2}([T_{k-1}, T_k) \times \cS)$ such that $(\bar{t},\bar{x})$ is a minimum point of $v_{k, *} - \varphi$.
		
		\item For each $T_k$ with $k = 1, \ldots, K-1$,
		\begin{equation}\label{k_vissup}
			\begin{aligned}
				\max \Big\{ & v_{k, *}(T_k, x) - \inf_{0 \leq \theta_k \leq x_0} \left[ w_k (G_k - \theta_k)^+ + v_{k+1, *}(T_k, x_0 - \theta_k, x_1) \right], \\
				& v_{k, *}(T_k, x) - \cM[v_{k, *}]_* (T_k, x) \Big\} \geq 0,
			\end{aligned}
		\end{equation}
		for all $ x \in \cS$.
		\item At the terminal time $T_K$,
		\begin{equation}\label{K_vissup}
			\begin{aligned}
				\max \Big\{ & v_{K, *}(T_K, x) - w_K \left[ G_K - x_0 - (x_1 - C(-x_1))^+ \right]^+ , \\
				& v_{K, *}(T_K, x) - \cM[v_{K, *}]_*(T_K, x) \Big\} \geq 0,
			\end{aligned}
		\end{equation}
		for all $ x \in \cS$.
		\item At the boundary $x = (0, 0)$,
		\begin{equation}
			v_{k, *}(t, 0) \geq \sum^{K}_{i = k} w_i G_i, \quad t \in [T_{k-1}, T_k], \quad k = 1, \ldots, K. 
		\end{equation}
	\end{enumerate}
	
\end{definition}

\begin{definition}[Viscosity solution]
	Consider an array of functions
	\begin{equation}\label{vis_sol}
		(\{ v_1(t, x) \}_{t \in [0, T_1]}, \ldots, \{ v_k(t, x) \}_{t \in [T_{k-1}, T_{k}]}, \ldots, \{ v_{K}(t, x) \}_{t \in [T_{K-1}, T]} ),
	\end{equation}
	where $v_k(t, x): [T_{k-1}, T_k] \times \barS \rightarrow \R$ is locally bounded for each $k=1,\ldots, K$. The array \eqref{vis_sol} is a viscosity solution of the QVI system if it is a viscosity subsolution under Definition \ref{def:vis_sub} and a viscosity supersolution under Definition \ref{def:vis_super}.
\end{definition}

The first main result of this paper is stated as follows:
\begin{theorem}\label{thm:viscosity}
	The value function array defined in \eqref{value} is the unique viscosity solution of the QVI system. For each $k = 1, \ldots, K$, the function $V_k(t, x)$ is continuous and bounded on $[T_{k-1}, T_k] \times \barS$.
\end{theorem}

The proof relies on the stochastic Perron's method developed in \cite{bayraktar2013stochastic,bayraktar2015stochastic}. The main advantage of this approach is that it avoids the need to establish the dynamic programming principle (DPP) a priori, instead deriving it after demonstrating that the value function satisfies the viscosity solution property. This method circumvents the technical difficulties and potential gaps in DPP proofs.

Theorem \ref{thm:viscosity} is proved in three steps: 
\begin{enumerate}[label={(\arabic*)}]	
	\item In Section \ref{sec:sto_super}, stochastic supersolutions are defined to bound the value function from above. The infimum of them is called the upper stochastic envelope and is shown to be a viscosity subsolution. 
			
	\item In Section \ref{sec:sto_sub}, stochastic subsolutions are defined to bound the value function from below. The supremum of them is called the lower stochastic envelope and is shown to be a viscosity supersolution. 
			
	\item In Section \ref{sec:compare}, a comparison argument is applied to complete the proof of Theorem \ref{thm:viscosity}.
\end{enumerate}

\section{Stochastic supersolution}\label{sec:sto_super}
In this paper, we fix a constant $p_0 \in (0, 1)$, which serves as the growth rate.

\begin{definition}[Stochastic supersolution]\label{def:sto_super}
	Consider an array of functions
	\begin{equation}\label{sto_super}
		(\{ v_1(t, x) \}_{t \in [0, T_1]}, \ldots, \{ v_k(t, x) \}_{t \in [T_{k-1}, T_{k}]}, \ldots, \{ v_{K}(t, x) \}_{t \in [T_{K-1}, T]} ).
	\end{equation}
	The array \eqref{sto_super} is a stochastic supersolution of the QVI system if the following conditions hold:
	\begin{enumerate}[label={(\arabic*)}]	
		\item For each $k=1,\ldots, K$, the function $v_k(t, x): [T_{k-1}, T_k] \times \barS \rightarrow \R$ is USC.
		\item There exists a constant $c>0$ such that
		\begin{equation*}
			|v_k(t, x)| \leq c (1 + |x|^{p_0}), \quad (t, x) \in [T_{k-1}, T_k] \times \barS, \quad k=1,\ldots, K.
		\end{equation*}
		\item For each $k=1, \ldots K$, consider any random initial condition $(\btau, \xi)$ with $\btau \in [T_{k-1}, T_k]$, $\xi \in \cF_{\btau}$ and $\p(\xi \in \barS) = 1$. There exists a $(\btau, \xi)$-admissible strategy $(\theta_{k:K}, \Lambda)$, such that for all stopping time $\rho \in [\btau, T]$, we have
		\begin{equation*}
			v_k(\btau, \xi) \geq \E \big[ \cH \big([\btau, \rho], v_{k:K}, X(\cdot; \btau, \xi, \theta_{k:K}, \Lambda) \big) \big| \cF_{\btau} \big],
		\end{equation*}
		where
		\begin{align}
			& \cH \big([\btau, \rho], v_{k:K}, X(\cdot; \btau, \xi, \theta_{k:K}, \Lambda) \big) \label{cH} \\
			& \quad := v_{k}(\rho, X(\rho; \btau, \xi, \theta_{k:K}, \Lambda)) \one_{\{\btau \leq \rho < T_k\}} \nonumber \\		
			& \qquad + \sum^{K-1}_{l=k} \Big\{ v_{l+1}(\rho, X(\rho; \btau, \xi, \theta_{k:K}, \Lambda)) +  \sum^{l}_{i = k} w_i (G_i - \theta_i)^+ \Big \} \one_{\{T_l \leq \rho < T_{l+1} \}} \nonumber \\
			& \qquad + \Big\{ \sum^{K}_{i = k} w_i (G_i - \theta_i)^+ \Big \} \one_{\{ \rho = T \}}. \nonumber
		\end{align}
		We refer to $(\theta_{k:K}, \Lambda)$ as a suitable strategy for $v_k$ (and $v_{k+1:K}$ in \eqref{sto_super}) with the random initial condition $(\btau, \xi)$. 	 	
	\end{enumerate}
	Denote by $\cV^+$ the set of stochastic supersolutions. Write $v : = (v_1, \ldots, v_k, \ldots, v_K)$ and use $v \in \cV^+$ to indicate that $v$ is a stochastic supersolution.
\end{definition}
The set $\cV^+$ is nonempty because 
\begin{equation}\label{ex_stosuper}
	v_{k}(t, x) = \sum^{K}_{i = k} w_i G_i, \quad (t, x) \in [T_{k-1}, T_k] \times \barS
\end{equation}
is a stochastic supersolution.

For $k = 1, \ldots, K$ and $(t, x) \in [T_{k-1}, T_k] \times \barS$, define
\begin{equation}\label{eq:sto_upper}
	v_{k, +}(t, x) := \inf \big\{ v_k(t, x) |\; v_k \text{ is the $k$-th element of some } v \in \cV^+ \big\}. 
\end{equation}
The upper stochastic envelope is denoted by $v_{+} := (v_{1, +}, \ldots, v_{k, +}, \ldots, v_{K, +})$. By definition, we can show that $v_{+}$ is an upper bound of the value function:
\begin{equation}\label{eq:val_up}
	v_{k, +}(t, x) \geq V_k(t, x), \quad (t, x) \in [T_{k-1}, T_k] \times \barS.
\end{equation}

Lemma \ref{lem:two_super} below establishes that the family $\cV^+$ of stochastic supersolutions is stable under taking the minimum. The proof follows directly from the definition and is therefore omitted.
\begin{lemma}\label{lem:two_super}
	If $(v^1_1, \ldots, v^1_k, \ldots, v^1_{K})$ and $ (v^2_1, \ldots, v^2_k, \ldots, v^2_{K})$ are stochastic supersolutions, then $(v^1_1 \wedge v^2_1, \ldots, v^1_k \wedge v^2_k, \ldots, v^1_{K} \wedge v^2_{K})$ is also a stochastic supersolution.
\end{lemma}

We now prove the viscosity subsolution property of $v_+$ in Proposition \ref{prop:vissub}. 
\begin{proposition}\label{prop:vissub}
	The upper stochastic envelope $v_{+}$ is a viscosity subsolution of the QVI system under Definition \ref{def:vis_sub}.
\end{proposition}
\begin{proof} 
	The proof proceeds as follows:
	\begin{enumerate}[label={(\arabic*)}]	
		\item Since \eqref{ex_stosuper} is a stochastic supersolution and $v_{k, +}$ is the infimum, Condition (4) at $x = (0, 0)$ holds. 
		\item Condition (3) at $T_K$ is established in Lemma \ref{lem:K_vissub}.
		\item Condition (2) at $T_k$, for $k = 1, \ldots, K-1$, is proved in Lemma \ref{lem:k_vissub}.
		\item Lemma \ref{lem:interior_vissub} verifies Condition (1) in Definition \ref{def:vis_super}, concerning the viscosity supersolution property on $[T_{k-1}, T_k) \times \cS$.
	\end{enumerate}
\end{proof}

\section{Stochastic subsolution}\label{sec:sto_sub}

\begin{definition}[Stochastic subsolution]\label{def:sto_sub}
	Consider an array of functions
	\begin{equation}\label{sto_sub}
		(\{ v_1(t, x) \}_{t \in [0, T_1]}, \ldots, \{ v_k(t, x) \}_{t \in [T_{k-1}, T_{k}]}, \ldots, \{ v_{K}(t, x) \}_{t \in [T_{K-1}, T]} ).
	\end{equation}
	The array \eqref{sto_sub} is a stochastic subsolution of the QVI system if the following conditions hold:
	\begin{enumerate}[label={(\arabic*)}]	
		\item For each $k=1,\ldots, K$, the function $v_k(t, x): [T_{k-1}, T_k] \times \barS \rightarrow \R$ is LSC.
		\item There exists a constant $c>0$ such that
		\begin{equation}
			|v_k(t, x)| \leq c (1 + |x|^{p_0}), \quad (t, x) \in [T_{k-1}, T_k] \times \barS, \quad k=1,\ldots, K.
		\end{equation}
		\item The function $v_k$ is nondecreasing in the direction of transactions, that is,
		\begin{equation}
			v_k(t, x) \leq \cM[v_k](t, x), \quad (t, x) \in [T_{k-1}, T_k] \times \barS, \quad k = 1, \ldots, K.
		\end{equation} 
		\item For each $k=1, \ldots K$, consider any random initial condition $(\btau, \xi)$ with $\btau \in [T_{k-1}, T_k]$, $\xi \in \cF_{\btau}$ and $\p(\xi \in \barS) = 1$. For any $(\btau, \xi)$-admissible withdrawals $\theta_{k:K}$, the following inequality holds:
		\begin{equation}\label{eq:submartingale}
			v_k(\btau, \xi) \leq \E \big[ \cH \big([\btau, \rho], v_{k:K}, X(\cdot; \btau, \xi, \theta_{k:K}, \emptyset) \big) \big| \cF_{\btau} \big]
		\end{equation}
		for any stopping time $\rho \in [\btau, T]$, where $\cH([\btau, \rho], v_{k:K}, X(\cdot; \btau, \xi, \theta_{k:K}, \emptyset))$ is defined in \eqref{cH}.
	\end{enumerate}
	Denote the set of stochastic subsolutions as $\cV^-$.
\end{definition}
Condition (4) implies the following terminal condition for $v_K$ when $\btau = T$, $\xi = x \in \barS$, and $\rho = T$:
\begin{equation}
	v_K(T, x) \leq w_K \left[ G_K - x_0 - (x_1 - C(-x_1))^+ \right]^+, \quad x \in \barS.
\end{equation}
This result uses the assumption that $\theta_K$ liquidates the stock position whenever it does not generate a net loss.

For brevity, we write $v \in \cV^-$ to indicate that $v$ is a stochastic subsolution. Lemma \ref{lem:two_sub} below shows that the family $\cV^-$ is stable under taking the maximum. The proof is omitted since it follows directly from the definition.
\begin{lemma}\label{lem:two_sub}
	If $ (v^1_1, \ldots, v^1_k, \ldots, v^1_{K})$ and $ (v^2_1, \ldots, v^2_k, \ldots, v^2_{K})$ are stochastic subsolutions, then $(v^1_1 \vee v^2_1, \ldots, v^1_k \vee v^2_k, \ldots, v^1_{K} \vee v^2_{K})$ is also a stochastic subsolution.
\end{lemma}

The following example is useful for constructing a strict classical subsolution and proving the comparison principle. The result in Lemma \ref{lem:classical_sub} remains valid if the constant $2$ in $C_k$ is replaced by a larger constant.
\begin{lemma}\label{lem:classical_sub}
	Let constants $a \in \{0, 1\}$, $q \in (0, 1)$, $\lambda > q \max\{r, \mu, 0\}$, and
	\begin{equation}
		C_k = \sum^K_{i = k} 2 w_i G^{1 - q}_{i} e^{\lambda (T_i - T_k)}.
	\end{equation}
	Define
	\begin{equation}\label{eq:Fa}
		F^{a}_k(t, x) = \sum^{K}_{i = k} w_i G_i - C_k (a + x_0 + x_1)^q e^{\lambda(T_k - t)}.
	\end{equation}
	Then there exist continuous functions $\{\kappa^c_k(x) \}^K_{k=1}$ and $\{\kappa^b_k(x) \}^K_{k=1}$, satisfying
	\begin{align*}
		& \kappa^c_k(x) \leq 0, \; \kappa^b_k(x) \leq 0, \quad x \in \barS, \\
		& \kappa^c_k(x) < 0, \; \kappa^b_k(x) < 0, \quad x \in \cS.
	\end{align*}
	Moreover, the following conditions hold:
	\begin{enumerate}[label={(\arabic*)}]	
		\item For each $k=1,\ldots, K$,
		\begin{equation}\label{eq:sub_t}
			\max \Big\{  \cL[F^a_k](t, x), F^a_k(t, x) - \cM[F^a_k](t, x) \Big\} \leq \kappa^c_k(x) < 0,
		\end{equation}
		for all $(t, x) \in [T_{k-1}, T_k) \times \cS$.
		
		\item For each $T_k$ with $k = 1, \ldots, K-1$,
		\begin{equation}\label{eq:subTk}
			\begin{aligned}
				\max \Big\{ & F^a_k(T_k, x) - \inf_{0 \leq \theta_k \leq x_0} \left[ w_k (G_k - \theta_k)^+ + F^a_{k+1}(T_k, x_0 - \theta_k, x_1) \right], \\
				& F^a_k(T_k, x) - \cM[F^a_k](T_k, x) \Big\} \leq \kappa^b_k(x) < 0,
			\end{aligned}
		\end{equation}
		for all $ x \in \cS$.
		\item At the terminal time $T_K$,
		\begin{equation}\label{eq:sub_TK}
			\begin{aligned}
				\max \Big\{ & F^a_K(T_K, x) - w_K \left[ G_K - x_0 - (x_1 - C(-x_1))^+ \right]^+ , \\
				& F^a_K(T_K, x) - \cM[F^a_K](T_K, x) \Big\} \leq \kappa^b_K(x) < 0,
			\end{aligned}
		\end{equation}
		for all $ x \in \cS$.
	\end{enumerate}
\end{lemma}

Based on Lemma \ref{lem:classical_sub}, an example of stochastic subsolutions is given as follows.
\begin{lemma}\label{lem:F0}
The array of functions
	\begin{equation}\label{eq:F0}
		F^0 := (\{ F^0_1(t, x) \}_{t \in [0, T_1]}, \ldots, \{ F^0_k(t, x) \}_{t \in [T_{k-1}, T_{k}]}, \ldots, \{ F^0_{K}(t, x) \}_{t \in [T_{K-1}, T]}),
	\end{equation}
where each element is defined in \eqref{eq:Fa} with $q = p_0$ and $a=0$, is a stochastic subsolution to the QVI system.
\end{lemma}

For each $k = 1, \ldots, K$ and $(t, x) \in [T_{k-1}, T_k] \times \barS$, define
\begin{equation}\label{eq:envelope}
	v_{k, -}(t, x) := \sup \big\{ v_k(t, x) \big|\; v_k \text{ is the $k$-th element of some } v \in \cV^- \big\}. 
\end{equation}
The supremum in \eqref{eq:envelope} is taken over all $v_k$ that can form part of a stochastic subsolution together with some $(v_1, \ldots, v_{k-1}, v_{k+1}, \ldots, v_K)$. Denote the lower stochastic envelope as $$v_{-} := (v_{1,-}, \ldots, v_{k,-}, \ldots, v_{K,-}).$$
The following properties hold for the lower stochastic envelope $v_{-}$: 
\begin{enumerate}[label={(\arabic*)}]	
	\item Stochastic subsolutions do not exceed the value function. For any $v \in \cV^-$, applying Fatou's lemma yields
	\begin{align}
		v_k(t, x) \leq \E \Big[& \sum^{K}_{i = k} w_i (G_i - \theta_i)^+ \Big| X(t-) = x \Big],
	\end{align}
	for any admissible $(\theta_{k:K}, \Lambda) \in \cA(t, x; k)$. Taking the infimum over all admissible controls $(\theta_{k:K}, \Lambda) \in \cA(t, x; k)$ gives
	\begin{equation}\label{eq:upper_hk}
		v_k(t, x) \leq V_k(t, x), \quad (t, x) \in [T_{k-1}, T_k] \times \barS.
	\end{equation}
	Taking the supremum on the left-hand side then implies
	\begin{equation}\label{eq:gkVk}
		v_{k,-}(t, x) \leq V_k(t, x), \quad (t, x) \in [T_{k-1}, T_k] \times \barS.
	\end{equation}
	Since the value function is bounded, \eqref{eq:upper_hk} also shows that stochastic subsolutions are bounded above. Therefore, Condition (2) in Definition \ref{def:sto_sub} can be imposed on the lower side only.
	
	\item The supremum in \eqref{eq:envelope} is attained and $v_{-} \in \cV^-$. The proof follows the argument of \citet[Lemma 3.5]{belak2017impulse}, which relies on the result of \cite{bayraktar2012linear} ensuring that the supremum can be chosen to be countable.
	
	\item Since Lemma \ref{lem:F0} establishes that $F^0$ in \eqref{eq:F0} is a stochastic subsolution, the following boundary condition holds:
	\begin{equation}
		v_{k,-}(t, 0) \geq \sum^{K}_{i = k} w_i G_i, \quad t \in [T_{k-1}, T_k], \quad k = 1, \ldots, K. 
	\end{equation}
	Combining this with \eqref{eq:gkVk} gives
	\begin{equation}\label{eq:gk00}
		v_{k,-}(t, 0) = \sum^{K}_{i = k} w_i G_i, \quad t \in [T_{k-1}, T_k], \quad k = 1, \ldots, K. 
	\end{equation}
\end{enumerate}

We prove the viscosity supersolution property of $v_-$ in Proposition \ref{prop:vissuper}. 

\begin{proposition}\label{prop:vissuper}
	The lower stochastic envelope $v_{-}$ is a viscosity supersolution of the QVI system under Definition \ref{def:vis_super}.
\end{proposition}
\begin{proof} 
	The proof proceeds as follows:
	\begin{enumerate}[label={(\arabic*)}]	
		\item Condition (4) at $x = (0, 0)$ has been verified in \eqref{eq:gk00}. 
		\item Condition (3) at $T_K$ is established in Lemma \ref{lem:K_vissup}.
		\item Condition (2) at $T_k$, for $k = 1, \ldots, K-1$, is proved in Lemma \ref{lem:k_vissup}.
		\item Following arguments similar to those in \citet[Theorem 4.1]{bayraktar2013stochastic}, Condition (1) in Definition \ref{def:vis_super}, which concerns the viscosity supersolution property on $[T_{k-1}, T_k) \times \cS$, can be established. The detailed proof is omitted.
	\end{enumerate}
\end{proof}

\section{Comparison principle}\label{sec:compare}
This section establishes a comparison principle that guarantees the continuity and uniqueness of viscosity solutions to the QVI system. The proof follows the standard approach based on Ishii's lemma \cite[Section 4.4]{pham2009book} and the treatment of the intervention operator $\cM$ described in \cite{belak2019utility,belak2022optimal}. The result is included here for completeness.

\begin{proposition}[Terminal comparison at $T_k$]\label{prop:compareTk}
	Let $k=1, \ldots, K-1$ and consider a continuous and bounded function $f(t, x): [T_k, T_{k+1}] \times \barS \rightarrow \R$. Suppose that the following conditions hold:
	\begin{enumerate}[label={(\arabic*)}]	
		\item The function $u(t, x): [T_{k-1}, T_{k}] \times \barS \rightarrow \R$ is USC and satisfies
		\begin{equation}
			\begin{aligned}
				\max \Big\{ & u(T_k, x) - \inf_{0 \leq \theta_k \leq x_0} \left[ w_k (G_k - \theta_k)^+ + f(T_k, x_0 - \theta_k, x_1) \right], \\
				& u(T_k, x) - \cM[u]^* (T_k, x) \Big\} \leq 0, \quad x \in \cS.
			\end{aligned}
		\end{equation}
		\item The function $v(t, x): [T_{k-1}, T_{k}] \times \barS \rightarrow \R$ is LSC and satisfies
		\begin{equation}
			\begin{aligned}
				\max \Big\{ & v(T_k, x) - \inf_{0 \leq \theta_k \leq x_0} \left[ w_k (G_k - \theta_k)^+ + f(T_k, x_0 - \theta_k, x_1) \right], \\
				& v(T_k, x) - \cM[v]_* (T_k, x) \Big\} \geq 0, \quad x \in \cS.
			\end{aligned}
		\end{equation}
		\item At the corner $0$,
        \begin{equation}\label{cond:0}
            u(T_k, 0) \leq v(T_k, 0), \quad v(T_k, 0) = \sum^K_{i=k} w_i G_i. 
        \end{equation}
        Furthermore, for any $x \in \barS$,
        \begin{equation}\label{eq:growth}
            \begin{aligned}
                0 &\leq u(T_k, x) \leq \sum^K_{i=k} w_i G_i, \\
                -c(1 + |x|^{p_0}) & \leq v(T_k, x) \leq \sum^K_{i=k} w_i G_i \quad \text{ with some constant } c > 0. 
            \end{aligned}
        \end{equation}
	\end{enumerate}
	Then it follows that
	\begin{equation}
		u(T_k, x) \leq v(T_k, x), \quad \forall \; x \in \barS.
	\end{equation}
\end{proposition}

\begin{proposition}[Terminal comparison at $T_K$]\label{prop:compareTK}
	Suppose that the following conditions hold:
	\begin{enumerate}[label={(\arabic*)}]	
		\item  The function $u(t, x): [T_{K-1}, T_{K}] \times \barS \rightarrow \R$ is USC and satisfies
		\begin{equation*}
			\begin{aligned}
				\max \Big\{ & u(T_K, x) - w_K \left[ G_K - x_0 - (x_1 - C(-x_1))^+ \right]^+, \\
				& u(T_K, x) - \cM[u]^* (T_K, x) \Big\} \leq 0, \quad x \in \cS.
			\end{aligned}
		\end{equation*}
		\item The function $v(t, x): [T_{K-1}, T_{K}] \times \barS \rightarrow \R$ is LSC and satisfies
		\begin{equation*}
			\begin{aligned}
				\max \Big\{ & v(T_K, x) - w_K \left[ G_K - x_0 - (x_1 - C(-x_1))^+ \right]^+, \\
				& v(T_K, x) - \cM[v]_* (T_K, x) \Big\} \geq 0, \quad x \in \cS.
			\end{aligned}
		\end{equation*}
		\item At the corner $0$, $$u(T_K, 0) \leq v(T_K, 0), \quad v(T_K, 0) =  w_K G_K.$$ 
        Furthermore, for any $x \in \barS$,
		\begin{align*}
			& 0 \leq u(T_K, x) \leq w_K G_K, \quad -c(1 + |x|^{p_0}) \leq v(T_K, x) \leq w_K G_K \; \text{with some constant $c > 0$.}
		\end{align*}
	\end{enumerate}
	Then it follows that
	\begin{equation*}
		u(T_K, x) \leq v(T_K, x), \quad \forall \; x \in \barS.
	\end{equation*}
\end{proposition}

\begin{proposition}[Comparison principle: $t \in [T_{k-1}, T_k)$]\label{prop:compare}
	Let $k=1, \ldots, K$. Suppose that the following conditions hold:
	\begin{enumerate}[label={(\arabic*)}]	
		\item  The function $u \in USC([T_{k-1}, T_{k}] \times \barS)$ is a viscosity subsolution of \eqref{QVI_t} on $[T_{k-1}, T_k) \times \cS$, that is, the USC function $u$ satisfies \eqref{eq:interior_vissub} in Definition \ref{def:vis_sub}.
		
		\item The function $v \in LSC([T_{k-1}, T_{k}] \times \barS)$ is a viscosity supersolution of \eqref{QVI_t} on $[T_{k-1}, T_k) \times \cS$, that is, the LSC function $v$ satisfies \eqref{t_vissup} in Definition \ref{def:vis_super}.
		\item There exists a constant $c > 0$ such that
		\begin{equation*}
			-c(1 + |x|^{p_0}) \leq v(t, x) \leq \sum^K_{i=k} w_i G_i, \quad (t, x) \in [T_{k-1}, T_{k}] \times \barS.
		\end{equation*}
		Furthermore,
		\begin{align*}
			& 0 \leq u(t, x) \leq \sum^K_{i=k} w_i G_i, \quad (t, x) \in [T_{k-1}, T_{k}] \times \barS, \\
			& u(t, 0) \leq v(t, 0) =  \sum^K_{i=k} w_i G_i, \quad t \in [T_{k-1}, T_k], \\
			& u(T_k, x) \leq v(T_k, x), \quad x \in \barS.
		\end{align*}
	\end{enumerate}
	Then it follows that
	\begin{equation*}
		u(t, x) \leq v(t, x), \quad \forall \; (t, x) \in [T_{k-1}, T_{k}] \times \barS.
	\end{equation*}
\end{proposition}

We now provide the proof of the viscosity solution properties of the value function.
\begin{proof}[Proof of Theorem \ref{thm:viscosity}]
	The argument proceeds by backward induction.
	
	\begin{enumerate}[label={(\arabic*)}]	
		\item At the terminal time $T_K$, Lemma \ref{lem:K_vissup} shows that $v_{K, -}$ is an LSC viscosity supersolution, and Lemma \ref{lem:K_vissub} shows that $v_{K, +}$ is a USC viscosity subsolution. Moreover, $v_{K, -}$ and $v_{K, +}$ satisfy the boundary and growth conditions required in Condition (3) of Proposition \ref{prop:compareTK}, which yields
		\begin{equation}\label{ineqK+-}
			v_{K, +}(T_K, x) \leq v_{K, -} (T_K, x), \quad x \in \overline{\cS}.
		\end{equation}
		As established earlier, $v_{K, -} (T_K, x)  \leq V_K(T_K, x) \leq  v_{K, +}(T_K, x)$ fo all $x \in \overline{\cS}$. Therefore,
		\begin{equation}
			v_{K, -} (T_K, x) = v_{K, +}(T_K, x) = V_K(T_K, x), \; x \in \overline{\cS}.
		\end{equation} 
		Moreover, $V_{K}(T_K, \cdot)$ is continuous and bounded on $\overline{\cS}$. 
		
		\item On the interval $[T_{K-1}, T_K)$, the functions $v_{K, -}$ and $v_{K, +}$ satisfy the boundary condition at $0$, the growth condition, and the viscosity supersolution and subsolution properties, respectively. By \eqref{ineqK+-} and Proposition \ref{prop:compare}, it follows that
		\begin{equation}
			v_{K, +}(t, x) \leq v_{K, -} (t, x), \quad (t, x) \in [T_{K-1}, T_K] \times \overline{\cS}.
		\end{equation} 
		Since $v_{K, -}(t, x) \leq V_K(t, x) \leq v_{K, +} (t, x)$, we obtain
		\begin{equation}
			v_{K, -}(t, x) = v_{K, +} (t, x) = V_K(t, x), \quad (t, x) \in [T_{K-1}, T_K] \times \overline{\cS}.
		\end{equation}
		Moreover, $V_K(t, x)$ is continuous and bounded on $[T_{K-1}, T_K] \times \overline{\cS}$.

		\item We repeat the previous steps for each $k = K-1, \ldots, 1$. Consequently, the value function array \eqref{value} is the unique viscosity solution of the QVI system. Moreover, the value function $V_k(t, x)$ is continuous and bounded on $[T_{k-1}, T_k] \times \overline{\cS}$.
	\end{enumerate}	
\end{proof}

In contrast to \cite{belak2022optimal}, we prove that the value function $V_k$ is the unique viscosity solution to the QVI system, instead of focusing on the lower stochastic envelope $v_{-}$ only. This choice is motivated by the fact that the positivity of $v_{-}$ cannot be established directly from its definition. When perturbing the continuation and intervention regions to construct optimal strategies, the non-negativity of $V_k$ becomes essential. Importantly, the existence of an optimal strategy requires only the continuity, rather than the smoothness, of the value function.

\section{Construction of optimal strategies}\label{sec:strategy}
First, we introduce several optimizers that will be used to construct an optimal strategy. Given $i=1, \ldots, K$, recall that $V_i(t, x)$ is the continuous value function with $t \in [T_{i-1}, T_i]$. The continuation region $\cC_i$ and the intervention region $\cI_i$ are defined as 
\begin{align}
	\cC_i & := \{ (t, x) \in [T_{i-1}, T_i] \times \barS : V_i(t, x) < \cM[V_i] (t, x)\}, \\
	\cI_i & := \{ (t, x) \in [T_{i-1}, T_i] \times \barS : V_i(t, x) = \cM[V_i] (t, x)\}. \label{eq:I_i}
\end{align}

By \citet[Corollary 4]{schal1974selection}, there exists a Borel measurable optimizer $g_i: [T_{i-1}, T_i] \times (\barS \setminus \cS_\emptyset) \rightarrow \R$ satisfying
\begin{equation}
	g_i(t, x) \in D(x) \quad \text{and} \quad \cM[V_i] (t, x) = V_i(t, \Gamma(x, g_i(t, x))),
\end{equation}
for all $(t, x) \in [T_{i-1}, T_i] \times (\barS \setminus \cS_\emptyset)$.

For $i \neq K$, another application of \citet[Corollary 4]{schal1974selection} yields a Borel measurable optimizer $\Theta_i(x): \barS \rightarrow \R$, such that $\Theta_i(x) \in [0, x_0]$ and
\begin{equation}
	\begin{aligned}
		& \inf_{0 \leq \theta_i \leq x_{0}} \left[ w_i (G_i - \theta_i)^+ + V_{i+1}(T_i, x_{0} - \theta_i, x_{1}) \right] \\
		& \quad = w_i (G_i - \Theta_i(x))^+ + V_{i+1}(T_i, x_{0} - \Theta_i(x), x_{1})
	\end{aligned}
\end{equation}
for all $x \in \barS$.

Given $k=1, \ldots, K$ and $(t, x) \in [T_{k-1}, T_k] \times \barS$, our goal is to construct an admissible strategy $(\theta^*_{k:K}, \Lambda^*)  \in \cA(t, x; k)$, such that 
\begin{equation}
	\begin{aligned}
		V_k(t, x) = \E \Big[& \sum^{K}_{i = k} w_i (G_i - \theta^*_i)^+ \Big| X^*(t-) = x \Big].
	\end{aligned}
\end{equation}
This implies that $(\theta^*_{k:K}, \Lambda^*)$ is an optimal strategy. Here, we denote the corresponding wealth process as $X^*(s) := X(s; t, x, \theta^*_{k:K}, \Lambda^*)$, $s \in [t, T]$. Note that $V_k(T_k, x)$ includes the funding amount $\theta^*_k$ for goal $k$, whereas $V_{k+1}(T_k, x)$ excludes $\theta^*_k$ since goal $k$ has expired.

The candidate optimal strategy is constructed recursively. The investment strategy $\Lambda^*$ is partitioned by goal deadlines as $\Lambda^* := (\Lambda^*_k, \ldots, \Lambda^*_K)$, where $\Lambda^*_i := \{ (\tau^{*, i}_n, \Delta^{*, i}_n)\}^\infty_{n=1}$ is specified as follows. For $\Lambda^*_k$, the initial position is set to $(\tau^{*, k}_0, \xi^{*, k}_0) = (t, x)$. For $n = 1, 2, \ldots$, define iteratively
\begin{equation}\label{eq:Lambda_k}
	\begin{aligned}
		\tau^{*, k}_n & := \inf\{ u \in [\tau^{*, k}_{n-1}, T_k] : X(u; \tau^{*, k}_{n-1}, \xi^{*, k}_{n-1}, \emptyset, \emptyset) \in \cI_k \}, \\
		\Delta^{*, k}_n & := g_k(\tau^{*, k}_n, X(\tau^{*, k}_n; \tau^{*, k}_{n-1}, \xi^{*, k}_{n-1}, \emptyset, \emptyset)) \one_{\{\tau^{*, k}_n \leq T_k\}}, \\
		\xi^{*, k}_n & := \Gamma(X(\tau^{*, k}_n; \tau^{*, k}_{n-1}, \xi^{*, k}_{n-1}, \emptyset, \emptyset), \Delta^{*, k}_n). 
	\end{aligned}
\end{equation}

If $k \neq K$, the candidate optimal supporting amount for goal $k$ is given by
\begin{equation}
	\theta^*_k := \Theta_k(X(T_k; \tau^{*, k}_0, \xi^{*, k}_0, \emptyset, \Lambda^*_k)).
\end{equation}

The next component $\Lambda^*_{k+1}$ is constructed with the initial position
\begin{equation}
	(\tau^{*, k+1}_0, \xi^{*, k+1}_0) = (T_k, X(T_k; \tau^{*, k}_0, \xi^{*, k}_0, \theta^*_k, \Lambda^*_k)),
\end{equation}
which satisfies
\begin{align*}
	X_0(T_k; \tau^{*, k}_0, \xi^{*, k}_0, \theta^*_k, \Lambda^*_k) & = X_0(T_k; \tau^{*, k}_0, \xi^{*, k}_0, \emptyset, \Lambda^*_k) - \theta^*_k, \\
	X_1(T_k; \tau^{*, k}_0, \xi^{*, k}_0, \theta^*_k, \Lambda^*_k) & = X_1(T_k; \tau^{*, k}_0, \xi^{*, k}_0, \emptyset, \Lambda^*_k).
\end{align*}
For $n = 1, 2, \ldots$, the terms $\tau^{*, k+1}_n$,  $\Delta^{*, k+1}_n$, and $\xi^{*, k+1}_n$ are defined as in \eqref{eq:Lambda_k}, with $k$ replaced by $k+1$.

This recursive procedure continues until the final goal $K$. The supporting amount for the last goal is determined by the liquidation value:
\begin{equation}
	\theta^*_K = L(X(T_K; \tau^{*, K}_0, \xi^{*, K}_0, \emptyset, \Lambda^*_K)).
\end{equation}

To verify that the strategy constructed above is indeed optimal, two technical results are required: Lemma \ref{lem:borel} and Lemma \ref{lem:Tk_submart}. These results are instrumental in establishing Theorem \ref{thm:optimalstrategy}. The proof of Lemma \ref{lem:borel} follows similar arguments to those in \citet[Proposition 4.3.1]{pham2009book} and \citet[Proposition 5.10]{belak2022optimal}.
\begin{lemma}\label{lem:borel}
	Consider an array of functions given by
	\begin{equation}\label{eq:borel}
		(\{ h_1(t, x) \}_{t \in [0, T_1]}, \ldots, \{ h_k(t, x) \}_{t \in [T_{k-1}, T_{k}]}, \ldots, \{ h_{K}(t, x) \}_{t \in [T_{K-1}, T]} ),
	\end{equation} 
	where $h_k(t, x): [T_{k-1}, T_k] \times \barS \rightarrow \R$, $k=1, \ldots, K$ is Borel measurable and satisfies $h_k(t, x) \leq C_g$ for a generic constant $C_g$. If \eqref{eq:borel} satisfies Conditions (2), (3), and (4) in Definition \ref{def:sto_sub}, then \eqref{eq:borel} also satisfies the viscosity subsolution properties (1), (2), and (3) in Definition \ref{def:vis_sub}.
\end{lemma}

\begin{lemma}\label{lem:Tk_submart}
	For each $k= 1, \ldots, K$, consider any random initial condition $(\btau, \xi)$ with $\btau \in [T_{k-1}, T_k]$, $\xi \in \cF_{\btau}$, and $\p(\xi \in \barS) = 1$. Then for any stopping time $\rho \in [\btau, T_k]$, the value function $V_k$ satisfies
	\begin{equation}\label{eq2:Tk_submart}
		V_k(\btau, \xi) \leq \E \big[ V_k(\rho, X(\rho; \btau, \xi, \emptyset, \emptyset)) \big| \cF_{\btau} \big].
	\end{equation}
\end{lemma}

The second main result establishes the existence of an optimal strategy, as stated in Theorem \ref{thm:optimalstrategy} below.
\begin{theorem}\label{thm:optimalstrategy}
	Consider $k=1, \ldots, K$ and $(t, x) \in [T_{k-1}, T_k] \times \barS$. The strategy $(\theta^*_{k:K}, \Lambda^*)$ is admissible and optimal, that is, 
	\begin{equation}
		(\theta^*_{k:K}, \Lambda^*) \in \cA(t, x; k) \quad \text{ and } \quad V_k(t, x) = \E \Big[ \sum^{K}_{i = k} w_i (G_i - \theta^*_i)^+ \Big| X^*(t-) = x \Big].
	\end{equation}
	The corresponding wealth process is denoted by $X^*(s) := X(s; t, x, \theta^*_{k:K}, \Lambda^*)$, $s \in [t, T]$.
\end{theorem}

\section{Numerical analysis}\label{sec:numerics}

In this section, we present numerical results for the optimal investment strategies. For simplicity, consider an investor with two goals, $G_1 = 3$ and $G_2 = 6$, with respective deadlines $T_1 = 1$ and $T_2 = 2$. In the benchmark setting, the goal importance weights are $w_1 = 1$ and $w_2 = 0.2$, which are close to those in \cite{capponi2024} after appropriate adjustments.

For the financial market, unless stated otherwise, the parameters are set as follows: the interest rate $r = 0$, the expected stock return $\mu = 0.3$, and the volatility $\sigma = 0.4$. In this numerical study, we consider only fixed transaction costs, specified by $C(\Delta) \equiv C_{\min} > 0$. The benchmark case assumes $C_{\min} = 0.02$. The algorithm employs a classical finite difference method combined with a penalty scheme; further details can be found in \cite{azimzadeh2017}. Following the rationale of \cite{belak2022optimal}, the computations are conducted on a triangular grid rather than the square grid used in \cite{azimzadeh2017}. For positions satisfying $x_0 + x_1 \ge G_1 + G_2 + C_{\min}$, the value function equals zero. Therefore, the computational domain is restricted to the triangular region where $x_0 + x_1 \le 9 + C_{\min}$. The wealth grid size is set to $\Delta x = (9 + C_{\min})/200$, which equals 0.0451 in the benchmark case. The tick size in all heatmap figures such as Figure \ref{fig:transfer_c0_t0} is set to $10 \Delta x$, and the axis labels are rounded to two decimal places. For comparison, Figure \ref{fig:largerDx} is computed with a coarser grid size of $9.02/50$ and a tick size of $2 \times 9.02/50$. The time step is fixed at $\Delta t = 0.01$.

\subsection{The frictionless case}

Before presenting the fixed-cost case, we first reproduce the $V$-shaped investment behavior observed in \cite{capponi2024}. Figure \ref{fig:invest_frictionless} shows the optimal proportion invested in the stock, given by $x_1/(x_0 + x_1)$, over time. The results indicate that the optimal strategy reduces the stock proportion when total wealth approaches $G_1$, the level required to meet the first goal, and increases it once wealth exceeds $G_1$. This $V$-shaped adjustment reflects an investor's tendency to reduce risk near the target level to avoid missing the primary goal. Figure \ref{fig:fund_plot_frictionless} illustrates the optimal funding ratio $\theta^*_1/G_1$ for goal 1. Since this goal has a significantly higher weight, the investor allocates all available funds to it until the target amount is achieved.

\begin{figure}[h]
	\centering
	\includegraphics[width=0.9\textwidth]{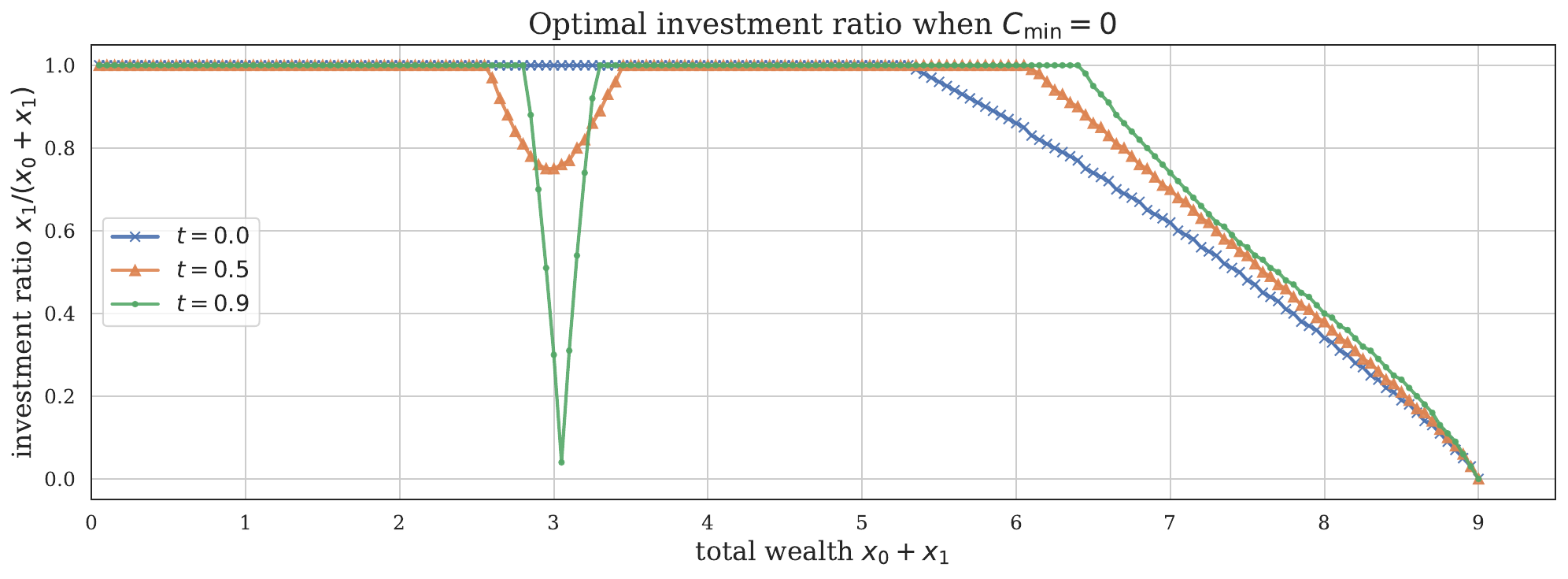}
	\caption{Optimal stock proportions without transaction costs.}
	\label{fig:invest_frictionless}
\end{figure}

\begin{figure}[h]
	\centering
	\includegraphics[width=0.9\textwidth]{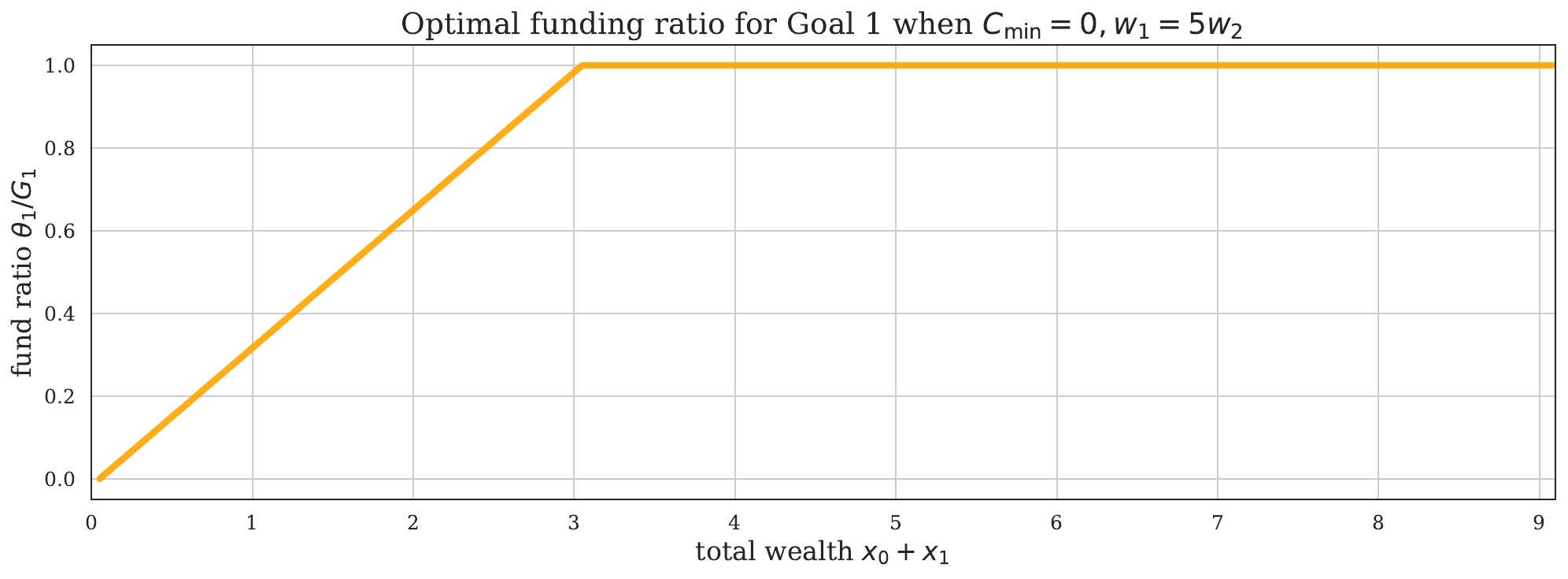}
	\caption{Optimal funding ratio without transaction costs.}
	\label{fig:fund_plot_frictionless}
\end{figure}

We now propose a conjecture regarding the optimal strategy under fixed costs. When the fixed cost is sufficiently small, the optimal stock exposure $x_1/(x_0 + x_1)$ should closely resemble that in the frictionless case for the same total wealth $x_0 + x_1$. The continuation region is expected to lie near the frictionless optimal investment proportion, within which the portfolio evolves without adjustment. Due to market fluctuations, the portfolio may occasionally reach the trading boundaries, prompting the investor to buy or sell the stock to reposition the portfolio onto the target set.

The analysis of the fixed-cost case proceeds in two steps. First, we consider a given level of total wealth. Second, by comparing with the frictionless optimal strategy, we identify trading regions that indicate whether to buy or sell when the current position deviates substantially from the target portfolio.

The following aspects are the main focus of our analysis:
\begin{enumerate}[label={(\arabic*)}]
	\item the effect of fixed costs on the portfolio's risk profile, particularly the relationship between stock investment and total wealth;
	\item the effect of fixed costs on the funding ratios required to meet investment goals.
\end{enumerate}
In addition, we discuss how these relationships evolve over time, as well as how they change when fixed costs increase or expected stock returns decrease.

\subsection{The benchmark case with $C_{\min} = 0.02$}

This subsection examines the properties of the optimal strategies when the fixed transaction cost is $C_{\min} = 0.02$.

\subsubsection{Time $t = 0.0$}
\begin{figure}[h]
	\centering
	\includegraphics[width=0.6\linewidth]{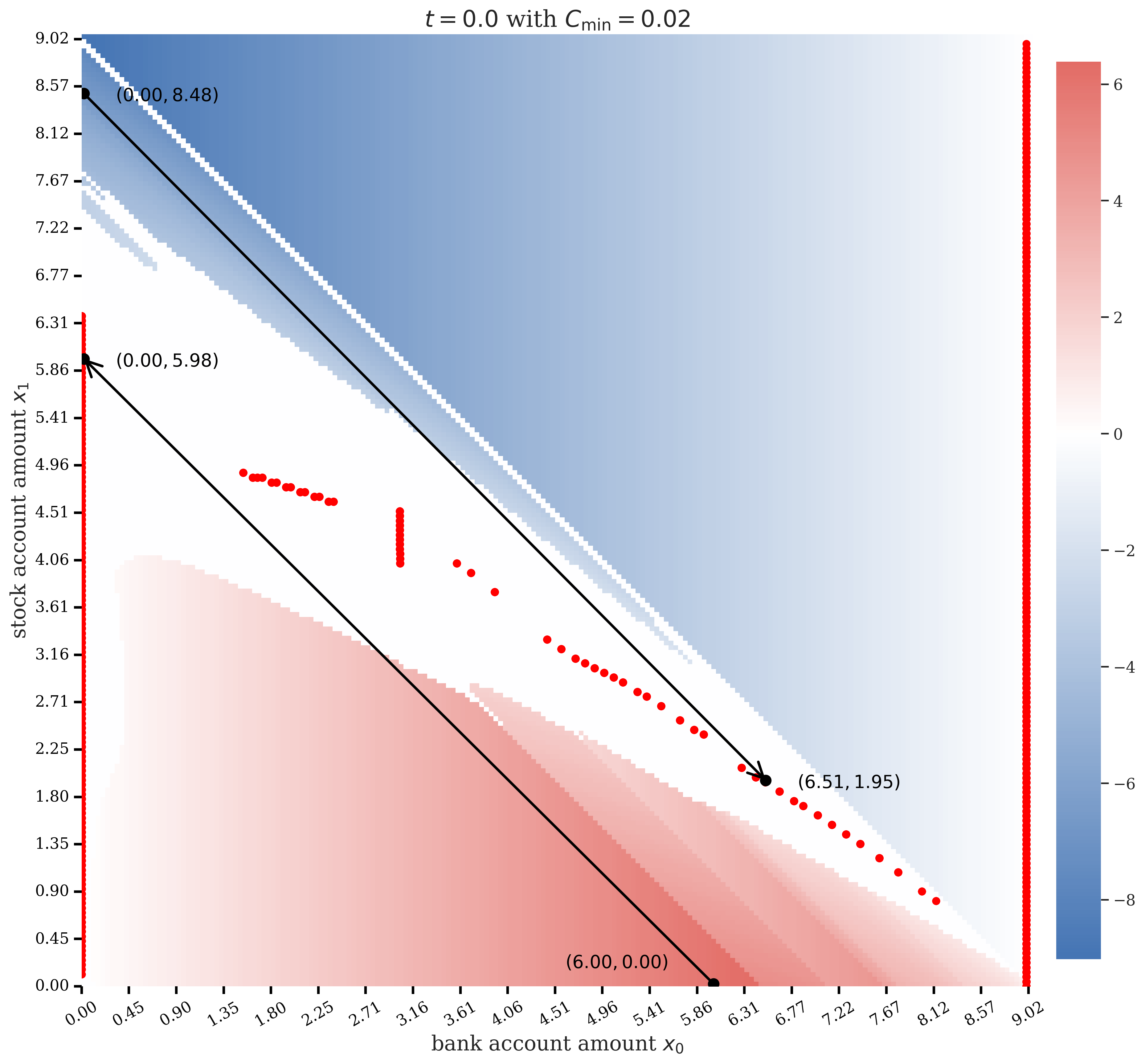}
	\caption{Optimal trading regions at $t=0.0$ with $C_{\min} = 0.02$.}
	\label{fig:transfer_c0_t0}
\end{figure}

\begin{figure}[h]
	\centering
	\includegraphics[width=0.8\linewidth]{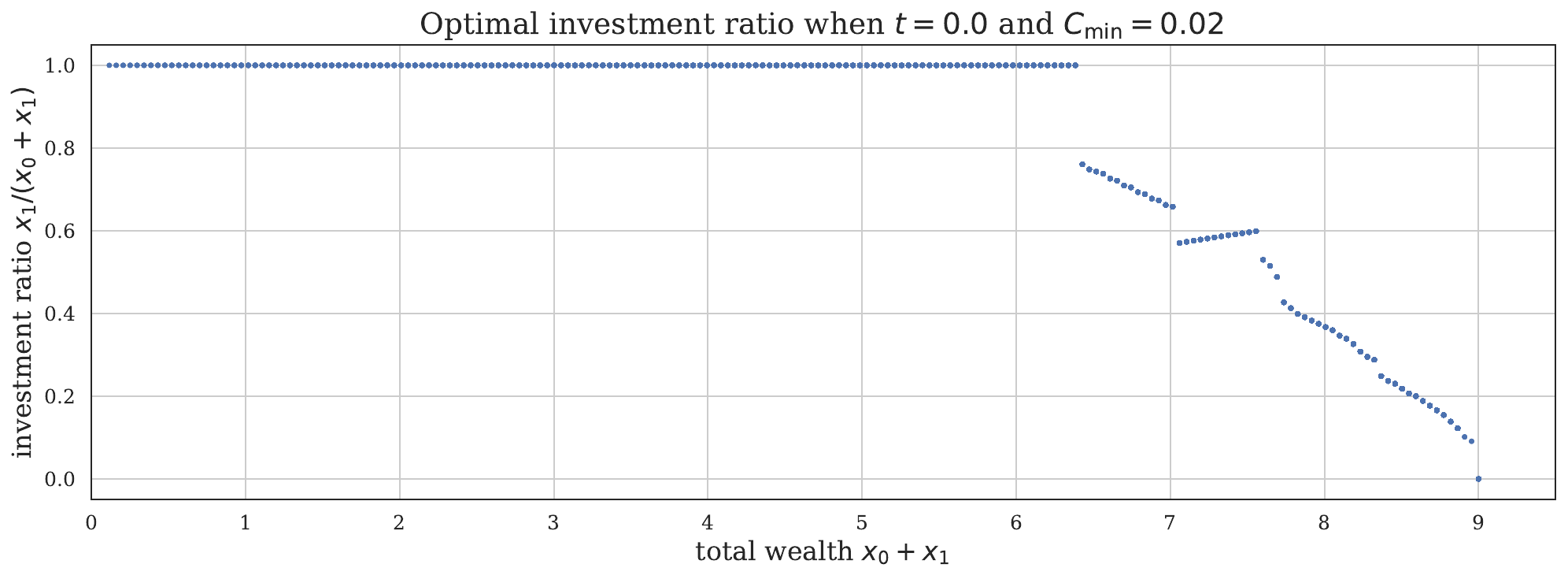}
	\caption{Stock proportions corresponding to the red target points in Figure \ref{fig:transfer_c0_t0}. }
	\label{fig:invest_t0_c0}
\end{figure}

In figures such as Figure \ref{fig:transfer_c0_t0}, which illustrate the optimal trading regions, the blue area corresponds to selling the stock, the red area corresponds to buying, and the red points mark the target portfolio positions. These red points may represent target positions from either side or from both sides. Since each trade reduces total wealth by $C_{\min}$, this property can be used to identify the correspondence between the red target points and the positions within the trading regions. A deeper color indicates a larger trade. 

The white area denotes the continuation region, where the portfolio evolves uncontrolled. The shape of this region differs significantly from that in Merton's problem with fixed transaction costs; see \citet[Figure 2]{belak2022optimal}. As shown in Figure \ref{fig:transfer_c0_t0}, the continuation region is not approximately $V$-shaped and is not symmetric with respect to the red target positions.

The behavior of the optimal strategies varies with total wealth, as summarized below:
\begin{enumerate}[label={(\arabic*)}]
	\item When $x_0 + x_1 \geq G_1 + G_2 + C_{\min}$, it is optimal to sell the stock so that the bank account holds the required amount $G_1 + G_2$ to meet both goals.
	
	\item When total wealth is slightly below $G_1 + G_2 + C_{\min}$, Figure \ref{fig:invest_frictionless} shows that the optimal stock ratio in the frictionless case remains low. If the current stock holding $x_1$ is high, one might expect selling to be optimal; however, after selling, the remaining stock position would be very small since the target level in the frictionless case is close to zero. In this case, even with a positive stock return $\mu$, the potential gain from such a small stock holding is unlikely to offset the fixed cost $C_{\min}$. Hence, the optimal strategy is not to trade when $x_1$ is high and the total wealth is just below $G_1 + G_2 + C_{\min}$, which corresponds to the white region in the upper-left corner of Figure~\ref{fig:transfer_c0_t0}. Conversely, when $x_1$ is low, buying the stock becomes optimal since otherwise it is difficult to exceed $G_1 + G_2 + C_{\min}$ with very small $x_1$. This behavior corresponds to the red area in the lower-right corner of Figure~\ref{fig:transfer_c0_t0}.
	
	\item When total wealth is lower than $G_1 + G_2 + C_{\min}$ but above $7.6$, Figure \ref{fig:invest_frictionless} shows that the optimal stock ratio in the frictionless case is higher than before. The investor can now offset the fixed cost through sufficient stock returns, leading to stock sales in the upper-left region of Figure~\ref{fig:transfer_c0_t0}, different from the previous case. This implies a non-monotonic relationship between risk exposure and wealth level in this range, resulting from the presence of fixed costs.

	\item  When total wealth is below $7.6$, a distinct pattern appears for $x_0 + x_1 \in [7.0, 7.6]$. A red vertical bar in the middle of Figure \ref{fig:transfer_c0_t0} indicates that the agent reserves roughly $3.0$ for goal 1. This corresponds to the increased stock proportion shown in Figure \ref{fig:invest_t0_c0} for the same wealth range, which is the only region where the stock proportion rises.
	
	In the frictionless case, Figure \ref{fig:invest_frictionless} shows that the optimal stock ratio is lower than $100\%$ at $t=0$ when total wealth exceeds $5.5$. Under fixed costs, however, if the bank balance is high, the optimal strategy is to allocate all wealth to the stock, for example when $x_0 + x_1 \in [5.0, 6.0]$. This suggests more aggressive behavior compared with the frictionless case, as potential transaction costs reduce overall profits. Moreover, for $x_1 \in [4.8, 6.6]$ and $x_0 \in [0.6, 1.4]$, the optimal decision is to refrain from trading.
\end{enumerate}

\subsubsection{Time $t=0.5$ and $t = 0.9$}

\begin{figure}[H]
	\centering
	\begin{minipage}{0.45\textwidth}
		\includegraphics[width=\linewidth]{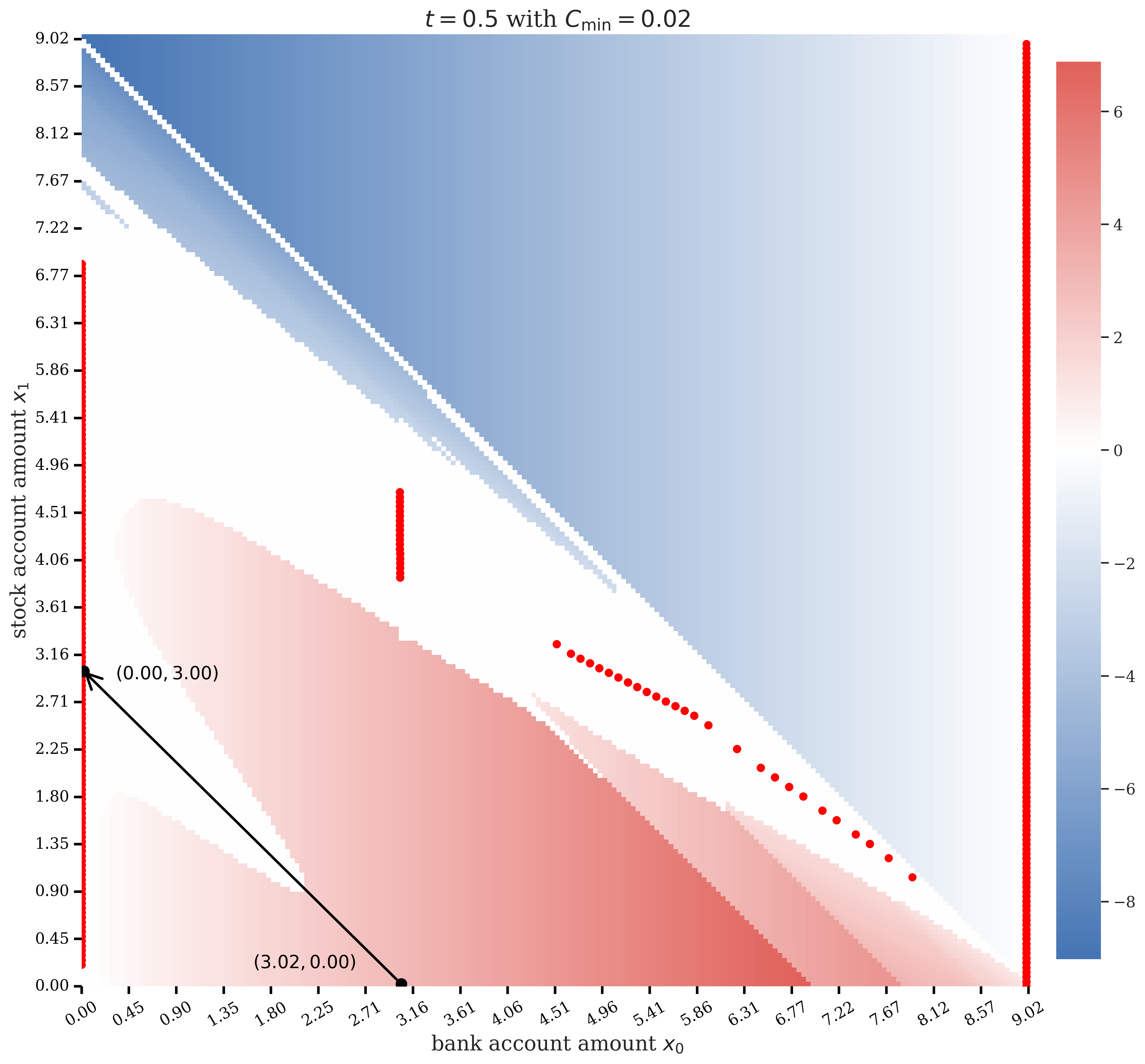}
		\subcaption{$t=0.5$}
		\label{fig:c002t05}
	\end{minipage}
	\begin{minipage}{0.45\textwidth}
		\includegraphics[width=\linewidth]{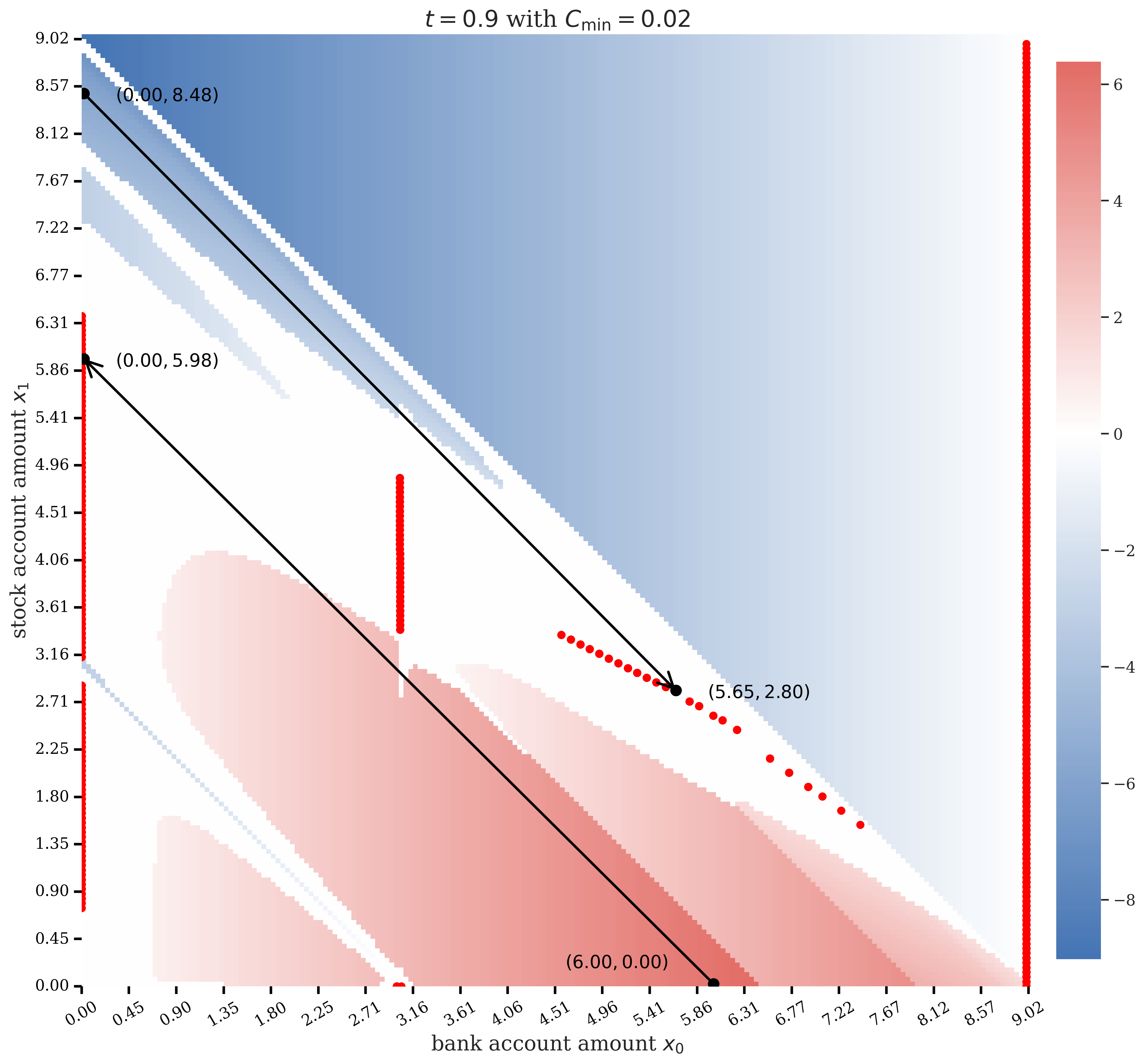}
		\subcaption{$t=0.9$}
		\label{fig:c002t09}
	\end{minipage}
	\caption{Optimal trading regions at $t=0.5$ and $t=0.9$ with $C_{\min} = 0.02$.}
	\label{fig:t0509}
\end{figure}

\begin{figure}[H]
	\centering
	\begin{minipage}{0.8\textwidth}
		\includegraphics[width=\linewidth]{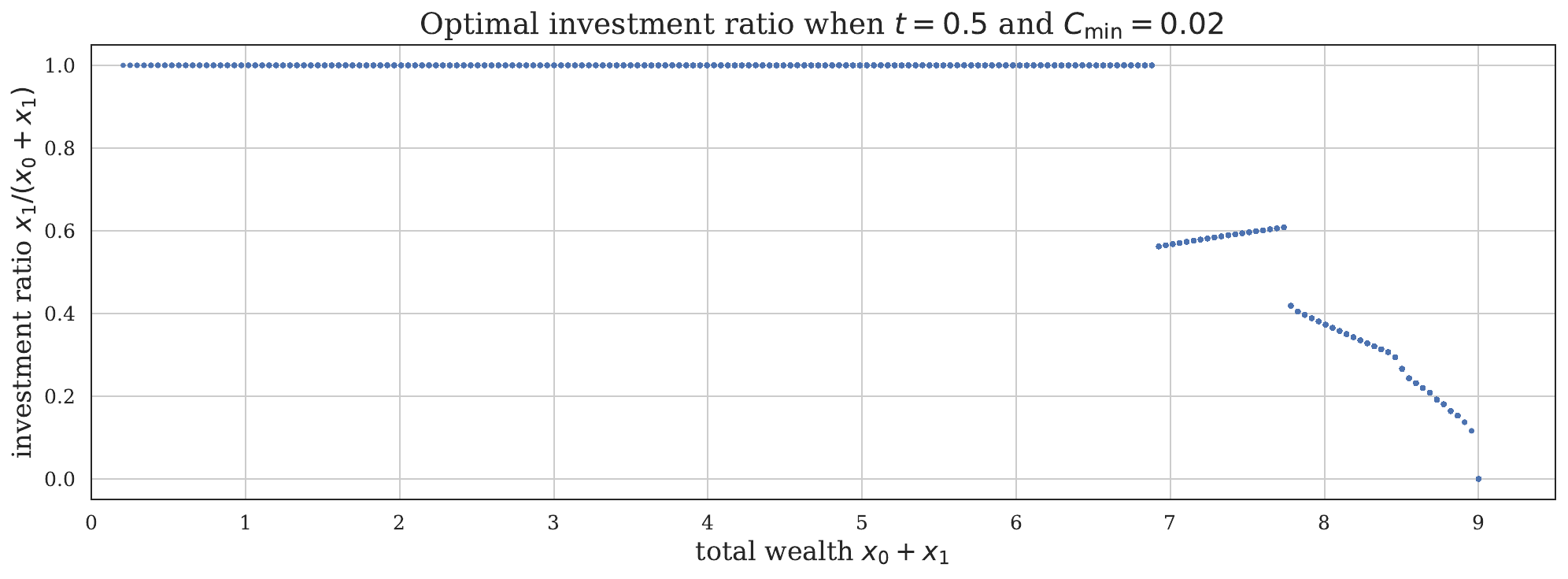}
		\subcaption{$t=0.5$}
		\label{fig:ratiot05c002}
	\end{minipage}
	\begin{minipage}{0.8\textwidth}
		\includegraphics[width=\linewidth]{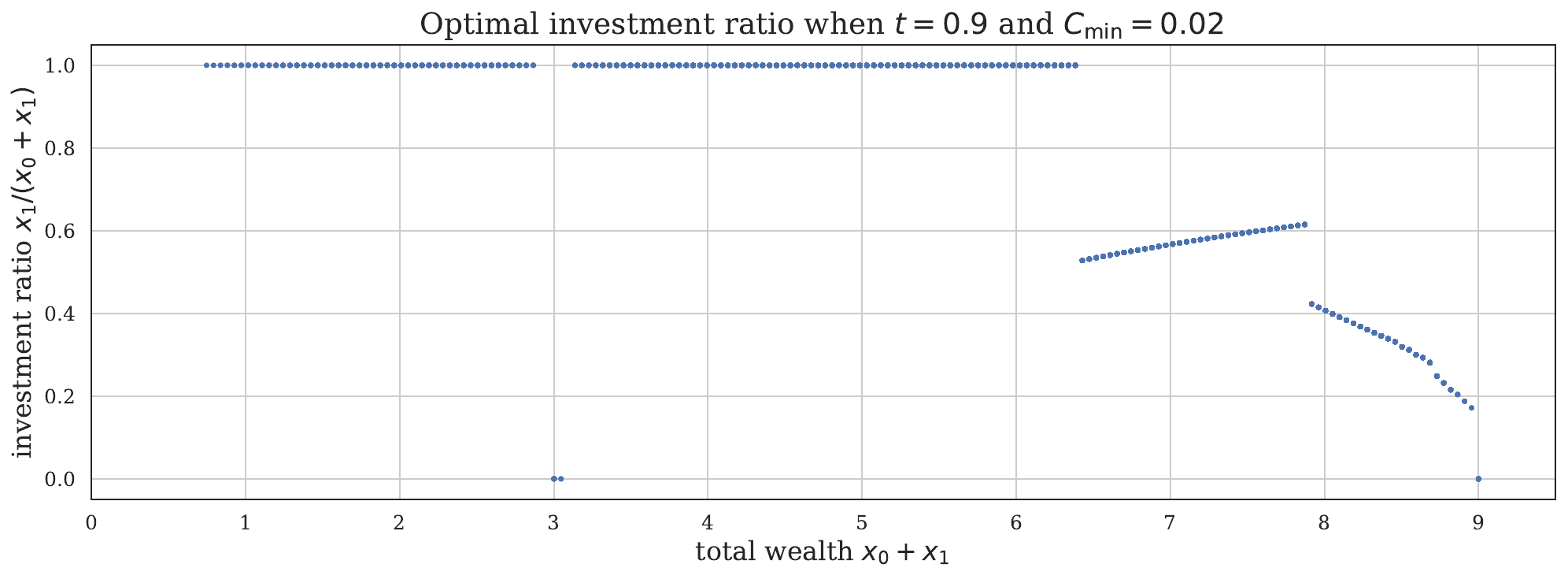}
		\subcaption{$t=0.9$}
	\end{minipage}
	\caption{Stock proportions for red target points in Figure \ref{fig:t0509}.}
	\label{fig:ratio0509}
\end{figure}

This subsection analyzes the optimal strategies at $t=0.5$ and $t=0.9$. The main observations are as follows:
\begin{enumerate}[label={(\arabic*)}]
	\item Comparing Figure \ref{fig:t0509} with Figure \ref{fig:transfer_c0_t0}, the red vertical bar near $x_0 = 3.0$ becomes longer as time approaches the deadline $T_1$. This indicates that the investor increasingly prioritizes reserving the required amount of $3.0$, investing only the excess wealth in the stock. Consequently, Figure \ref{fig:ratio0509} shows that the interval where the stock proportion increases with total wealth also widens. The red vertical bar in the middle originates from the right side when the amount $x_0$ in the bank account is high.
	
	In contrast, when $x_0 + x_1 \in [3.6, 6.6]$, which is below the total wealth corresponding to the central red bar, the optimal decision is to invest fully in the stock if $x_0$ is large. This provides another example where the optimal risk exposure is not monotonic in wealth levels.

	\item A distinct feature is the wedge-shaped white area in the lower-left corner of Figure \ref{fig:t0509} when $x_0 + x_1$ is near $3.0$. This reflects a behavior different from the $V$-shaped investment strategy described in \cite{capponi2024}.  
	
	At $t = 0.5$, when $x_0 = 3.0$ and $x_1 = 0.0$, the optimal choice is to invest entirely in the stock. As shown in Figure \ref{fig:ratiot05c002}, the target position allocates nearly $100\%$ to the stock around total wealth $3.0$, in contrast to the $V$-shaped strategy in the frictionless case. However, if the current position lies within the white wedge area, the optimal decision is to refrain from trading.  
	
	At $t = 0.9$, a similar continuation region appears near the equi-wealth line $x_0 + x_1 = 3.0$. Consequently, the investor must consider both the stock and bank account holdings, rather than total wealth alone, when determining the optimal stock exposure.
\end{enumerate}

\subsubsection{Time $t=1.0$: Funding ratios and importance weights}
\begin{figure}[h]
	\centering
	\begin{minipage}{0.45\textwidth}
		\includegraphics[width=\linewidth]{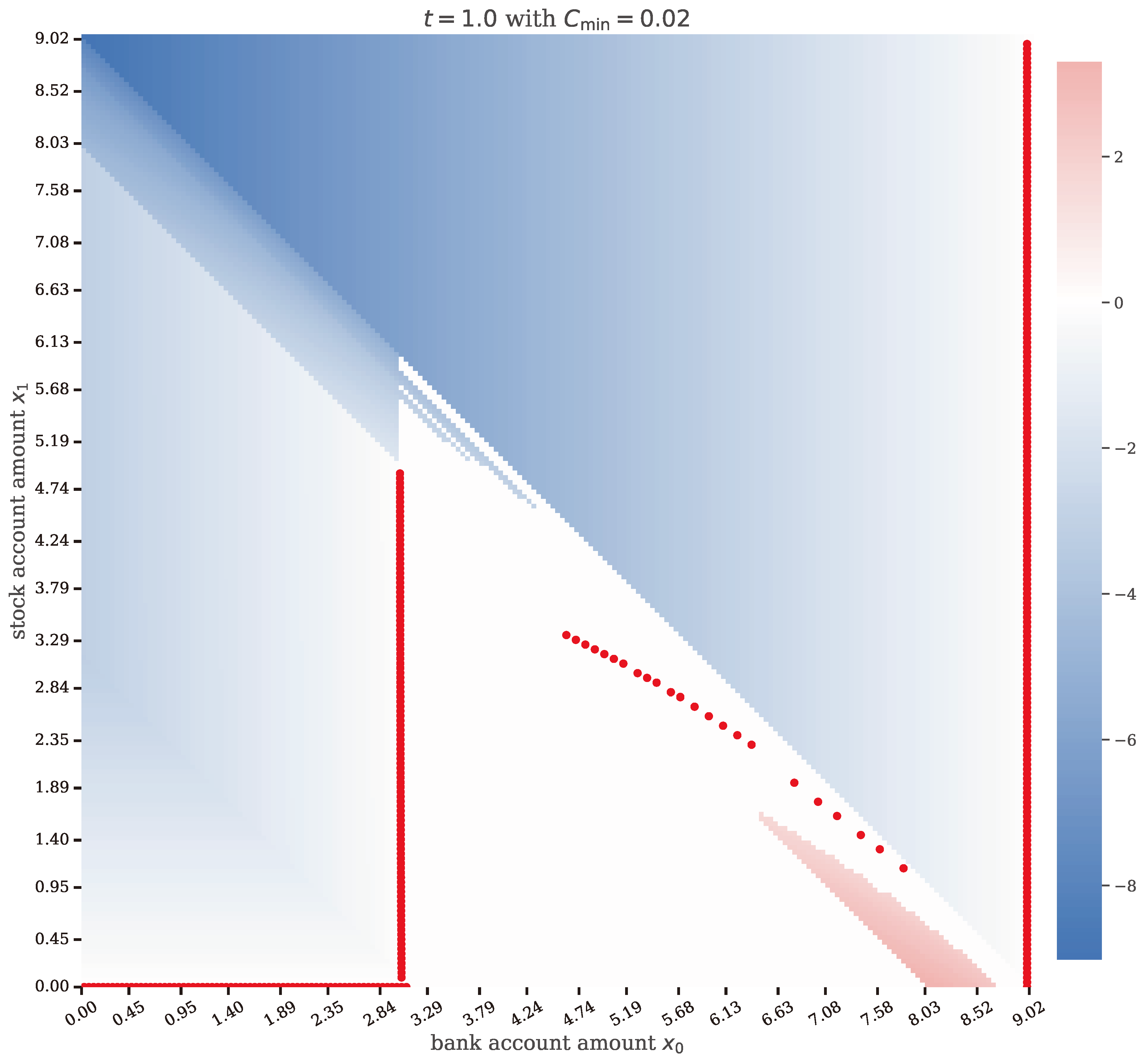}
		\subcaption{$w_1 = 5 w_2$}
		\label{fig:T1_5w2}
	\end{minipage} 
	\begin{minipage}{0.45\textwidth}
		\includegraphics[width=\linewidth]{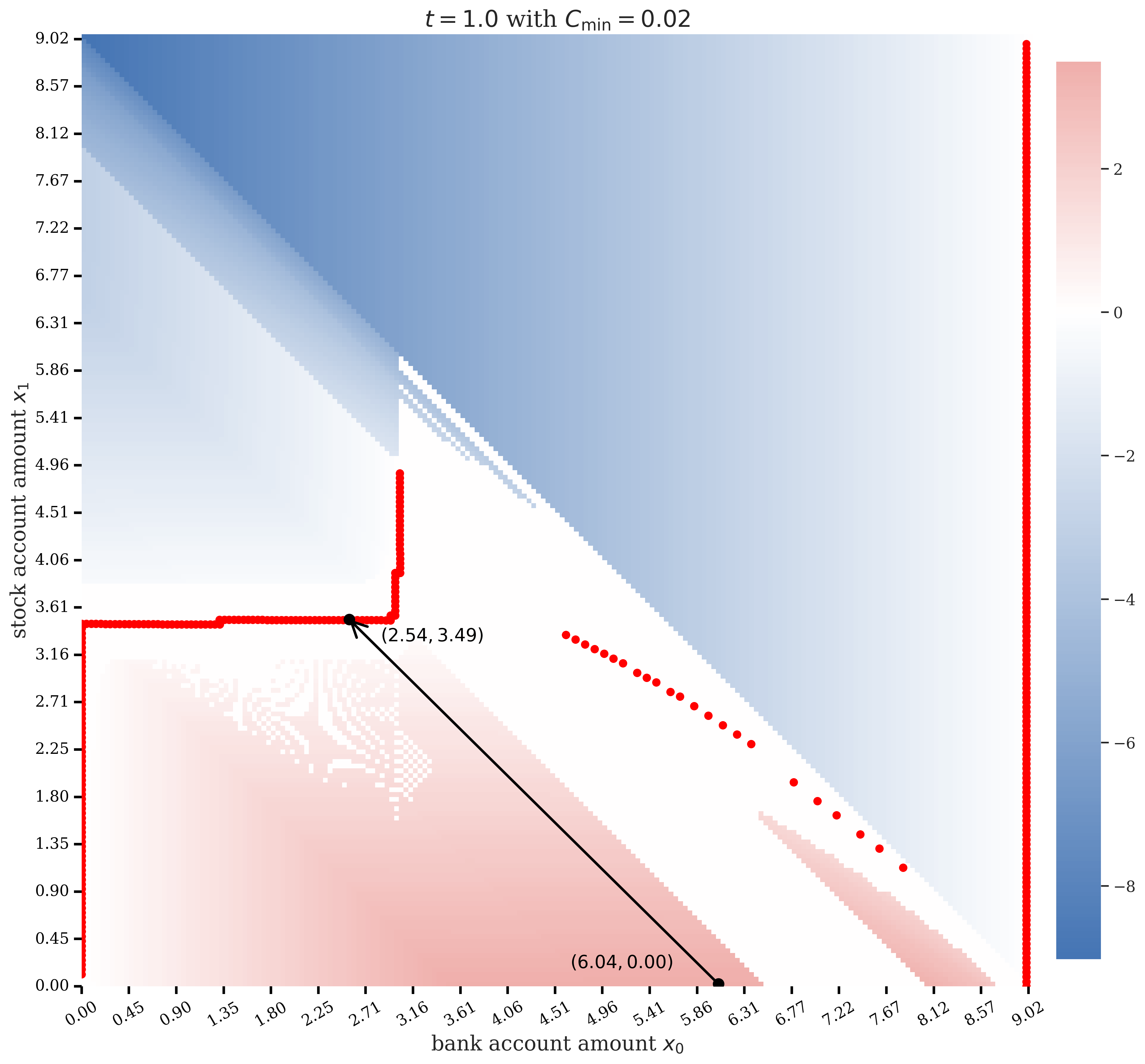}
		\subcaption{$w_1 = w_2$}
		\label{fig:T1_1w2}
	\end{minipage} 
	\caption{Optimal trading regions at the deadline $T_1$ under different goal weights.} 
	\label{fig:T1}
\end{figure}

At the deadline $T_1$, the following observations can be made:
\begin{enumerate}[label={(\arabic*)}]
	\item The weight configuration $w_1 = 5 w_2$ indicates that the first goal is substantially more important than the second. As shown in Figure~\ref{fig:T1_5w2}, the agent allocates all available funds to support the first goal, similar to the frictionless case in Figure \ref{fig:fund_plot_frictionless}. The importance of the first goal outweighs the potential additional returns from investing in the stock for another year. In this case, the fixed transaction cost has little influence on the optimal funding ratio.

	\item When the weights are equal, $w_1 = w_2$, the impact of fixed costs on the optimal funding ratio becomes more pronounced, especially when the total wealth ranges between $4.0$ and $6.0$. Figure \ref{fig:equal_w_0} illustrates that, in the absence of transaction costs, no funding is allocated to goal 1 when the total wealth is below $3.6$. For wealth between $3.6$ and $6.6$, approximately $3.6$ is retained in the stock, and the remainder is allocated to goal 1. The horizontal red bar in Figure \ref{fig:T1_1w2} is close to this critical threshold of $3.6$. The continuation region around this line highlights the influence of fixed costs. The optimal funding amount is determined by considering only the bank account, as no transfer occurs within the continuation region. Each point in Figure \ref{fig:equal_w_002} represents the corresponding funding ratio $\theta^*_1/ G_1$ for positions in the continuation region. The results indicate that the optimal funding ratios exhibit greater variability at a given level of total wealth when fixed costs are present. 
\end{enumerate}

\begin{figure}[H]
	\centering
	\begin{minipage}{0.9\textwidth}
		\includegraphics[width=\linewidth]{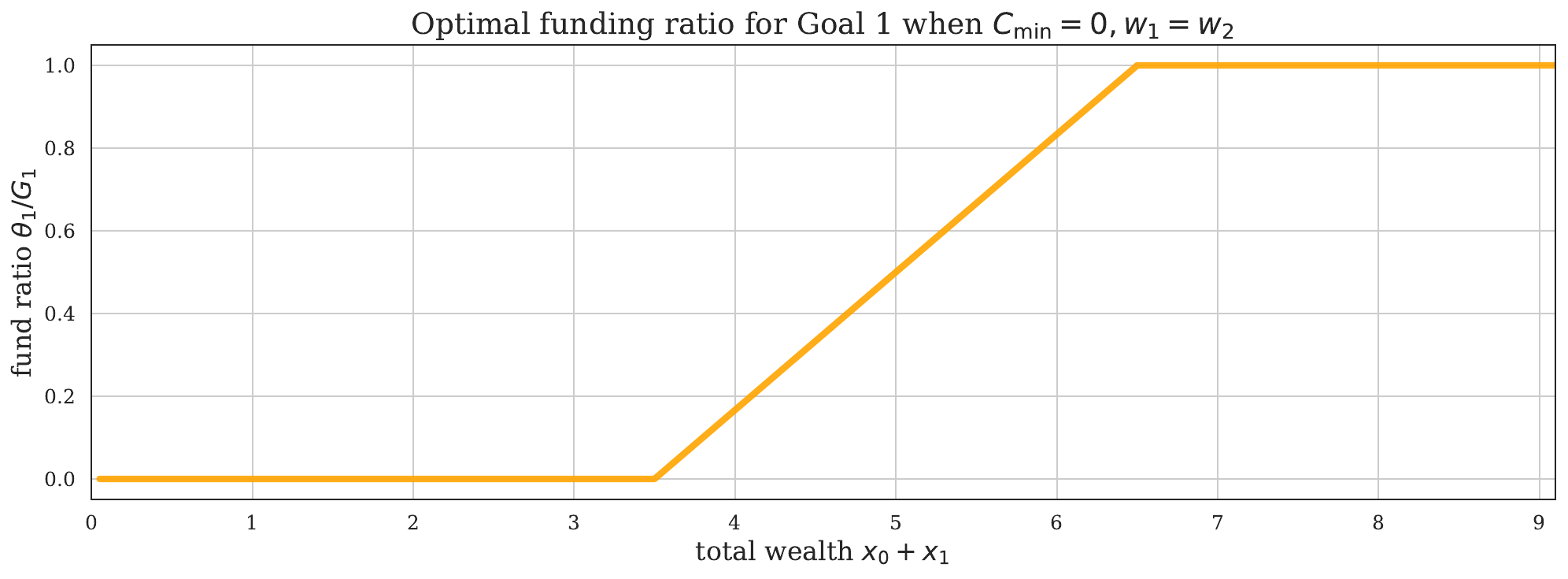}
		\subcaption{$C_{\min}=0$}
		\label{fig:equal_w_0}
	\end{minipage}
	\begin{minipage}{0.9\textwidth}
		\includegraphics[width=\linewidth]{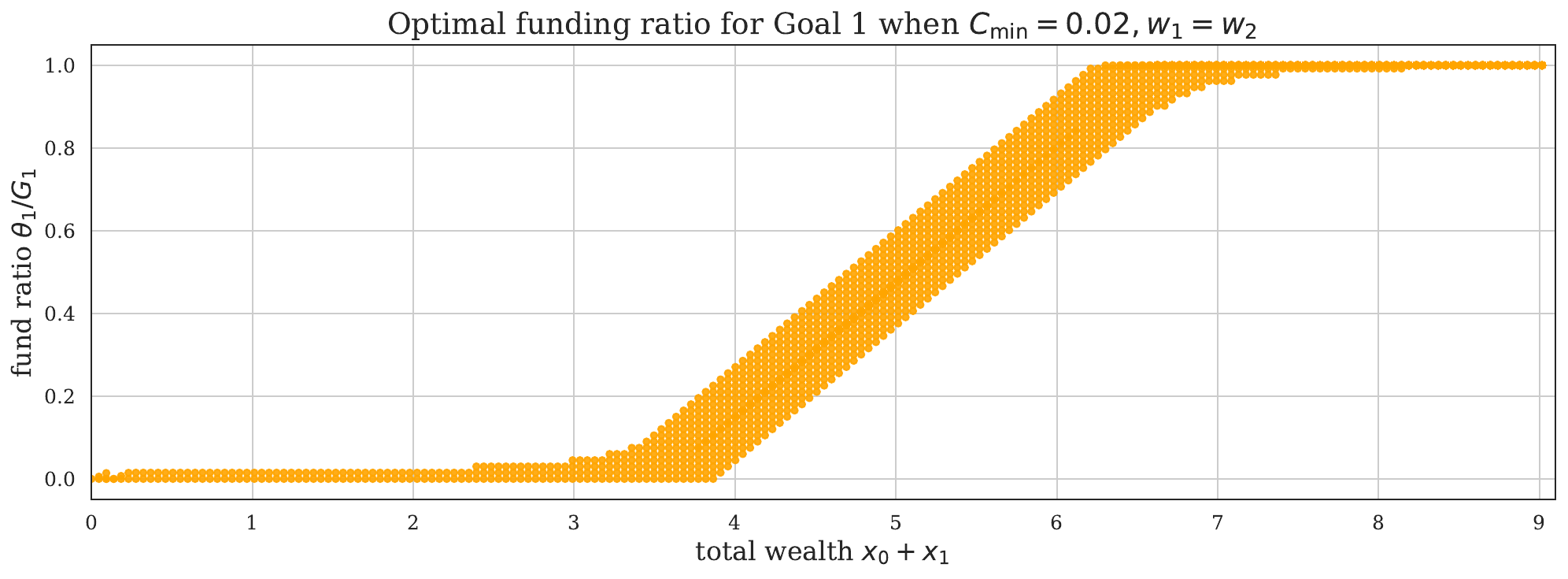}
		\subcaption{$C_{\min}=0.02$}
		\label{fig:equal_w_002}
	\end{minipage}
%	\begin{minipage}{0.9\textwidth}
%		\includegraphics[width=\linewidth]{fund_plot_c1_equal_w.pdf}
%		\subcaption{$C_{\min}=0.2$}
%	\end{minipage}
	\caption{Optimal funding ratios under equal weights.}
	\label{fig:fund_equal_w}
\end{figure}

\subsection{The straight continuation region near $G_1 + G_2 + C_{\min}$}\label{sec:straight}

The straight continuation region at the wealth level just below $G_1 + G_2 + C_{\min}$, illustrated as the narrow strip between the blue regions in the top-left panels of Figures \ref{fig:transfer_c0_t0} and \ref{fig:t0509}, is a distinctive feature that arises under the fixed-cost formulation. A closer examination of this pattern is provided below:
\begin{enumerate}[label={(\arabic*)}]
	\item As a consistency check, we verify that this phenomenon is not caused by discretization errors. Indeed, as shown in Figure \ref{fig:largerDx}, the pattern disappears when a coarser wealth grid is used. The explanation is straightforward: with a larger grid size, fixed costs become relatively less significant. Therefore, a finer grid is required to achieve higher numerical accuracy and to capture this behavior properly.

	\item When the stock return decreases, the straight continuation region becomes wider, as illustrated in Figure \ref{fig:lowerrtn}. This can be interpreted as follows. A lower expected return motivates the agent to hold a larger proportion of wealth in the stock to achieve the investment goals, reducing the likelihood of selling the asset. From another perspective, it also becomes more difficult to generate sufficient returns to offset the fixed transaction cost. Both effects contribute to a broader straight continuation region in the top-left area of Figure \ref{fig:lowerrtn}.
\end{enumerate}

\begin{figure}[H]
	\centering
	\begin{minipage}{0.45\textwidth}
		\includegraphics[width=\linewidth]{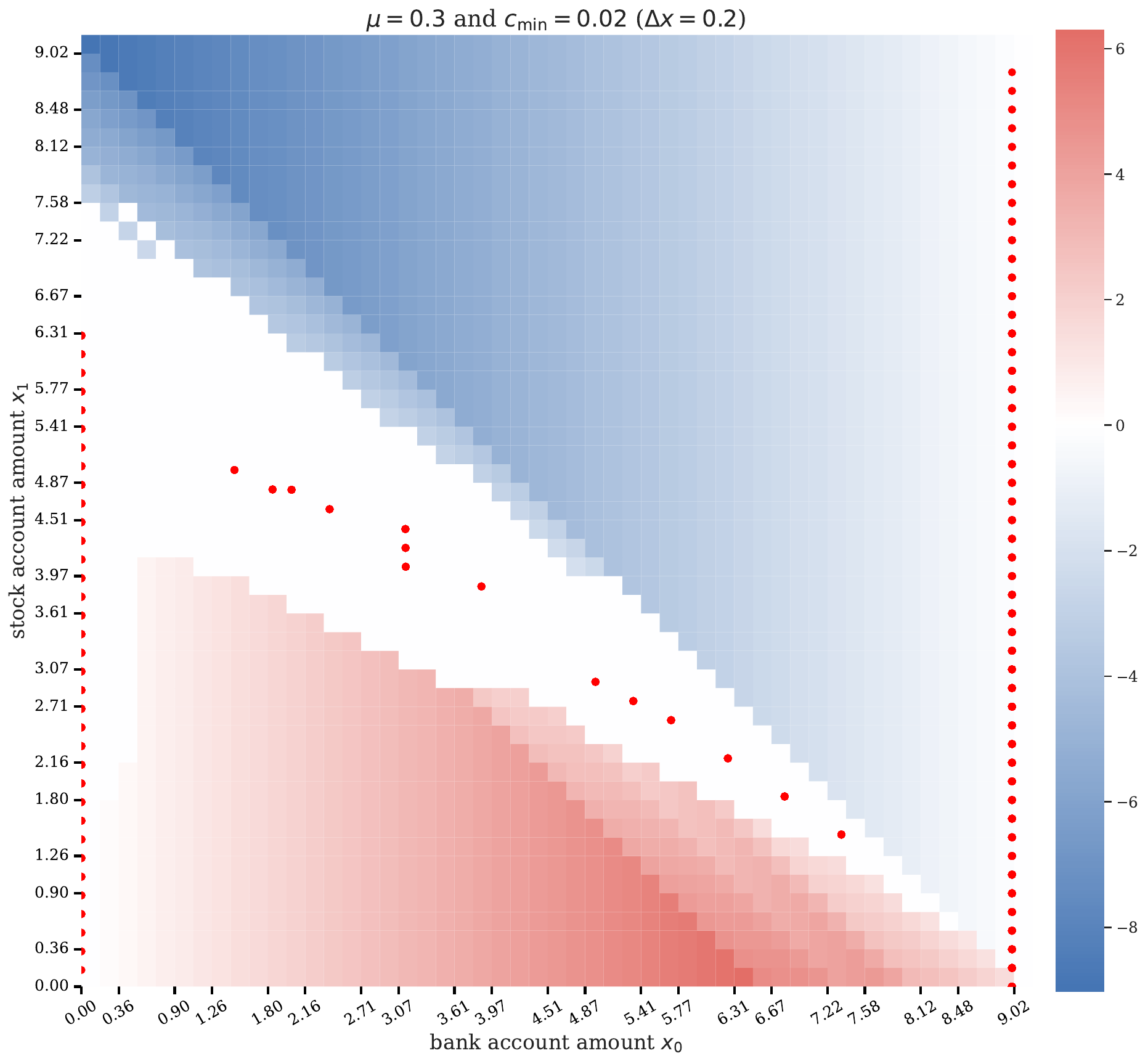}
		\subcaption{Larger wealth grid size ($\Delta x = 0.2$)}
		\label{fig:largerDx}
	\end{minipage}% 	
	\begin{minipage}{0.45\textwidth}
		\includegraphics[width=\linewidth]{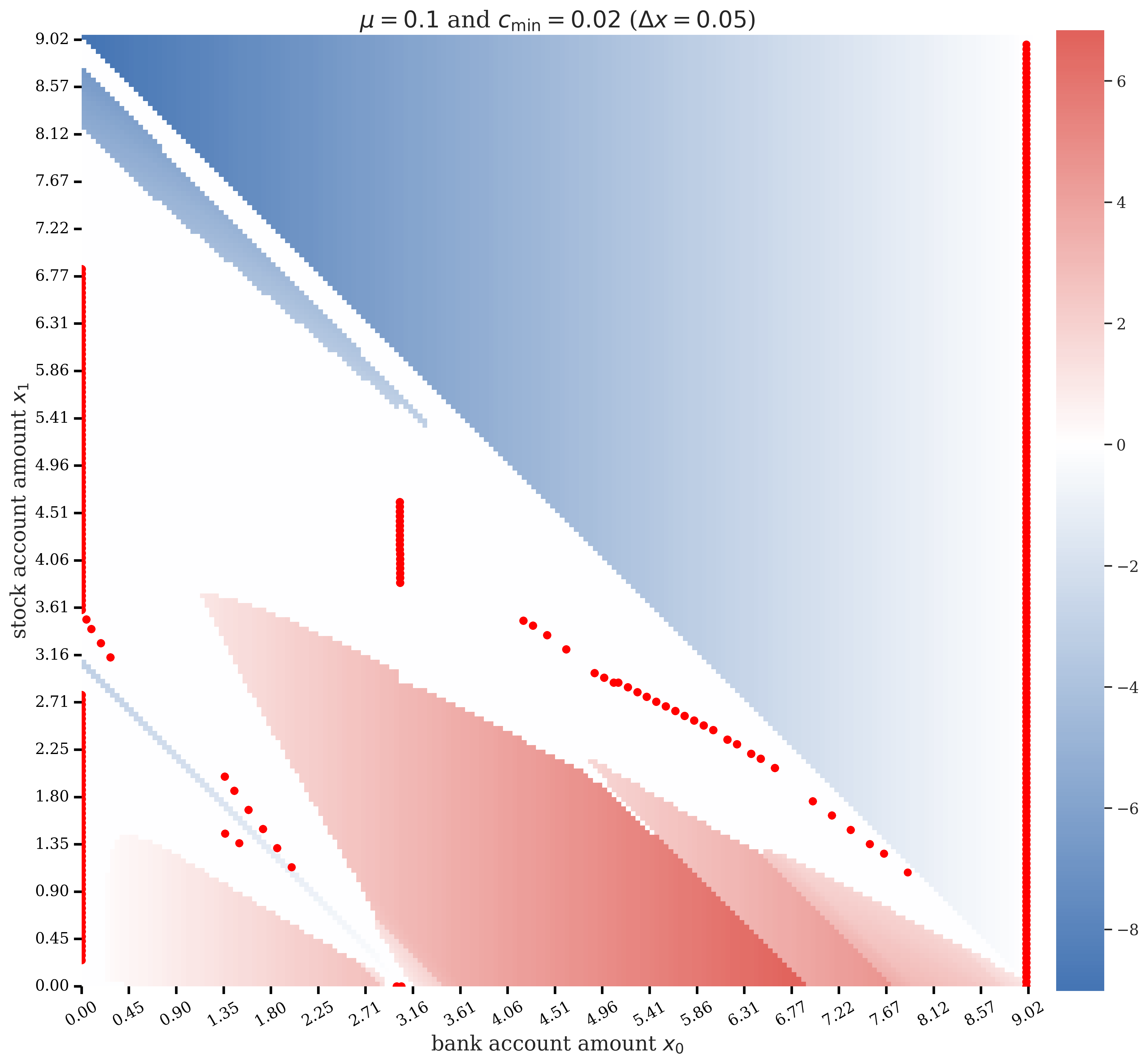}
		\subcaption{Lower stock return ($\mu = 0.1$)}
		\label{fig:lowerrtn}
	\end{minipage} 
	\caption{The straight continuation region near $G_1 + G_2 + C_{\min}$.}
	\label{fig:straight}
\end{figure}

\subsection{Higher fixed costs}

\begin{figure}[H]
	\centering
	\begin{minipage}{0.45\textwidth}
		\includegraphics[width=\linewidth]{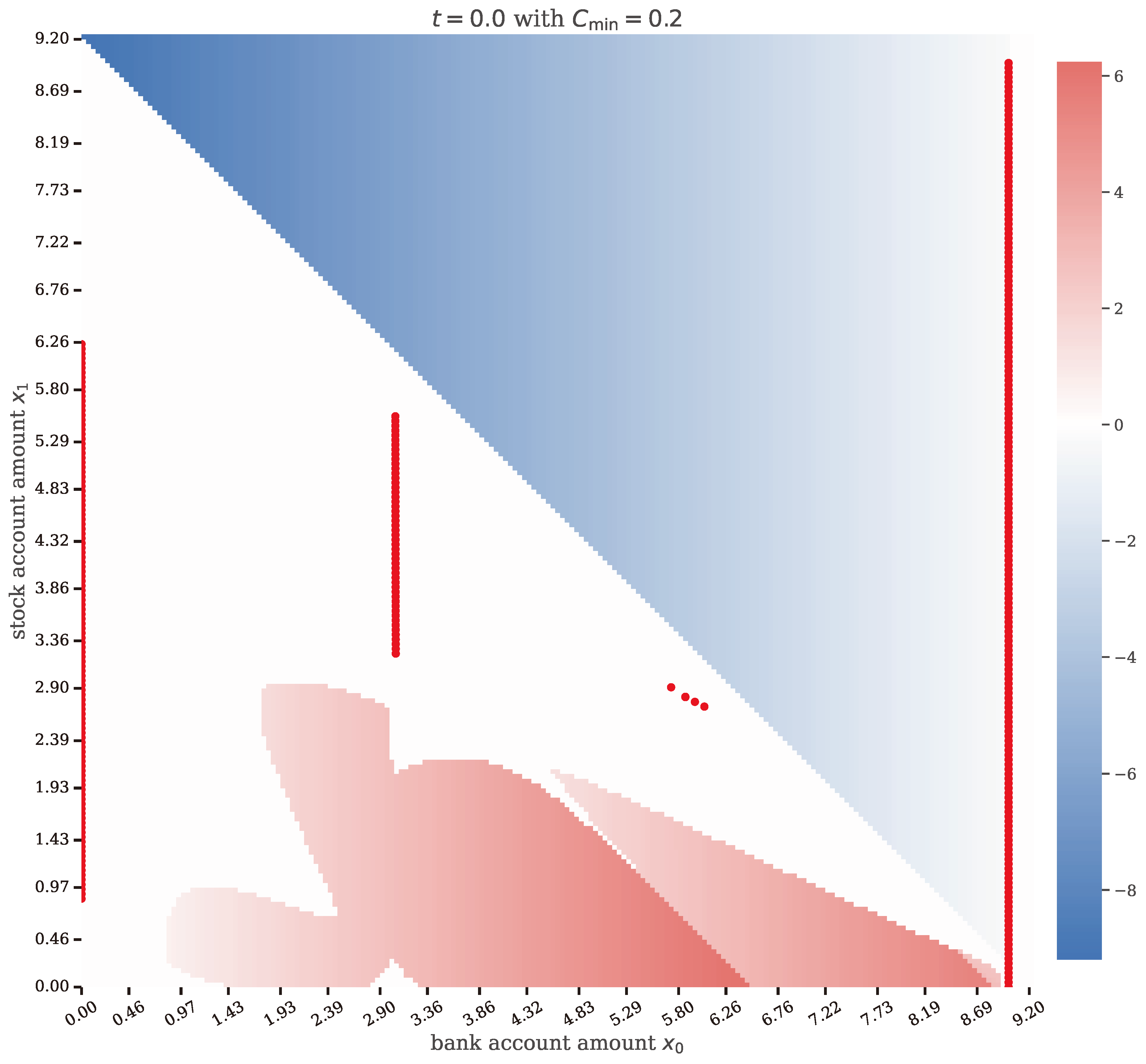}
		\subcaption{$t = 0.0$}
		\label{fig:t00c02}
	\end{minipage}
	%	\begin{minipage}{0.45\textwidth}
		%			\includegraphics[width=\linewidth]{transfer_c1_t1.pdf}
		%			\caption{}
		%	\end{minipage}
	\begin{minipage}{0.45\textwidth}
		\includegraphics[width=\linewidth]{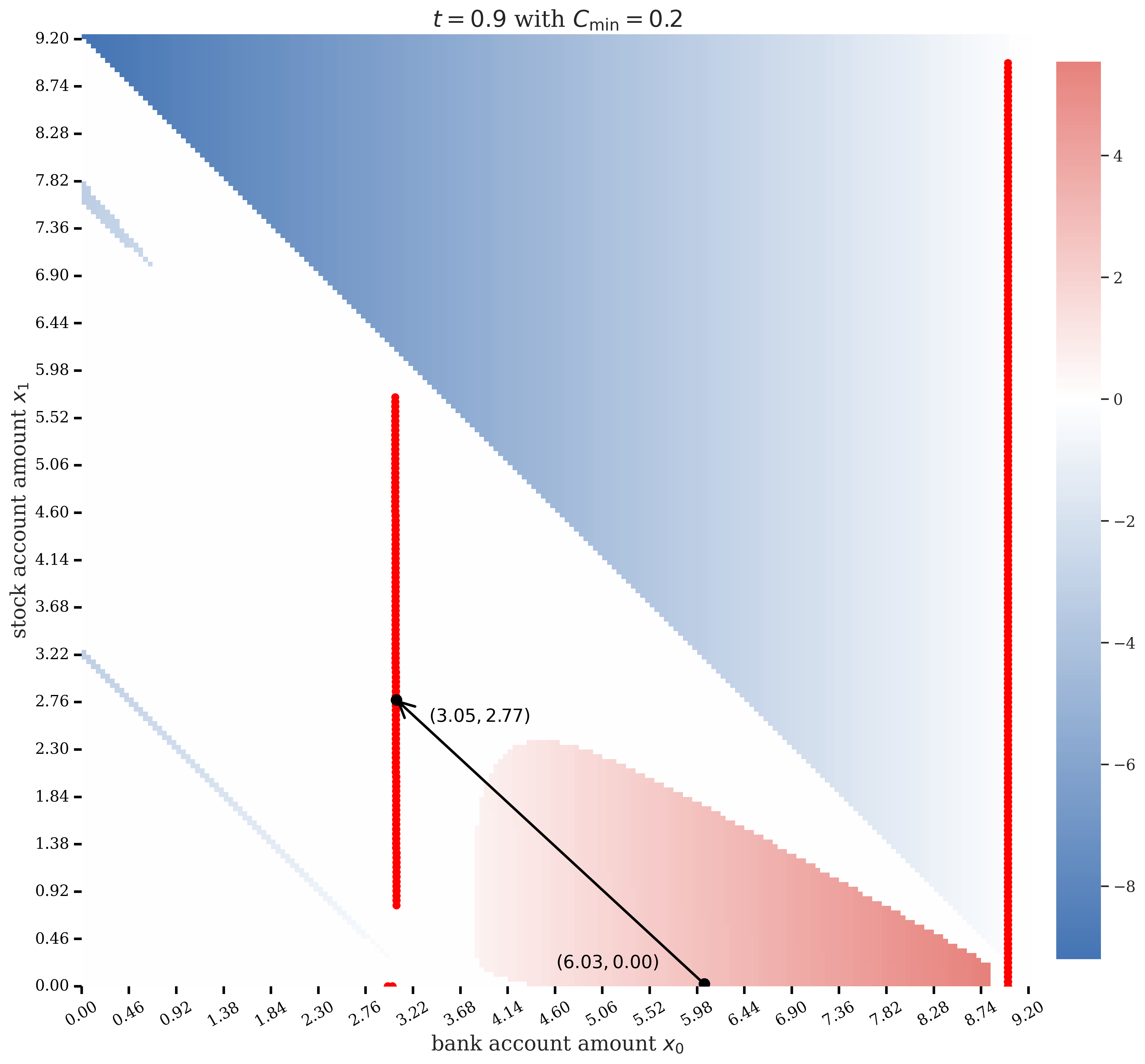}
		\subcaption{$t = 0.9$}
		\label{fig:t09c02}
	\end{minipage}
	\caption{Optimal trading regions at different times with $C_{\min} = 0.2$.}
	\label{fig:c02}
\end{figure}

When the fixed cost increases from $0.02$ to $0.2$, several phenomena can be observed in Figure \ref{fig:c02}:
\begin{enumerate}[label={(\arabic*)}]
	\item The continuation region becomes substantially wider. The higher fixed cost discourages trading activity, acting as a barrier to stock transactions. Consequently, the blue region in the upper left of Figure~\ref{fig:c02}, corresponding to wealth levels below $G_1 + G_2 + C_{\min}$, disappears, and the red region shrinks in size.
	
	\item The red vertical bar near $x_0 = 3.0$ becomes considerably longer, indicating that it is now more common to reserve the cash amount required for the first goal.
	
	\item A higher fixed cost may either reduce or increase exposure to the stock, depending on the specific situation:
	\begin{itemize}
		\item For $(x_0, x_1) = (6.0, 0.0)$ at $t = 0.9$, the target position is $(x_0, x_1) = (0.0, 5.98)$ when $C_{\min} = 0.02$, as shown in Figure \ref{fig:c002t09}. In contrast, when $C_{\min} = 0.2$, Figure \ref{fig:t09c02} shows that the target position from $(6.03, 0.0)$ is around $(3.05, 2.77)$, corresponding to a lower stock exposure.
		\item For $(x_0, x_1) = (0.0, 8.48)$ in the upper-left region at $t=0.9$, the higher cost case remains at the same position, while the lower cost case involves selling some stock. This illustrates a scenario where a higher fixed cost leads to higher stock exposure. 
	\end{itemize}
\end{enumerate}

\bibliographystyle{apalike}
\bibliography{ref.bib}

\appendix
\section{Proofs of the stochastic supersolution}
Lemma \ref{lem:Mproperty} gives some useful properties of the minimum operator $\cM$. The proof is similar to \citet[Lemma 5.1]{belak2022optimal} and thus omitted.
\begin{lemma}\label{lem:Mproperty}
	Let $k = 1, \ldots, K$ and $f: [T_{k-1}, T_k] \times \barS \rightarrow \R$. Then: 
	\begin{enumerate}[label={(\arabic*)}]
		\item If $f$ is LSC, then $\cM[f]_*(t, x) = \cM[f](t, x)$ for all $(t, x) \in [T_{k-1}, T_k] \times \barS$.
		\item If $f$ is USC, then $\cM[f]^*(t, x) = \cM[f](t, x)$ for all $(t, x) \in [T_{k-1}, T_k] \times (\barS \setminus \overline{\cS_\emptyset})$.
	\end{enumerate}
\end{lemma}

\begin{lemma}\label{lem:K_vissub}
	The upper stochastic envelope $v_{+}$ satisfies the viscosity subsolution property \eqref{eq:TK_vissub} at $T_K$, under Definition \ref{def:vis_sub}.
\end{lemma}
\begin{proof}
	Since $v_{K, +}$ is USC, it follows that $v^*_{K, +} = v_{K, +}$. For any $ \bar{x} \in \cS$, we aim to prove that
	\begin{equation*}
		\begin{aligned}
			\max \Big\{ & v_{K, +}(T_K, \bar{x}) - w_K \left[ G_K - \bar{x}_0 - (\bar{x}_1 - C(-\bar{x}_1))^+ \right]^+ , \\
			& v_{K, +}(T_K, \bar{x}) - \cM[v_{K, +}]^*(T_K, \bar{x}) \Big\} \leq 0.
		\end{aligned}
	\end{equation*}
	Assume on the contrary that the left-hand side is strictly positive. There are two possible cases.
	
	{\bf Case 1}. $v_{K, +}(T_K, \bar{x}) - w_K \left[ G_K - \bar{x}_0 - (\bar{x}_1 - C(-\bar{x}_1))^+ \right]^+ > 0$.
	
	Since the terminal value is continuous in $x$, there exists a small $\varepsilon > 0$ such that
	\begin{equation}\label{eq:bnd_TK}
		v_{K, +}(T_K, \bar{x}) - w_K \left[ G_K - x_0 - (x_1 - C(-x_1))^+ \right]^+ \geq \varepsilon,
	\end{equation}
	for $x \in \overline{B(\bar{x}, \varepsilon)}$, which is the closure of $B(\bar{x}, \varepsilon) :=\{ x : |x - \bar{x} | < \varepsilon \}$.
	
	For later use, define the sets
	\begin{align*}
		\cD(T_K, \bar{x}, \varepsilon) & := (T_K - \varepsilon, T_K] \times B(\bar{x}, \varepsilon), \\
		E(\varepsilon) & := \overline{\cD(T_K, \bar{x}, \varepsilon)} \backslash \cD(T_K, \bar{x}, \varepsilon/2), 
	\end{align*}
	where $\overline{\cD}$ denotes the closure of $\cD$.
	
	Note that $v_{K, +}$ is USC and $E(\varepsilon)$ is compact. Then $v_{K, +}$ is bounded from above on $E(\varepsilon)$. For a small enough $\eta > 0$, we have
	\begin{equation*}
		\sup_{(t, x) \in E(\varepsilon)} v_{K, +}(t, x) - v_{K, +} (T_K, \bar{x}) < \frac{\varepsilon^2}{4 \eta} - \varepsilon. 
	\end{equation*}
	As this inequality is strict, \citet[Proposition 4.1]{bayraktar2012linear} and \citet[Lemma 2.4]{bayraktar2014Dynkin} ensure that there exists $v^{n}_K$, which corresponds to a stochastic supersolution $v^{n} = (v^{n}_1, \ldots, v^{n}_K)$ and
	\begin{equation}\label{n2}
		\sup_{(t, x) \in E(\varepsilon)} v^n_{K}(t, x) - v_{K, +} (T_K, \bar{x}) < \frac{\varepsilon^2}{4 \eta} - \varepsilon. 
	\end{equation}
	
	For $p > 0$, define
	\begin{equation*}
		\psi^{\varepsilon, \eta, p}(t, x) := v_{K,+} (T_K, \bar{x}) + \frac{|x -  \bar{x}|^2}{\eta} + p(T_K - t).
	\end{equation*}
	With a large enough $p$, we can ensure that
	\begin{equation*}
		\cL[\psi^{\varepsilon, \eta, p}](t, x) > 0 \quad \text{ on } \overline{\cD(T_K, \bar{x}, \varepsilon)}.
	\end{equation*}	
	Thanks to the definition of $E(\varepsilon)$, the inequality \eqref{n2}, and a large enough $p$, we have
	\begin{align}
		\psi^{\varepsilon, \eta, p}(t, x) & \geq v_{K, +} (T_K, \bar{x}) + \frac{\varepsilon^2}{4 \eta} > \varepsilon + \sup_{(t, x) \in E(\varepsilon)} v^{n}_{K}(t, x) \nonumber \\
		& \geq \varepsilon + v^n_K (t, x) \quad \text{ on } E(\varepsilon). \label{Eineq}
	\end{align}
	
	Besides, for any $t \leq T_K$ and $x \in \overline{B(\bar{x}, \varepsilon)}$, \eqref{eq:bnd_TK} leads to
	\begin{align}
		\psi^{\varepsilon, \eta, p}(t, x) & \geq v_{K, +} (T_K, \bar{x}) \nonumber \\
		& \geq w_K \left[ G_K - x_0 - (x_1 - C(-x_1))^+ \right]^+ + \varepsilon. \label{psi_bd}
	\end{align}
	
	Let $0 <\delta < \varepsilon$ and set $\psi^{p, \delta} := \psi^{\varepsilon, \eta, p} - \delta$. Define
	\begin{equation}
		v^{p, \delta}_{K} (t, x) := \left\{ 
		\begin{array}{ c l }
			v^n_K (t, x) \wedge \psi^{p, \delta}(t, x) & \text{on } \overline{\cD(T_K, \bar{x},\varepsilon)}, \\
			v^n_K (t, x), & \text{otherwise}.
		\end{array}
		\right.
	\end{equation}
	
	Next, we show that $(v^n_1, \ldots, v^n_{K-1}, v^{p, \delta}_{K})$ is a stochastic supersolution, which leads to the following contradiction:
	$$v^{p, \delta}_{K} (T_K, \bar{x}) = v_{K, +} (T_K, \bar{x})  - \delta < v_{K, +} (T_K, \bar{x}).$$
	
	Clearly, $(v^n_1, \ldots, v^n_{K-1}, v^{p, \delta}_{K})$ satisfies Conditions (1) and (2) in Definition \ref{def:sto_super}. For the supermartingale property in Definition \ref{def:sto_super} (3), we first verify it when the random initial condition $(\btau, \xi)$ satisfies $\btau \in [T_{K-1}, T_K]$.
	
	Define the event
	\begin{equation*}
		A := \{ (\btau, \xi) \in \cD(T_K, \bar{x},\varepsilon/2) \} \cap \{\psi^{p, \delta}(\btau, \xi) < v^n_K(\btau, \xi)\}.
	\end{equation*}
	Then $A \in \cF_{\btau}$. 
	
	Let $U^0 := (\theta^0_{K}, \Lambda^0) := (L(X(T_K; \btau, \xi, \emptyset, \Lambda^0)), \{ (\tau^0_n, \Delta^0_n) \}^\infty_{n=1})$ be a suitable control for $v^n_{K}$ with the random initial condition $(\btau, \xi)$. Here, we recall that $\{ X(t; \btau, \xi, \emptyset, \Lambda^0) \}_{t \in [\btau, T]}$ denotes the solution where $\Lambda^0$ is used while $\theta_K$ is not determined.
	
	Define a new control $U^1 := (\theta^1_{K}, \Lambda^1)$ by
	\begin{equation}\label{U1}
		\theta^1_K  :=  \one_A \emptyset + \one_{A^c} \theta^0_K, \quad \Lambda^1  = \{ (\tau^1_n, \Delta^1_n) \}^\infty_{n=1} := \one_{A^c} \{ (\tau^0_n, \Delta^0_n) \}^\infty_{n=1}.
	\end{equation}
	Here, if $A$ happens, we do not conduct any transactions between the stock and the bank account. The funding amount $\theta_K$ is also to be determined. Instead, if $A^c$ happens, then $v^{p, \delta}_{K} (\btau, \xi) = v^n_{K} (\btau, \xi)$. Hence, $U^1$ follows a suitable control for $v^n_K$. Denote $\{ X(t; \btau, \xi, U^1) \}_{t \in [\btau, T]}$ as the solution of the state process with the random initial condition $(\btau, \xi)$ under the control $U^1$. Then $$\p(X(t; \btau, \xi, U^1) \in \overline{\cS}, \; \btau \leq t \leq T ) = 1.$$
	
	Define the exit time and position as
	\begin{align*}
		\tau' & := \inf \{ t \in [\btau, T_K] \, | \,  (t,  X(t; \btau, \xi, U^1)) \notin  \cD(T_K, \bar{x}, \varepsilon/2) \} \wedge T_K, \\
		\xi' & :=  X(\tau'; \btau, \xi, U^1) \in \cF_{\tau'}.
	\end{align*}

	There is a suitable control $U^2 := (\theta^2_{K}, \Lambda^2) := (L(X(T_K; \tau', \xi', \emptyset, \Lambda^2)), \{ (\tau^2_n, \Delta^2_n) \}^\infty_{n=1})$ for $v^n_{K}$ with the random initial condition $(\tau', \xi')$. This control $U^2$ will only be used when $\tau' < T_K$ happens. Finally, define a control $U := (\theta_{K}, \Lambda)$ by
	\begin{align*}
		\Lambda  & := \{ (\tau^1_n, \Delta^1_n) \one_{ \{ \tau^1_n \leq \tau' \}} \}^\infty_{n=1}  + \{ (\tau^2_n, \Delta^2_n) \one_{\{ \tau' \leq \tau^2_n \} \cap  \{ \tau' < T_K \}} \}^\infty_{n=1} ,\\
		\theta_K  & := L(X(T_K; \btau, \xi, \emptyset, \Lambda)).
	\end{align*}
	The control $U$ satisfies $$\p(X(t; \btau, \xi, U) \in \overline{\cS}, \; \btau \leq t \leq T ) = 1.$$
	We verify that $U$ is suitable for $v^{p, \delta}_{K}$ with $(\btau, \xi)$.
	
	Consider a stopping time $\rho \in [\btau, T_K]$. Applying the It\^o's formula to $\psi^{p, \delta}$ from $\tau$ to $\rho \wedge \tau'$ under the event $A$ and control $U^1$, we obtain
	\begin{align}
		& \one_A v^{p, \delta}_{K} (\btau, \xi) \nonumber \\
		& = \one_A \psi^{p, \delta} (\btau, \xi) \nonumber \\
		& = \one_A \psi^{p, \delta} (\btau, X(\btau; \btau, \xi, U^1)) \nonumber \\
		& \geq \E\Big[ \one_{A \cap \{ \rho < \tau'\}}  \psi^{p, \delta} (\rho, X(\rho; \btau, \xi, U^1)) + \one_{A \cap \{ \rho \geq \tau'\}}   \psi^{p, \delta} (\tau', \xi') \Big| \cF_{\btau} \Big]. \label{psiAineq1}
	\end{align} 
	Moreover, \eqref{Eineq} and \eqref{psi_bd} lead to
	\begin{align}
		\one_{A \cap \{ \rho \geq \tau'\}} \psi^{p, \delta} (\tau', \xi') \geq & \one_{A \cap \{ \rho \geq \tau'\} \cap \{ \tau' < T_K\}} v^n_K(\tau', \xi') \label{psiAineq2} \\
		& + \one_{A \cap \{ \rho \geq \tau'\} \cap \{ \tau' = T_K\}} w_K(G_K - \xi'_0 - (\xi'_1 - C(-\xi'_1))^+)^+. \nonumber
	\end{align}
	Combining \eqref{psiAineq1} and \eqref{psiAineq2}, since $v^{p, \delta}_{K} \leq \psi^{p, \delta}$ on $\overline{\cD(T_K, \bar{x}, \varepsilon)}$, we obtain
	\begin{align}
		& \one_A v^{p, \delta}_{K} (\btau, \xi) \nonumber \\
		& \geq \E\Big[ \one_{A \cap \{ \rho < \tau'\}}  v^{p, \delta}_{K} (\rho, X(\rho; \btau, \xi, U^1)) \label{psiAineq} \\
		& \qquad + \one_{A \cap \{ \rho \geq \tau'\} \cap \{ \tau' < T_K\}}   v^n_K(\tau', \xi') \nonumber \\
		& \qquad  + \one_{A \cap \{ \rho \geq \tau'\} \cap \{ \tau' = T_K\}} w_K(G_K - \xi'_0 - (\xi'_1 - C(-\xi'_1))^+)^+  \Big| \cF_{\btau} \Big] \nonumber \\
		& = \E\Big[ \one_{A \cap \{ \rho < \tau'\}}  v^{p, \delta}_{K} (\rho, X(\rho; \btau, \xi, U)) \nonumber \\
		& \qquad + \one_{A \cap \{ \rho \geq \tau'\} \cap \{ \tau' < T_K\}}   v^n_K(\tau', \xi') + \one_{A \cap \{ \rho \geq \tau'\} \cap \{ \tau' = T_K\}} w_K(G_K - \theta_K)^+  \Big| \cF_{\btau} \Big]. \nonumber
	\end{align} 
	In the last equality, we use the definition of $U$ and the fact that $\theta_K = L(\xi')$ under the event $A \cap \{ \rho \geq \tau'\} \cap \{ \tau' = T_K\}$.
	
	Under the event $A^c$, because $U^1$ is a suitable control for $v^n_{K}$ with the random initial condition $(\btau, \xi)$, we have
	\begin{align}
		& \one_{A^c} v^{p, \delta}_{K} (\btau, \xi) = \one_{A^c} v^{n}_{K} (\btau, \xi) \nonumber \\
		& \geq \E\Big[ \one_{A^c \cap \{ \rho < \tau'\}}  v^{n}_{K} (\rho, X(\rho; \btau, \xi, U^1))  \label{Acvnineq} \\
		& \qquad + \one_{A^c \cap \{ \rho \geq \tau'\} \cap \{ \tau' < T_K \}}   v^{n}_{K} (\tau', \xi') + \one_{A^c \cap \{ \rho \geq \tau'\} \cap \{ \tau' = T_K \}} w_K(G_K - \theta^1_K)^+ \Big| \cF_{\btau} \Big] \nonumber \\
		& = \E\Big[ \one_{A^c \cap \{ \rho < \tau'\}}  v^{n}_{K} (\rho, X(\rho; \btau, \xi, U))  \nonumber \\
		& \qquad + \one_{A^c \cap \{ \rho \geq \tau'\} \cap \{ \tau' < T_K \}}   v^{n}_{K} (\tau', \xi') + \one_{A^c \cap \{ \rho \geq \tau'\} \cap \{ \tau' = T_K \}} w_K(G_K - \theta_K)^+ \Big| \cF_{\btau} \Big]. \nonumber
	\end{align} 
	Here, we use $\theta^1_K = \theta_K$ under $A^c \cap \{ \rho \geq \tau'\} \cap \{ \tau' = T_K \}$. As $v^n_K \geq v^{p, \delta}_{K}$ everywhere, the definition of $U$,  \eqref{psiAineq}, and \eqref{Acvnineq} yield
	\begin{align}
		v^{p, \delta}_{K} (\btau, \xi) \geq \E\Big[ & \one_{\{ \rho < \tau'\}}  v^{p, \delta}_{K} (\rho, X(\rho; \btau, \xi, U))  \label{ineq:tau1} \\
		& + \one_{ \{ \rho \geq \tau'\} \cap \{ \tau' < T_K \}}  v^{n}_{K} (\tau', \xi') + \one_{\{ \rho \geq \tau'\} \cap \{ \tau' = T_K \}} w_K(G_K - \theta_K)^+ \Big| \cF_{\btau} \Big]. \nonumber
	\end{align} 
	
	Since $U^2$ is a suitable control for $v^n_{K}$ with the random initial condition $(\tau', \xi')$, \eqref{ineq:tau1} and the definition of $U$ yield the desired result:
	\begin{align*}
		v^{p, \delta}_{K} (\btau, \xi) \geq \E\Big[& \one_{\{ \btau \leq \rho < T_K \}}  v^{p, \delta}_{K} (\rho, X(\rho; \btau, \xi, U)) +  \one_{\{ \rho = T_K\}}  w_K (G_K - \theta_K)^+ \Big| \cF_{\btau} \Big].
	\end{align*}
	It is direct to verify the supermartingale property when $\tau \in [T_{k-1}, T_k]$, $k \neq K$. We omit it here.
	
	{\bf Case 2}. $v_{K, +}(T_K, \bar{x}) - \cM[v_{K, +}]^*(T_K, \bar{x}) > 0$.
	
	Because $\cM[v_{K, +}]^*(T_K, x)$ equals to infinity when $x \in \cS_\emptyset$, we should have $\bar{x} \notin \cS_\emptyset$. Moreover, if $\bar{x} \in \partial \cS_\emptyset$, then there exists a sequence $\{ x_k \}^\infty_{k=1} \subset \cS_\emptyset$ and $x_k \rightarrow \bar{x}$ when $k \rightarrow \infty$, such that  $\cM[v_{K, +}]^*(T_K, \bar{x})$ equals to infinity. It implies that $\bar{x} \notin \partial \cS_\emptyset$. Therefore, we have $\bar{x} \in \barS \setminus \overline{\cS_\emptyset}$ and $\cM[v_{K, +}]^*(T_K, \bar{x}) = \cM[v_{K, +}](T_K, \bar{x})$ by Lemma \ref{lem:Mproperty}. 
	
	Since $v_{K, +}(T_K, \bar{x}) - \cM[v_{K, +}](T_K, \bar{x}) > 0$ and $\cM[v_{K, +}]$ is USC when $(t, x) \in [T_{K-1}, T_K] \times (\barS \setminus \overline{\cS}_\emptyset)$, there exists $\varepsilon > 0$ such that
	\begin{equation}\label{v_eps1}
		v_{K, +}(T_K, \bar{x}) - \cM[v_{K, +}](t, x) \geq \varepsilon, \quad (t, x) \in \overline{\cD(T_K, \bar{x}, \varepsilon)}.
	\end{equation}
	
	Suppose $B(\bar{x}, \varepsilon) \subset \barS \setminus \overline{\cS_\emptyset}$ by choosing $\varepsilon$ small, which implies that $D(x) \neq \emptyset$ for all $x \in B(\bar{x}, \varepsilon)$. Note that after any admissible transaction $\Delta$, the total wealth is reduced by at least $C_{\min}$. We can further assume that the radius $\varepsilon > 0$ is small enough, such that the rebalancing position $\Gamma (x, \Delta)$ is out of $B(\bar{x}, \varepsilon)$ for all $x \in B(\bar{x}, \varepsilon)$ and $\Delta \in D(x)$.
	
	Denote the set of all positions that can be reached by $x \in \overline{B(\bar{x}, \varepsilon)}$ as
	\begin{equation*}
		I_\Gamma := \big\{ \Gamma(x, \Delta) \; \big| \; x \in \overline{B(\bar{x}, \varepsilon)} \text{ and } \Delta \in D(x) \big\}.
	\end{equation*}
	
	By Dini's argument, for $\delta' > 0$, there exists a stochastic supersolution $v^n_K$ such that
	\begin{equation*}
		0 \leq v^n_K(t, x) - v_{K, +}(t, x) \leq \delta', \quad (t, x) \in [T_K - \varepsilon, T_K] \times \overline{I_\Gamma}.
	\end{equation*}
	We can prove that 
	\begin{equation}\label{v_eps2}
		0 \leq \cM[v^n_K](t, x) - \cM[v_{K, +}](t, x) \leq \delta', \quad (t, x) \in \overline{\cD(T_K, \bar{x}, \varepsilon)}.
	\end{equation}
	Define $\psi(t, x) := v_{K, +}(T_K, \bar{x})$. With \eqref{v_eps1} and \eqref{v_eps2}, we obtain 
	\begin{align*}
		\psi(t, x) - \cM[v^n_K](t, x) \geq \varepsilon - \delta', \quad (t, x) \in \overline{\cD(T_K, \bar{x}, \varepsilon)}.
	\end{align*}
	By \citet[Theorem 4.8 (b)]{rieder1978measurable}, for $\delta'' >0$, there exists a Borel measurable $\delta''$-minimizer for $\cM[v^n_K](t, x)$ on $\overline{\cD(T_K, \bar{x}, \varepsilon)}$, denoted as $\Delta^*(t, x)$, such that
	\begin{equation*}
		\cM[v^n_K](t, x) \geq v^n_K(t, \Gamma(x, \Delta^*(t, x))) - \delta'', \quad (t, x) \in \overline{\cD(T_K, \bar{x}, \varepsilon)}.
	\end{equation*}
	If we take $\delta' = \delta'' = \varepsilon/4$, then
	\begin{equation*}
		\psi(t, x) \geq  v^n_K(t, \Gamma(x, \Delta^*(t, x))) + \varepsilon/2, \quad (t, x) \in \overline{\cD(T_K, \bar{x}, \varepsilon)}.
	\end{equation*}
	Let $0 < \eta < \varepsilon/2$ and set $\psi^\eta(t, x) := \psi(t, x) - \eta$. Then
	\begin{equation}\label{eq:psi_Delta}
		\psi^\eta(t, x) \geq  v^n_K(t, \Gamma(x, \Delta^*(t, x))), \quad (t, x) \in \overline{\cD(T_K, \bar{x}, \varepsilon)}.
	\end{equation}
	
	Define
	\begin{equation*}
		v^{\eta}_{K} (t, x) := \left\{ 
		\begin{array}{ c l }
			v^n_K (t, x) \wedge \psi^{\eta}(t, x) & \text{on } \overline{\cD(T_K, \bar{x},\varepsilon)}, \\
			v^n_K (t, x), & \text{otherwise}.
		\end{array}
		\right.
	\end{equation*}
	
	We verify the supermartingale property in Definition \ref{def:sto_super} (3) when the random initial condition $(\btau, \xi)$ satisfies $\btau \in [T_{K-1}, T_K]$.
	
	Similarly, define the event
	\begin{equation*}
		A := \{ (\btau, \xi) \in \cD(T_K, \bar{x},\varepsilon/2) \} \cap \{\psi^{\eta}(\btau, \xi) < v^n_K(\btau, \xi)\}.
	\end{equation*}
	
	Let $U^0 := (\theta^0_{K}, \Lambda^0) := (L(X(T_K; \btau, \xi, \emptyset, \Lambda^0)), \{ (\tau^0_n, \Delta^0_n) \}^\infty_{n=1})$ be a suitable control for $v^n_{K}$ with the random initial condition $(\btau, \xi)$. Define a new control $U^1 := (\theta^1_{K}, \Lambda^1)$ by
	\begin{equation*}
		\theta^1_K  :=  \one_A \emptyset + \one_{A^c} \theta^0_K, \quad \Lambda^1  := \{ (\tau^1_n, \Delta^1_n) \}^\infty_{n=1} := \one_{A} (\btau, \Delta^*(\btau, \xi)) + \one_{A^c} \{ (\tau^0_n, \Delta^0_n) \}^\infty_{n=1}.
	\end{equation*}
	Let 
	\begin{align*}
		\tau' := \inf \{ t \in [\btau, T_K] \, | \,  (t,  X(t; \btau, \xi, U^1)) \notin  \cD(T_K, \bar{x}, \varepsilon/2) \} \wedge T_K
	\end{align*}
	be the exit time and $\xi' :=  X(\tau'; \btau, \xi, U^1) \in \cF_{\tau'}$ be the exit position. 
	
	There is a suitable control $U^2 := (\theta^2_{K}, \Lambda^2) := (L(X(T_K; \tau', \xi', \emptyset, \Lambda^2)), \{ (\tau^2_n, \Delta^2_n) \}^\infty_{n=1})$ for $v^n_{K}$ with the random initial condition $(\tau', \xi')$. This control will only be used when $A^c \cap \{ \tau' < T_K\}$ or $A$ happens. Finally, define a control $U := (\theta_{K}, \Lambda)$ by
	\begin{align*}
		\Lambda  & := \{ (\tau^1_n, \Delta^1_n) \one_{ \{ \tau^1_n \leq \tau' \}} \}^\infty_{n=1}  + \{ (\tau^2_n, \Delta^2_n) \one_{\{ \tau' \leq \tau^2_n \} \cap  \{ A^c \cap \{\tau' < T_K\} \text{ or } A \}} \}^\infty_{n=1} ,\\
		\theta_K  & := L(X(T_K; \btau, \xi, \emptyset, \Lambda)).
	\end{align*}
	We verify that $U$ is suitable for $v^{\eta}_{K}$ with $(\btau, \xi)$.
	
	Consider a stopping time $\rho \in [\btau, T_K]$. Under the event $A$ and control $U^1$, \eqref{eq:psi_Delta} leads to	
	\begin{equation*}
		\one_A v^{\eta}_{K} (\btau, \xi) = \one_A \psi^{\eta} (\btau, \xi) \geq \one_A v^n_K(\btau, \Gamma(\xi, \Delta^*(\btau, \xi))) = \one_A v^n_K(\tau', \xi').
	\end{equation*} 
	Here, we note that the rebalancing position $\Gamma(\xi, \Delta^*(\btau, \xi))$ exits $B(\bar{x}, \varepsilon)$ immediately and hence $\tau' = \btau$. Since $U^2$ is a suitable control for $v^n_K$ with $(\tau', \xi')$, we have
	\begin{align}
		& \one_A v^n_K(\tau', \xi') \nonumber \\
		& \geq \E\Big[ \one_{A \cap \{ \btau \leq \rho < T_K\}}  v^n_{K} (\rho, X(\rho; \tau', \xi', U^2)) + \one_{A \cap \{ \rho = T_K\}} w_K(G_K - \theta^2_K)^+  \Big| \cF_{\btau} \Big] \label{Acase} \\
		& \geq \E\Big[ \one_{A \cap \{ \btau \leq \rho < T_K\}}  v^{\eta}_{K} (\rho, X(\rho; \btau, \xi, U)) + \one_{A \cap \{ \rho = T_K\}} w_K(G_K - \theta_K)^+  \Big| \cF_{\btau} \Big]. \nonumber
	\end{align} 
	The second inequality uses the definition of $U$ and the fact that $v^n_K \geq v^\eta_K$ everywhere.
	
	For the $A^c$ case, we apply the control $U^2$ on $A^c \cap \{ \tau' < T_K\}$ after obtaining the counterpart inequality of \eqref{Acvnineq}. Combining with \eqref{Acase}, the result follows as desired.
	
	It is direct to verify the supermartingale property when $\tau \in [T_{k-1}, T_k]$, $k \neq K$. We omit the detail.  
\end{proof}

\begin{lemma}\label{lem:k_vissub}
	The upper stochastic envelope $v_{+}$ satisfies the viscosity subsolution property \eqref{eq:k_vissub} at $T_k$, $k=1, \ldots, K-1$, under Definition \ref{def:vis_sub}.
\end{lemma}
\begin{proof}
	As $v_{k, +}$ is USC, we obtain $v^*_{k, +} = v_{k, +}$. Assume on the contrary that, there exists $ \bar{x} \in \cS$, such that
	\begin{equation*}
		\begin{aligned}
			\max \Big\{ & v_{k, +}(T_k, \bar{x}) - \inf_{0 \leq \theta_k \leq \bar{x}_0} \left[ w_k (G_k - \theta_k)^+ + v_{k+1, +}(T_k, \bar{x}_0 - \theta_k, \bar{x}_1) \right], \\
			& v_{k, +}(T_k, \bar{x}) - \cM[v_{k, +}]^*(T_k, \bar{x}) \Big\} > 0.
		\end{aligned}
	\end{equation*}
	
	{\bf Case 1}. $v_{k, +}(T_k, \bar{x}) - \inf_{0 \leq \theta_k \leq \bar{x}_0} \left[ w_k (G_k - \theta_k)^+ + v_{k+1, +}(T_k, \bar{x}_0 - \theta_k, \bar{x}_1) \right] > 0$.
	
	By \citet[Theorem 17.21 and Lemma 17.29]{aliprantis2006infinite}, the function given by
	\begin{equation*}
		(x_0, x_1) \mapsto \inf_{0 \leq \theta_k \leq x_0} \left[ w_k (G_k - \theta_k)^+ + v_{k+1, +}(T_k, x_0 - \theta_k, x_1) \right]
	\end{equation*}
	is USC. Then there exists $\varepsilon > 0$ small enough, such that
	\begin{equation}\label{eq:Tk_eps1}
		v_{k, +}(T_k, \bar{x}) - \inf_{0 \leq \theta_k \leq x_0} \left[ w_k (G_k - \theta_k)^+ + v_{k+1, +}(T_k, x_0 - \theta_k, x_1) \right] \geq \varepsilon, \quad \text{ for } x \in \overline{B(\bar{x}, \varepsilon)}.
	\end{equation}
	We introduce the set of positions that can be reached by withdrawing $\theta_k$:
	\begin{equation*}
		I_\theta := \big\{ (x_0 - \theta_k, x_1) \big| x \in \overline{B(\bar{x}, \varepsilon)} \text{ and } 0 \leq \theta_k \leq x_0 \big\}.
	\end{equation*}
	
	By \citet[Proposition 4.1]{bayraktar2012linear}, there exists a nonincreasing sequence of stochastic supersolutions $v^n_{k+1} \searrow v_{k+1, +}$. Moreover, every $v^n_{k+1}$ has a corresponding stochastic supersolution $v^n = (v^n_1, \ldots, v^n_K)$. By \citet[Lemma 2.4]{bayraktar2014Dynkin}, for $\delta' > 0$, there exists a large enough $n_1$ such that 
	\begin{equation*}
		0 \leq v^{n_1}_{k+1}(T_k, x) - v_{k+1, +}(T_k, x) \leq \delta', \quad x \in \overline{I_\theta}.
	\end{equation*}
	By a minimizing sequence argument, we can show that
	\begin{equation}\label{eq:vn1}
		v_{k, +}(T_k, \bar{x}) - \inf_{0 \leq \theta_k \leq x_0} \left[ w_k (G_k - \theta_k)^+ + v^{n_1}_{k+1}(T_k, x_0 - \theta_k, x_1) \right] \geq \varepsilon - \delta', \; \text{ for } x \in \overline{B(\bar{x}, \varepsilon)}.
	\end{equation}
	Besides, $v^{n_1}_{k+1}$ corresponds to a stochastic supersolution $v^{n_1} = (v^{n_1}_1, \ldots$, $ v^{n_1}_K)$. 
	
	With a slight abuse of notation, we define sets
	\begin{align*}
		\cD(T_k, \bar{x}, \varepsilon) & := (T_k - \varepsilon, T_k] \times B(\bar{x}, \varepsilon), \\
		E(\varepsilon) & := \overline{\cD(T_k, \bar{x}, \varepsilon)} \backslash \cD(T_k, \bar{x},\varepsilon/2). 
	\end{align*}

	Similar to Lemma \ref{lem:K_vissub}, for a small $\eta>0$, we can find $v^{n_2}_k$, which corresponds to a stochastic supersolution $v^{n_2} = (v^{n_2}_1, \ldots, v^{n_2}_K)$, and
	\begin{equation}\label{eq:vn2}
		\sup_{(t, x) \in E(\varepsilon)} v^{n_2}_{k}(t, x) - v_{k, +} (T_k, \bar{x}) < \frac{\varepsilon^2}{4 \eta} - \varepsilon. 
	\end{equation}
	Finally, we take
	$$v^n := (v^{n}_1, \ldots, v^{n}_{K}) := (v^{n_1}_1 \wedge v^{n_2}_1, \ldots, v^{n_1}_{K} \wedge v^{n_2}_{K}),$$
	which is a stochastic supersolution by Lemma \ref{lem:two_super}. The inequalities \eqref{eq:vn1} and \eqref{eq:vn2} also hold for $v^n_{k+1}$ and $v^n_k$, respectively.

	By \citet[Theorem 4.8 (b)]{rieder1978measurable}, for $\delta'' >0$, there exists a Borel measurable $\delta''$-minimizer $\theta^*_k(x)$, such that
	\begin{align}
		& \inf_{0 \leq \theta_k \leq x_0} \left[ w_k (G_k - \theta_k)^+ + v^{n}_{k+1}(T_k, x_0 - \theta_k, x_1) \right] \nonumber \\
		& \quad \geq  w_k (G_k - \theta^*_k(x))^+ + v^{n}_{k+1}(T_k, x_0 - \theta^*_k(x), x_1) - \delta'', \quad x \in \overline{\cS}. \label{eq:Tk_eps3}
	\end{align}
	
	With $p > 0$, we introduce
	\begin{equation*}
		\psi^{\varepsilon, \eta, p}(t, x) := v_{k,+} (T_k, \bar{x}) + \frac{|x -  \bar{x}|^2}{\eta} + p(T_k - t).
	\end{equation*}
	Let $\delta' = \delta'' = \varepsilon/4$ and $0 <\delta < \frac{\varepsilon}{2}$. Define $\psi^{p, \delta} := 	\psi^{\varepsilon, \eta, p} - \delta$. With a large enough $p > 0$, we can ensure that $\psi^{p, \delta}$ satisfies the following properties:
	\begin{itemize}
		\item $\cL[\psi^{p, \delta}](t, x) > 0 \quad \text{ on } \overline{\cD(T_k, \bar{x}, \varepsilon)}$.
		
		\item By \eqref{eq:vn2} and the definition of $v^n_k$,
		\begin{equation}\label{prop1_E}
			\psi^{p, \delta}(t, x) \geq v^n_k (t, x) \quad \text{ on } E(\varepsilon).
		\end{equation}
		
		\item By \eqref{eq:Tk_eps1}, \eqref{eq:vn1}, and \eqref{eq:Tk_eps3},
		\begin{equation}\label{prop2_Tkpsi}
			\psi^{p, \delta}(t, x) \geq  w_k (G_k - \theta^*_k(x))^+ + v^{n}_{k+1}(T_k, x_0 - \theta^*_k(x), x_1), \quad (t, x) \in \overline{\cD(T_k, \bar{x}, \varepsilon)}.
		\end{equation}
	\end{itemize}
	
	Hence, we define
	\begin{equation*}
		v^{p, \delta}_{k} (t, x) := \left\{ 
		\begin{array}{ c l }
			v^n_k (t, x) \wedge \psi^{p, \delta}(t, x) & \text{on } \overline{\cD(T_k, \bar{x}, \varepsilon)}, \\
			v^n_k (t, x), & \text{otherwise}.
		\end{array}
		\right.
	\end{equation*}
	
	Next, we show that $(v^n_1, \ldots, v^n_{k-1}, v^{p, \delta}_{k}, v^n_{k+1}, \ldots, v^n_{K})$ is a stochastic supersolution. Only the supermartingale property with $\btau \in [T_{k-1}, T_k]$ is non-trivial. Define the event
	\begin{equation*}
		A := \{ (\btau, \xi) \in \cD(T_k, \bar{x},\varepsilon/2) \} \cap \{\psi^{p, \delta}(\btau, \xi) < v^n_k(\btau, \xi)\}.
	\end{equation*}
	
	Let $U^0 := (\theta^0_{k:K}, \Lambda^0) := (\theta^0_{k:K}, \{ (\tau^0_n, \Delta^0_n) \}^\infty_{n=1})$ be a suitable control for $v^n_{k}$ with the random initial condition $(\btau, \xi)$. Define a new control $U^1 := (\theta^1_{k:K}, \Lambda^1)$ by
	\begin{equation*}
		\theta^1_{k:K}  :=  \one_A \emptyset + \one_{A^c} \theta^0_{k:K}, \quad \Lambda^1  := \{ (\tau^1_n, \Delta^1_n) \}^\infty_{n=1} := \one_{A^c} \{ (\tau^0_n, \Delta^0_n) \}^\infty_{n=1}.
	\end{equation*}
	Here, if $A$ happens, we do not conduct any transactions. Let 
	\begin{align*}
		\tau' := \inf \{ t \in [\btau, T_k] \, | \,  (t,  X(t; \btau, \xi, U^1)) \notin  \cD(T_k, \bar{x}, \varepsilon/2) \} \wedge T_k
	\end{align*}
	be the exit time and $\xi' :=  X(\tau'; \btau, \xi, U^1) \in \cF_{\tau'}$ be the exit position.

	There is a suitable control $U^2 := (\theta^2_{k:K}, \Lambda^2) := (\theta^2_{k:K}, \{ (\tau^2_n, \Delta^2_n) \}^\infty_{n=1})$ for $v^n_{k}$ with the random initial condition $(\tau', \xi')$. Since $\tau' \leq T_k$, the tuple $(T_k, \xi'_0 - \theta^*_k(\xi'), \xi'_1)$ is also a random initial condition. Similarly, there is a suitable control $U^3 := (\theta^3_{k+1:K}, \Lambda^3) := (\theta^3_{k+1:K}, \{ (\tau^3_n, \Delta^3_n) \}^\infty_{n=1})$ for $v^n_{k+1}$ with the random initial condition $(T_k, \xi'_0 - \theta^*_k(\xi'), \xi'_1)$. In the same manner, we introduce a suitable control $U^4 := (\theta^4_{k+1:K}, \Lambda^4) := (\theta^4_{k+1:K}, \{ (\tau^4_n, \Delta^4_n) \}^\infty_{n=1})$ for $v^n_{k+1}$ with the random initial condition $(T_k, \xi')$. Define a control $U := (\theta_{k:K}, \Lambda)$ by
	\begin{align*}
		\Lambda  := & \{ (\tau^1_n, \Delta^1_n) \one_{ \{ \tau^1_n \leq \tau' \}} \}^\infty_{n=1}  + \{ (\tau^2_n, \Delta^2_n) \one_{\{ \tau' \leq \tau^2_n \} \cap  \{ \tau' < T_k \}} \}^\infty_{n=1} \\
		& + \{ (\tau^3_n, \Delta^3_n) \one_{ \{\tau' \leq \tau^3_n \} \cap A \cap \{\tau' = T_k\}} \}^\infty_{n=1}  + \{ (\tau^4_n, \Delta^4_n) \one_{ \{\tau' \leq \tau^4_n \} \cap A^c \cap \{\tau' = T_k\}} \}^\infty_{n=1}, \\
		\theta_k  := & \theta^0_k \one_{A^c \cap \{\tau' = T_k\}}  + \theta^*_k(\xi') \one_{A \cap \{\tau' = T_k\}} + \theta^2_k \one_{\{\tau' < T_k \}}, \\
		\theta_{k+1:K} :=& \theta^2_{k+1:K} \one_{\{\tau' < T_k \}} + \theta^3_{k+1:K} \one_{A \cap \{\tau' = T_k\}} + \theta^4_{k+1:K} \one_{A^c \cap \{\tau' = T_k\}}.
	\end{align*}
	The control $U$ is constructed as follows. First, $U^1$ is applied on $[\btau, \tau']$. Then:
	\begin{itemize}
		\item It the event $A \cap \{\tau' = T_k\}$ occurs, $\xi'$ is the position before the $k$-th withdrawal. We use the amount $\theta^*_k(\xi')$ to support goal $G_k$. After that, we follow $U^3$ on $[\tau', T_K]$.
		\item If the event $A^c \cap \{\tau' = T_k\}$ occurs, it means that the amount $\theta^0_k$ is used and $\xi'$ is the position after supporting $G_k$ already. Then we continue to use $U^4$ on $[\tau', T_K]$.
		\item If the event $\tau' < T_k$ occurs, the control $U^2$ is applied on $[\tau', T_K]$.
	\end{itemize}
	We verify that $U$ is suitable for $v^{p, \delta}_{k}$ with $(\btau, \xi)$.
	
	Consider a stopping time $\rho \in [\btau, T_K]$. Applying the It\^o's formula to $\psi^{p, \delta}$ from $\tau$ to $\rho \wedge \tau'$ under the event $A$ with the control $U^1$, we obtain
	\begin{align*}
		& \one_A v^{p, \delta}_{k} (\btau, \xi) = \one_A \psi^{p, \delta} (\btau, \xi) = \one_A \psi^{p, \delta} (\btau, X(\btau; \btau, \xi, U^1)) \\
		& \geq \E\Big[ \one_{A \cap \{ \rho < \tau'\}}  \psi^{p, \delta} (\rho, X(\rho; \btau, \xi, U^1)) + \one_{A \cap \{ \rho \geq \tau'\}}   \psi^{p, \delta} (\tau', \xi') \Big| \cF_{\btau} \Big].
	\end{align*} 
	Moreover, \eqref{prop1_E} and \eqref{prop2_Tkpsi} lead to
	\begin{align*}
		\one_{A \cap \{ \rho \geq \tau'\}} \psi^{p, \delta} (\tau', \xi') \geq & \one_{A \cap \{ \rho \geq \tau'\} \cap \{ \tau' < T_k\}} v^n_k(\tau', \xi')  \\
		& + \one_{A \cap \{ \rho \geq \tau'\} \cap \{ \tau' = T_k\}} \Big( w_k (G_k - \theta^*_k(\xi'))^+ + v^{n}_{k+1}(T_k, \xi'_0 - \theta^*_k(\xi'), \xi'_1) \Big). 
	\end{align*}
	Since $v^{p, \delta}_{k} \leq \psi^{p, \delta}$ on $\overline{\cD(T_k, \bar{x}, \varepsilon)}$, we obtain
	\begin{align}
		& \one_A v^{p, \delta}_{k} (\btau, \xi) \nonumber \\
		& \geq \E\Big[ \one_{A \cap \{ \rho < \tau'\}}  v^{p, \delta}_{k} (\rho, X(\rho; \btau, \xi, U^1)) \\
		& \qquad + \one_{A \cap \{ \rho \geq \tau'\} \cap \{ \tau' < T_k\}}   v^n_k(\tau', \xi') \nonumber \\
		& \qquad  + \one_{A \cap \{ \rho \geq \tau'\} \cap \{ \tau' = T_k\}} \Big( w_k (G_k - \theta^*_k(\xi'))^+ + v^{n}_{k+1}(T_k, \xi'_0 - \theta^*_k(\xi'), \xi'_1) \Big) \Big| \cF_{\btau} \Big] \nonumber \\
		& = \E\Big[ \one_{A \cap \{ \rho < \tau'\}}  v^{p, \delta}_{k} (\rho, X(\rho; \btau, \xi, U)) \\
		& \qquad + \one_{A \cap \{ \rho \geq \tau'\} \cap \{ \tau' < T_k\}}   v^n_k(\tau', \xi') \nonumber \\
		& \qquad  + \one_{A \cap \{ \rho \geq \tau'\} \cap \{ \tau' = T_k\}} \Big( w_k (G_k - \theta_k)^+ + v^{n}_{k+1}(T_k, \xi'_0 - \theta_k, \xi'_1) \Big) \Big| \cF_{\btau} \Big]. \nonumber
	\end{align} 
	In the last equality, we use the definition of $U$ and the fact that $\theta_k = \theta^*_k(\xi')$ under the event $A \cap \{ \rho \geq \tau'\} \cap \{ \tau' = T_k\}$.
	
	Similar to Lemma \ref{lem:K_vissub}, under the event $A^c$, we have
	\begin{align}
		& \one_{A^c} v^{p, \delta}_{k} (\btau, \xi) = \one_{A^c} v^{n}_{k} (\btau, \xi) \nonumber \\
		& \geq \E\Big[ \one_{A^c \cap \{ \rho < \tau'\}}  v^{n}_{k} (\rho, X(\rho; \btau, \xi, U^1))  \\
		& \qquad + \one_{A^c \cap \{ \rho \geq \tau'\} \cap \{ \tau' < T_k \}}   v^{n}_{k} (\tau', \xi') \nonumber \\
		& \qquad + \one_{A^c \cap \{ \rho \geq \tau'\} \cap \{ \tau' = T_k\}} \Big( w_k (G_k - \theta^0_k)^+ + v^{n}_{k+1}(T_k, \xi') \Big) \Big| \cF_{\btau} \Big] \nonumber \\
		& = \E\Big[ \one_{A^c \cap \{ \rho < \tau'\}}  v^{n}_{k} (\rho, X(\rho; \btau, \xi, U))  \nonumber \\
		& \qquad + \one_{A^c \cap \{ \rho \geq \tau'\} \cap \{ \tau' < T_k \}}   v^{n}_{k} (\tau', \xi') \nonumber \\
		& \qquad + \one_{A^c \cap \{ \rho \geq \tau'\} \cap \{ \tau' = T_k\}} \Big( w_k (G_k - \theta_k)^+ + v^{n}_{k+1}(T_k, \xi') \Big) \Big| \cF_{\btau} \Big]. \nonumber
	\end{align} 
	These two inequalities yield  
	\begin{align}
		v^{p, \delta}_{k} (\btau, \xi) \geq \E\Big[ & \one_{\{ \rho < \tau'\}}  v^{p, \delta}_{k} (\rho, X(\rho; \btau, \xi, U)) \\
		& + \one_{ \{ \rho \geq \tau'\} \cap \{ \tau' < T_k \}}  v^{n}_{k} (\tau', \xi') \nonumber \\
		& + \one_{A \cap \{ \rho \geq \tau'\} \cap \{ \tau' = T_k\}} \Big( w_k (G_k - \theta_k)^+ + v^{n}_{k+1}(T_k, \xi'_0 - \theta_k, \xi'_1) \Big) \nonumber \\
		& + \one_{A^c \cap \{ \rho \geq \tau'\} \cap \{ \tau' = T_k\}} \Big( w_k (G_k - \theta_k)^+ + v^{n}_{k+1}(T_k, \xi') \Big) \Big| \cF_{\btau} \Big]. \nonumber
	\end{align} 
	The definition of $U$ leads to the desired result:
	\begin{equation*}
		v^{p, \delta}_k(\btau, \xi) \geq \E \big[ \cH \big([\btau, \rho], (v^{p, \delta}_k, v_{k+1:K}), X(\cdot; \btau, \xi, \theta_{k:K}, \Lambda) \big) \big| \cF_{\btau} \big].
	\end{equation*}
	
	{\bf Case 2}.  $v_{k, +}(T_k, \bar{x}) - \cM[v_{k, +}]^*(T_k, \bar{x}) > 0.$
	
	This case is similar to Lemma \ref{lem:K_vissub}. We report the control $U$ only. Let  $U^0$ be a suitable control for $v^n_{k}$ with $(\btau, \xi)$. Define $U^1 := (\theta^1_{k:K}, \Lambda^1)$ by
	\begin{equation*}
		\theta^1_{k:K}  :=  \one_A \emptyset + \one_{A^c} \theta^0_{k:K}, \quad \Lambda^1  := \{ (\tau^1_n, \Delta^1_n) \}^\infty_{n=1} := \one_{A} (\btau, \Delta^*(\btau, \xi)) + \one_{A^c} \{ (\tau^0_n, \Delta^0_n) \}^\infty_{n=1},
	\end{equation*}
	where $\Delta^*(t, x)$ is defined similarly as in Lemma \ref{lem:K_vissub}. Let $(\tau', \xi')$ be the exit time and position as before. There is a suitable control $U^2 := (\theta^2_{k:K}, \Lambda^2) := (\theta^2_{k:K}, \{ (\tau^2_n, \Delta^2_n) \}^\infty_{n=1})$ for $v^n_{k}$ with the random initial condition $(\tau', \xi')$. Besides, we introduce a suitable control $U^4 := (\theta^4_{k+1:K}, \Lambda^4) := (\theta^4_{k+1:K}, \{ (\tau^4_n, \Delta^4_n) \}^\infty_{n=1})$ for $v^n_{k+1}$ with the random initial condition $(T_k, \xi')$. Define a control $U := (\theta_{k:K}, \Lambda)$ by
	\begin{align*}
		\Lambda  := & \{ (\tau^1_n, \Delta^1_n) \one_{ \{ \tau^1_n \leq \tau' \}} \}^\infty_{n=1}  + \{ (\tau^2_n, \Delta^2_n) \one_{\{ \tau' \leq \tau^2_n \} \cap  \{ A^c \cap \{\tau' < T_k \} \text{ or } A \} } \}^\infty_{n=1} \\
		&  + \{ (\tau^4_n, \Delta^4_n) \one_{ \{\tau' \leq \tau^4_n \} \cap A^c \cap \{\tau' = T_k\}} \}^\infty_{n=1}, \\
		\theta_k  := & \theta^0_k \one_{A^c \cap \{\tau' = T_k\}} + \theta^2_k \one_{\{ A^c \cap \{\tau' < T_k \} \text{ or } A \} }, \\
		\theta_{k+1:K}  := & \theta^4_{k+1:K} \one_{A^c \cap \{\tau' = T_k\}} + \theta^2_{k+1:K} \one_{\{ A^c \cap \{\tau' < T_k \} \text{ or } A \} }.
	\end{align*}
	The control $U$ is suitable for $v^{p, \delta}_k$ with $(\btau, \xi)$.
\end{proof}

\begin{lemma}\label{lem:interior_vissub}
	The upper stochastic envelope $v_{+}$ satisfies the interior viscosity subsolution property \eqref{eq:interior_vissub} on $[T_{k-1}, T_k) \times \cS$, $k=1,\ldots, K$, under Definition \ref{def:vis_sub}.
\end{lemma}
\begin{proof}
	Let $(\bar{t},\bar{x}) \in [T_{k-1}, T_k) \times \cS$. Consider a test function $\varphi \in C^{1, 2}([T_{k-1}, T_k) \times \cS)$, such that  $v_{k, +} - \varphi$ attains a strict local maximum of zero at $(\bar{t},\bar{x})$. Assume on the contrary that
	\begin{equation}
		\max \Big\{  \cL[\varphi](\bar{t}, \bar{x}), v_{k, +}(\bar{t}, \bar{x}) - \cM[v_{k, +}]^*(\bar{t}, \bar{x}) \Big\} > 0.
	\end{equation} 
	
	{\bf Case 1}. $\cL[\varphi](\bar{t}, \bar{x}) > 0$.
	
	The proof is similar to \citet[Theorem 3.1]{bayraktar2013stochastic}. We give the main steps and omit similar arguments. With a small $\eta > 0$, we define $\varphi^\eta(t, x) := \varphi(t, x) - \eta$. Moreover, $\varphi^\eta$ satisfies the following properties:
	\begin{itemize}
		\item $\cL[\varphi^{\eta}](t, x) > 0 \text{ on } \overline{B(\bar{t}, \bar{x}, \varepsilon)}$.
		
		\item $\varphi^{\eta}(t, x) \geq v^n_k (t, x) \text{ on } \overline{B(\bar{t}, \bar{x}, \varepsilon)} \setminus B(\bar{t}, \bar{x}, \varepsilon/2)$.
		
		\item $\varphi^{\eta}(\bar{t}, \bar{x}) < v_{k, +}(\bar{t}, \bar{x}) $.
	\end{itemize}
	
	We introduce
	\begin{equation}
		v^{\eta}_{k} (t, x) := \left\{ 
		\begin{array}{ c l }
			v^n_k (t, x) \wedge \varphi^{\eta}(t, x) & \text{on } \overline{B(\bar{t}, \bar{x}, \varepsilon)}, \\
			v^n_k (t, x), & \text{otherwise}.
		\end{array}
		\right.
	\end{equation}
	
	To show that $(v^n_1, \ldots, v^n_{k-1}, v^{\eta}_{k}, v^n_{k+1}, \ldots, v^n_{K})$ is a stochastic supersolution, we only need to consider the case with $\btau \in [T_{k-1}, T_k]$. Define the event
	\begin{equation*}
		A := \{ (\btau, \xi) \in B(\bar{t}, \bar{x},\varepsilon/2) \} \cap \{\varphi^{\eta}(\btau, \xi) < v^n_k(\btau, \xi)\}.
	\end{equation*}
	
	Let $U^0 := (\theta^0_{k:K}, \Lambda^0) := (\theta^0_{k:K}, \{ (\tau^0_n, \Delta^0_n) \}^\infty_{n=1})$ be a suitable control for $v^n_{k}$ with the random initial condition $(\btau, \xi)$. Define a new control $U^1 := (\theta^1_{k:K}, \Lambda^1)$ by
	\begin{equation*}
		\theta^1_{k:K}  :=  \one_A \emptyset + \one_{A^c} \theta^0_{k:K}, \quad \Lambda^1  := \{ (\tau^1_n, \Delta^1_n) \}^\infty_{n=1} := \one_{A^c} \{ (\tau^0_n, \Delta^0_n) \}^\infty_{n=1}.
	\end{equation*}
	Let 
	\begin{align*}
		\tau' := \inf \{ t \in [\btau, T_k] \, | \,  (t,  X(t; \btau, \xi, U^1)) \notin  B(\bar{t}, \bar{x}, \varepsilon/2) \} \wedge T_k
	\end{align*}
	be the exit time and $\xi' :=  X(\tau'; \btau, \xi, U^1) \in \cF_{\tau'}$ be the exit position. 
	
	We introduce a suitable control $U^2 := (\theta^2_{k:K}, \Lambda^2) := (\theta^2_{k:K}, \{ (\tau^2_n, \Delta^2_n) \}^\infty_{n=1})$ for $v^n_{k}$ with $(\tau', \xi')$, and a suitable control $U^4 := (\theta^4_{k+1:K}, \Lambda^4) := (\theta^4_{k+1:K}, \{ (\tau^4_n, \Delta^4_n) \}^\infty_{n=1})$ for $v^n_{k+1}$ with $(T_k, \xi')$. Define $U := (\theta_{k:K}, \Lambda)$ by
	\begin{equation}\label{eq:interior_U}
		\begin{aligned}
			\Lambda  := & \{ (\tau^1_n, \Delta^1_n) \one_{ \{ \tau^1_n \leq \tau' \}} \}^\infty_{n=1}  + \{ (\tau^2_n, \Delta^2_n) \one_{\{ \tau' \leq \tau^2_n \} \cap  \{ A^c \cap \{\tau' < T_k \} \text{ or } A \} } \}^\infty_{n=1} \\
			&  + \{ (\tau^4_n, \Delta^4_n) \one_{ \{\tau' \leq \tau^4_n \} \cap A^c \cap \{\tau' = T_k\}} \}^\infty_{n=1}, \\
			\theta_k  := & \theta^0_k \one_{A^c \cap \{\tau' = T_k\}} + \theta^2_k \one_{\{ A^c \cap \{\tau' < T_k \} \text{ or } A \} }, \\
			\theta_{k+1:K}  := & \theta^4_{k+1:K} \one_{A^c \cap \{\tau' = T_k\}} + \theta^2_{k+1:K} \one_{\{ A^c \cap \{\tau' < T_k \} \text{ or } A \} }.
		\end{aligned}
	\end{equation}
	Then the remaining proof follows similarly.
	
	{\bf Case 2}. $v_{k, +}(\bar{t}, \bar{x}) - \cM[v_{k, +}]^*(\bar{t}, \bar{x}) > 0$.
	
	Again, we report the control $U$ only. Let  $U^0$ be a suitable control for $v^n_{k}$ with $(\btau, \xi)$. Define $U^1 := (\theta^1_{k:K}, \Lambda^1)$ by
	\begin{equation*}
		\theta^1_{k:K}  :=  \one_A \emptyset + \one_{A^c} \theta^0_{k:K}, \quad \Lambda^1  := \{ (\tau^1_n, \Delta^1_n) \}^\infty_{n=1} := \one_{A} (\btau, \Delta^*(\btau, \xi)) + \one_{A^c} \{ (\tau^0_n, \Delta^0_n) \}^\infty_{n=1},
	\end{equation*}
	where $\Delta^*(t, x)$ is defined similarly as in Lemma \ref{lem:K_vissub}. Then $U$ can be constructed as in \eqref{eq:interior_U}.
\end{proof}

\section{Proofs of the stochastic subsolution}
\begin{proof}[Proof of Lemma \ref{lem:classical_sub}]
	{\bf Step 1}. We prove the inequality \eqref{eq:sub_TK} at $T_K$ first. Consider the first term in \eqref{eq:sub_TK}. Since $(x_1 - C(-x_1))^+ \leq x_1$, we have $G_K - x_0 - (x_1 - C(-x_1))^+ \geq G_K - x_0 - x_1$, which further implies
	\begin{equation}
		[G_K - x_0 - (x_1 - C(-x_1))^+]^+ \geq (G_K - x_0 - x_1)^+.
	\end{equation} 
	If $x_0 + x_1 > G_K$, then
	\begin{align*}
		& F^a_K(T_K, x) - w_K \left[ G_K - x_0 - (x_1 - C(-x_1))^+ \right]^+ \\
		& \leq F^a_K(T_K, x) - w_K \left[ G_K - x_0 - x_1 \right]^+ \\
		& = F^a_K(T_K, x) \\
		& = w_K G_K - 2 w_K G^{1-q}_K (a + x_0 + x_1)^q \\
		& \leq w_K G_K - 2 w_K G^{1-q}_K G^q_K = - w_K G_K < 0.
	\end{align*} 
	If $x_0 + x_1 \leq G_K$ and $(x_0, x_1) \neq (0, 0)$, we obtain
	\begin{align*}
		& F^a_K(T_K, x) - w_K \left[ G_K - x_0 - (x_1 - C(-x_1))^+ \right]^+ \\
		& \leq F^a_K(T_K, x) - w_K \left[ G_K - x_0 - x_1 \right]^+ \\
		& = F^a_K(T_K, x) - w_K G_K + w_K(x_0 + x_1)\\
		& = w_K G_K - 2 w_K G^{1-q}_K (a + x_0 + x_1)^q - w_K G_K + w_K(x_0 + x_1)\\
		& \leq - 2 w_K G^{1-q}_K (x_0 + x_1)^q + w_K(x_0 + x_1) \\
		& = w_K (x_0 + x_1)^q [-2 G^{1-q}_K + (x_0 + x_1)^{1-q}] \\
		& \leq w_K (x_0 + x_1)^q [-2 G^{1-q}_K + G^{1-q}_K] = - w_K (x_0 + x_1)^q G^{1-q}_K < 0.
	\end{align*}
	Combining these two inequalities together,
	\begin{equation}\label{eq:FK_bd}
		\begin{aligned}
			& F^a_K(T_K, x) - w_K \left[ G_K - x_0 - (x_1 - C(-x_1))^+ \right]^+ \\
			& \leq - w_K (\min\{x_0 + x_1, G_K\})^q G^{1-q}_K < 0, \quad x \in \cS.
		\end{aligned}
	\end{equation} 
	For the second term in \eqref{eq:sub_TK}, if $x \in \cS_\emptyset$, then
	\begin{equation}\label{eq:trans0}
		F^a_K(T_K, x) - \cM[F^a_K](T_K, x) = - \infty.
	\end{equation}
	If $x \notin \cS_\emptyset$, then
	\begin{equation}\label{eq:trans1}
		\begin{aligned}
			& F^a_K(T_K, x) - \cM[F^a_K](T_K, x) \\
			& = w_K G_K - 2 w_K G^{1-q}_K (a + x_0 + x_1)^q \\
			& \quad - \inf_{\Delta \in D(x)} [w_K G_K - 2 w_K G^{1-q}_K (a + x_0 + x_1 - C(\Delta))^q] \\
			& \leq 2 w_K G^{1-q}_K [(a + x_0 + x_1 - C_{\min})^q - (a+x_0 + x_1)^q] < 0.
		\end{aligned}
	\end{equation}
	
	Clearly, a continuous function $\kappa^b_K(x)$ exists, with $\kappa^b_K(x) \leq 0$ for $x \in \barS$ and $\kappa^b_K(x) < 0$ for $x \in \cS$. 
	
	{\bf Step 2}. Next, we prove \eqref{eq:sub_t}. Clearly, the term $F^a_k(t, x) - \cM[F^a_k](t, x)$ can be handled as in the Step 1. For the infinitesimal generator term, we have
	\begin{align*}
		& \cL[F^a_k](t, x) \\
		& \quad = C_k e^{\lambda (T_k - t)} (a + x_0 + x_1)^q \Big\{ - \lambda + \frac{qr x_0}{a + x_0 + x_1} + \frac{q \mu x_1}{a + x_0 + x_1} + \frac{q (q - 1) \sigma^2 x^2_1}{2 (a + x_0 + x_1)^2} \Big\} \\
		& \quad \leq C_k e^{\lambda (T_k - t)} (a + x_0 + x_1)^q ( - \lambda + q \max\{ r, \mu, 0\} ) < 0, \; \text{ if } x \in \cS.
	\end{align*}
	Then we can find $\kappa^c_k(x)$ satisfying required properties.
	
	{\bf Step 3}. For the inequality at $T_k$, we only need to consider the first term:
	\begin{align*}
		& F^a_k(T_k, x) - \inf_{0 \leq \theta_k \leq x_0} \left[ w_k (G_k - \theta_k)^+ + F^a_{k+1}(T_k, x_0 - \theta_k, x_1) \right] \\
		& \quad \leq  F^a_k(T_k, x)  - w_k (G_k - x_0 - x_1)^+ - F^a_{k+1}(T_k, x_0 - 0, x_1) \\
		& \quad = \sum^{K}_{i = k} w_i G_i - C_k (a + x_0 + x_1)^q - w_k (G_k - x_0 - x_1)^+ \\
		& \qquad - \Big\{ \sum^{K}_{i = k+1} w_i G_i - C_{k+1} (a + x_0 + x_1)^q e^{\lambda(T_{k+1} - T_k)}  \Big\} \\
		& \quad = w_k G_k - w_k (G_k - x_0 - x_1)^+ - 2w_k  G^{1-q}_k (a + x_0 + x_1)^q \\
		& \quad \leq w_k G_k - w_k (G_k - x_0 - x_1)^+ - 2w_k  G^{1-q}_k (x_0 + x_1)^q.
	\end{align*}
	Similar to the Step 1, if $x_0 + x_1 > G_k$, then
	\begin{align*}
		& w_k G_k - w_k (G_k - x_0 - x_1)^+ - 2w_k  G^{1-q}_k (x_0 + x_1)^q \leq - w_k G_k < 0.
	\end{align*}
	If $x_0 + x_1 \leq G_k$ and $(x_0, x_1) \neq (0, 0)$, we have
	\begin{align*}
		& w_k G_k - w_k (G_k - x_0 - x_1)^+ - 2w_k  G^{1-q}_k (x_0 + x_1)^q \leq - w_k G^{1-q}_k (x_0 + x_1)^q < 0.
	\end{align*}
	Hence, there exists $\kappa^b_k(x)$ with desired properties.
\end{proof}

\begin{proof}[Proof of Lemma \ref{lem:F0}]
	Since $F^0_k$ is continuous, Condition (1) on the LSC property holds. The growth condition (2) also holds directly. Condition (3) is verified in the proof of Lemma \ref{lem:classical_sub}, in the same spirit of \eqref{eq:trans0} and \eqref{eq:trans1}.
	
	Finally, we verify Condition (4). At the goal deadline $T_k$, where $k = 1, \ldots, K-1$, Lemma \ref{lem:classical_sub} indicates that 
	\begin{equation}\label{ineq:Tk}
		F^0_k(T_k, x) \leq w_k (G_k - \theta_k)^+ + F^0_{k+1}(T_k, x_0 - \theta_k, x_1), 
	\end{equation}
	for all $x \in \cS$ and admissible $\theta_k$. At the last deadline $T_K$, \eqref{eq:FK_bd} in the proof for Lemma \ref{lem:classical_sub} and $F^0_K(T_K, 0) = w_K G_K$ imply that
	\begin{equation}\label{eq:FKTK}
		F^0_K(T_K, x) \leq w_K \left[ G_K - x_0 - (x_1 - C(-x_1))^+ \right]^+,
	\end{equation}
	for all $x \in \barS$.

	Between goal deadlines, we can apply the It\^o's formula together with the property $\cL[F^0_k](t, x) < 0$ for $x \in \cS$. As a demonstration, we consider the case when $k=K-1$, $\btau \in [T_{K-2}, T_{K-1}]$, and $T_{K-1} \leq \rho \leq T$. If the random initial value $\xi \neq 0$, then a recursive application of the properties mentioned above shows that
	\begin{align*}
		F^0_{K-1} (\btau, \xi) & = F^0_{K-1} (T_{K-1},  X(T_{K-1} -)) \\
		& \quad + \int^{T_{K-1}}_{\btau} \cL[F^0_{K-1}](t, X(t)) dt - \int^{T_{K-1}}_{\btau} \sigma X_1(t) \frac{\partial F^0_{K-1}}{\partial x_1}(t, X(t)) dW(t) \\
		& \leq w_{K-1} (G_{K-1} - \theta_{K-1})^+ + F^0_K(T_{K-1}, X_0(T_{K-1}-) - \theta_{K-1}, X_1(T_{K-1})) \\
		& \quad + \int^{T_{K-1}}_{\btau} \cL[F^0_{K-1}](t, X(t)) dt - \int^{T_{K-1}}_{\btau} \sigma X_1(t) \frac{\partial F^0_{K-1}}{\partial x_1}(t, X(t)) dW(t) \\
		& = w_{K-1} (G_{K-1} - \theta_{K-1})^+ + F^0_{K}(\rho, X(\rho-)) \\
		& \quad + \int^{\rho}_{T_{K-1}} \cL[F^0_{K}](t, X(t)) dt - \int^{\rho}_{T_{K-1}} \sigma X_1(t) \frac{\partial F^0_{K}}{\partial x_1}(t, X(t)) dW(t)  \\
		& \quad + \int^{T_{K-1}}_{\btau} \cL[F^0_{K-1}](t, X(t)) dt - \int^{T_{K-1}}_{\btau} \sigma X_1(t) \frac{\partial F^0_{K-1}}{\partial x_1}(t, X(t)) dW(t),
	\end{align*}
	where $X(t)$ represents $X(t; \btau, \xi, \theta_{K-1:K}, \emptyset)$. Thanks to \eqref{eq:FKTK}, we have
	\begin{align*}
		F^0_{K}(\rho, X(\rho-)) & = F^0_{K}(\rho, X(\rho)) \one_{\{\rho < T \}}  + F^0_{K}(T, X(T-)) \one_{\{\rho = T \}} \\
		& \leq F^0_{K}(\rho, X(\rho)) \one_{\{\rho < T \}} + w_K \left[ G_K - \theta_K \right]^+ \one_{\{\rho = T \}}. 
	\end{align*}
	Combining these two inequalities together, a localization argument with Fatou's lemma yields the corresponding Condition (4) when $\xi \neq 0$. If $\xi = 0$, both $X_0$ and $X_1$ stay at zero and Condition (4) follows from the explicit value of $F^0_k(t, 0)$. The proof for the general $k$ and $\rho \in [\btau, T]$ is in the same spirit while lengthy. 
\end{proof}

\begin{lemma}\label{lem:K_vissup}
	The lower stochastic envelope $v_{-}$ satisfies the viscosity supersolution property \eqref{K_vissup} at $T_K$, under Definition \ref{def:vis_super}.
\end{lemma}
\begin{proof} 
	Since $v_{K,-}$ itself is also LSC, we have $v_{K,-,*} = v_{K,-}$ and $\cM[v_{K,-}]_* = \cM[v_{K,-}]$ by Lemma \ref{lem:Mproperty}. Assume on the contrary that there exists $\bar{x} := (\bar{x}_0, \bar{x}_1) \in \cS$, such that 
	\begin{equation}
		\begin{aligned}
			\max \Big\{ & v_{K,-}(T_K, \bar{x}) - w_K \left[ G_K - \bar{x}_0 - (\bar{x}_1 - C(-\bar{x}_1))^+ \right]^+ , \\
			& v_{K,-}(T_K, \bar{x}) - \cM[v_{K,-}](T_K, \bar{x}) \Big\} < 0.
		\end{aligned}
	\end{equation}
	For $\varepsilon > 0$ small enough, we define several sets for later use:
	\begin{equation}
		\begin{aligned}
			B(\bar{x}, \varepsilon) & := \{ x | x \in \barS \text{ and } |x - \bar{x}| < \varepsilon \}, \\
			\cD(T_K, \bar{x}, \varepsilon) & := (T_K - \varepsilon, T_K] \times B(\bar{x}, \varepsilon), \\
			E(\varepsilon) &:= \overline{\cD(T_K, \bar{x}, \varepsilon)} \setminus \cD(T_K, \bar{x}, \varepsilon/2).
		\end{aligned}
	\end{equation}
	
	Since $\cM[v_{K,-}]$ is LSC and $\left[ G_K - x_0 - (x_1 - C(-x_1))^+ \right]^+$ is continuous, there exists $\varepsilon > 0$ small enough, such that 
	\begin{equation}\label{eq:sub_eps}
		\begin{aligned}
			\max \Big\{ & v_{K,-}(T_K, \bar{x}) - w_K \left[ G_K - x_0 - (x_1 - C(-x_1))^+ \right]^+ , \\
			& v_{K,-}(T_K, \bar{x}) - \cM[v_{K,-}](t, x) \Big\} \leq -\varepsilon,
		\end{aligned}
	\end{equation}
	when $(t, x) \in \overline{\cD(T_K, \bar{x}, \varepsilon)}$.
	
	As $v_{K,-}$ is LSC and $E(\varepsilon)$ is compact, the function $v_{K,-}$ is bounded from below on $E(\varepsilon)$. With a small enough $\eta > 0$, we have
	\begin{equation}\label{eq:g_eta}
		- \inf_{(t, x) \in E(\varepsilon)} v_{K,-}(t, x) + v_{K,-}(T_K, \bar{x}) < \frac{\varepsilon^2}{4 \eta} - \varepsilon.
	\end{equation} 
	Note that $v_{-} \in \cV^-$. For $p >0$, we define
	\begin{equation*}
		\psi^{\varepsilon, \eta, p}(t, x) := v_{K,-}(T_K, \bar{x}) - \frac{|x - \bar{x}|^2}{\eta} - p(T_K - t).
	\end{equation*} 
	With a large enough $p$,
	\begin{equation}\label{eq:L_psi}
		\cL[\psi^{\varepsilon, \eta, p}] (t, x) < 0 \text{ holds for } (t, x) \in \overline{\cD(T_K, \bar{x}, \varepsilon)}. 
	\end{equation}
	By the definition of $E(\varepsilon)$, the property of $v_{K,-}$ in \eqref{eq:g_eta}, and making $p$ sufficiently large, we obtain the following inequality when $(t, x) \in E(\varepsilon)$:
	\begin{align*}
		\psi^{\varepsilon, \eta, p}(t, x) < v_{K,-}(T_K, \bar{x}) - \frac{\varepsilon^2}{4 \eta} < \inf_{(t, x) \in E(\varepsilon)} v_{K,-}(t, x) - \varepsilon \leq v_{K,-}(t, x) - \varepsilon.
	\end{align*}
	Besides, \eqref{eq:sub_eps} implies that
	\begin{equation}\label{eq:psi}
		\psi^{\varepsilon, \eta, p}(t, x) \leq v_{K,-}(T_K, \bar{x}) \leq w_K \left[ G_K - x_0 - (x_1 - C(-x_1))^+ \right]^+  - \varepsilon,
	\end{equation}
	when $(t, x) \in \overline{\cD(T_K, \bar{x}, \varepsilon)}$.
	
	Let $0 < \delta < \varepsilon$ be small enough and define
	\begin{equation}\label{hdelta}
		v^{\delta}_{K} (t, x) = \left\{ 
		\begin{array}{cl}
			v_{K,-} (t, x) \vee ( \psi^{\varepsilon, \eta, p}(t, x) + \delta) & \text{on } \overline{\cD(T_K, \bar{x}, \varepsilon)}, \\
			v_{K,-} (t, x), & \text{otherwise}.
		\end{array}
		\right.
	\end{equation}
	
	We verify that $(v_{1,-}, \ldots, v_{K-1,-}, v^\delta_K)$ is a stochastic subsolution under Definition \ref{def:sto_sub}. Since $v^\delta_K$ is LSC and satisfies the polynomial growth condition with order $p_0 \in (0, 1)$, Conditions (1) and (2) in Definition \ref{def:sto_sub} are satisfied. 
	
	For Condition (3), we only need to verify it for $v^\delta_K$. As $v^\delta_K \geq v_{K,-}$ everywhere, we have
	\begin{equation}
		v^\delta_K(t, x) - \cM[v^\delta_K](t, x) \leq v^\delta_K(t, x) - \cM[v_{K,-}](t, x).
	\end{equation}
	If $v^\delta_K(t, x) = v_{K,-}(t, x)$, then Condition (3) is satisfied. If $v^\delta_K(t, x) = \psi^{\varepsilon, \eta, p}(t, x) + \delta$ instead, then it must be $(t, x) \in \overline{\cD(T_K, \bar{x}, \varepsilon)}$. It leads to
	\begin{align*}
		& \psi^{\varepsilon, \eta, p}(t, x) + \delta - \cM[v_{K,-}](t, x) \\
		& = v_{K,-}(T_K, \bar{x}) - \frac{|x - \bar{x}|^2}{\eta} - p(T_K - t) + \delta - \cM[v_{K,-}](t, x) \\
		& \leq v_{K,-}(T_K, \bar{x})  - \cM[v_{K,-}](t, x) + \delta \\
		& \leq - \varepsilon + \delta < 0,
	\end{align*}
	where \eqref{eq:sub_eps} is used. Therefore, $v^\delta_K$ satisfies Condition (3).
	
	For Condition (4), as $(v_{1,-}, \ldots, v_{K-1,-})$ satisfies Condition (4) and $v^\delta_K \geq v_{K,-}$ everywhere, we only need to prove it for $v^\delta_K$. Consider any random initial condition $(\btau, \xi)$ with $\btau \in [T_{K-1}, T_K]$, $\xi \in \cF_{\btau}$, and $\p(\xi \in \barS) = 1$. Note that the last withdrawal $\theta_K$ is specified by the liquidation. Define the event 
	\begin{equation}
		A := \{(\btau, \xi) \in \cD(T_K, \bar{x}, \varepsilon/2) \} \cap \{ \psi^{\varepsilon, \eta, p}(\btau, \xi) + \delta > v_{K,-}(\btau, \xi)\}.
	\end{equation}
	Then $A \in \cF_{\btau}$. Let 
	\begin{align*}
		\tau^1 := \inf \big\{ t \in [\btau, T_K] \, \big| &  (t,  X(t; \btau, \xi, \theta_K, \emptyset)) \notin \cD(T_K, \bar{x}, \varepsilon/2) \big\} \wedge T_K
	\end{align*}
	be the exit time and denote
	\begin{align*}
		\xi^{1} := (\xi^{1}_0, \xi^{1}_1) :=  X(\tau^1; \btau, \xi, \theta_K, \emptyset) \in \cF_{\tau^1}
	\end{align*}
	as the exit position. Since it is possible that $\tau^1 = T$, we also introduce  
	\begin{align*}
		\xi^{1-} := (\xi^{1-}_0, \xi^{1-}_1) :=  X(\tau^1-; \btau, \xi, \theta_K, \emptyset)
	\end{align*}
	as the position that excludes any jump caused by $\theta_K$. 
	
	Let $\rho \in [\btau, T]$ be another stopping time. For notational simplicity, denote $$\psi^{\delta}(\btau, \xi) := \psi^{\varepsilon, \eta, p}(\btau, \xi) + \delta.$$
	
	Under the event $A$,
	\begin{align}
		\one_{A} v^\delta_K(\btau, \xi) & = \one_A \psi^{\delta}(\btau, \xi) \nonumber \\
		& \leq \E [\one_A \psi^{\delta} (\tau^1 \wedge \rho, X((\tau^1 \wedge \rho) -;\btau, \xi, \theta_K, \emptyset)) | \cF_{\btau}] \nonumber \\
		& = \E [\one_{A \cap \{\rho < \tau^1 \}} \psi^{\delta} (\rho, X(\rho;\btau, \xi, \theta_K, \emptyset)) | \cF_{\btau}] + \E [\one_{A \cap \{ \rho \geq \tau^1\}} \psi^{\delta} (\tau^1, \xi^{1-}) | \cF_{\btau}]. \label{eq:A_psi}
	\end{align}
	The first line follows from the definition of $A$. The second line is from applying It\^o's formula to $\psi^{\delta}(t, X)$ from $\btau$ to $(\tau^1 \wedge \rho) -$, together with \eqref{eq:L_psi}. The third line uses the definition of $\xi^{1-}$.
	
	When the event $A \cap \{\rho < \tau^1 \}$ happens, we have $(\rho, X(\rho;\btau, \xi, \theta_K, \emptyset)) \in \cD(T_K, \bar{x}, \varepsilon/2)$. By the definition of $v^\delta_K$, 
	\begin{equation}\label{eq:A1}
		\one_{A \cap \{\rho < \tau^1 \}} \psi^{\delta} (\rho, X(\rho;\btau, \xi, \theta_K, \emptyset)) \leq \one_{A \cap \{\rho < \tau^1 \}} v^{\delta}_K (\rho, X(\rho;\btau, \xi, \theta_K, \emptyset)).
	\end{equation}
	
	For the second term in \eqref{eq:A_psi}, we separate it into two cases. When $\tau^1 = T_K$, \eqref{eq:psi} leads to
	\begin{equation}\label{eq:tau1TK1}
		\one_{A \cap \{\rho \geq \tau^1\} \cap \{ \tau^1 = T_K\}}  \psi^\delta(\tau^1, \xi^{1-}) \leq \one_{A \cap \{\rho \geq \tau^1\} \cap \{ \tau^1 = T_K\}}  w_K (G_K - \theta_K)^+.
	\end{equation}
	If $\tau^1 < T_K$, then
	\begin{align}
		&  \one_{A \cap \{\rho \geq \tau^1\} \cap \{ \tau^1 < T_K\}} \psi^\delta (\tau^1, \xi^{1-}) \nonumber \\
		& \leq \one_{A \cap \{\rho \geq \tau^1\} \cap \{ \tau^1 < T_K\}} v_{K,-} (\tau^1, \xi^{1-}) \nonumber \\
		& \leq \one_{A \cap \{\rho \geq \tau^1\} \cap \{ \tau^1 < T_K\}} \E[ \one_{\{\rho < T_K\}} v_{K,-}(\rho, X(\rho; \tau^1, \xi^{1-}, \theta_K, \emptyset)) + \one_{\{\rho = T_K\}} w_K(G_K - \theta_K)^+ | \cF_{\tau^1}] \nonumber \\
		& = \one_{A \cap \{\rho \geq \tau^1\} \cap \{ \tau^1 < T_K\}} \E[ \one_{\{\rho < T_K\}} v_{K,-}(\rho, X(\rho; \btau, \xi, \theta_K, \emptyset)) + \one_{\{\rho = T_K\}} w_K(G_K - \theta_K)^+ | \cF_{\tau^1}] \nonumber \\
		& \leq \one_{A \cap \{\rho \geq \tau^1\} \cap \{ \tau^1 < T_K\}} \E[ \one_{\{\rho < T_K\}} v^\delta_K(\rho, X(\rho; \btau, \xi, \theta_K, \emptyset)) + \one_{\{\rho = T_K\}} w_K(G_K - \theta_K)^+ | \cF_{\tau^1}]. \label{eq:tau1TK2}
	\end{align}
	The first inequality uses $\xi^{1-} \in \partial B(\bar{x}, \varepsilon/2)$ when $\tau^1 < T_K$. The second inequality follows from the submartingale property of $v_{K,-}$, with the random initial condition $(\tau^1, \xi^{1-})$. The equality uses the definition of $X(\cdot; \btau, \xi, \theta_K, \emptyset)$. The last inequality is from the fact that $v^\delta_K \geq v_{K,-}$ everywhere.
	
	Combining \eqref{eq:tau1TK1} and \eqref{eq:tau1TK2} and taking expectation conditional on $\cF_{\btau}$, we have
	\begin{align}
		& \E [\one_{A \cap \{ \rho \geq \tau^1\}} \psi^{\delta} (\tau^1, \xi^{1-}) | \cF_{\btau}] \label{eq:A2}\\
		& \leq \one_{A} \E[ \one_{\{\rho \geq \tau^1\} \cap \{\rho < T_K\}} v^\delta_K(\rho, X(\rho; \btau, \xi, \theta_K, \emptyset)) + \one_{\{\rho \geq \tau^1\} \cap \{\rho = T_K\}} w_K(G_K - \theta_K)^+ | \cF_{\btau}]. \nonumber
	\end{align}
	
	With \eqref{eq:A1} and \eqref{eq:A2}, \eqref{eq:A_psi} reduces to
	\begin{align}
		& \one_{A} v^\delta_K(\btau, \xi)  \nonumber \\
		& \leq \one_{A} \E[ \one_{\{\rho < T_K\}} v^\delta_K(\rho, X(\rho; \btau, \xi, \theta_K, \emptyset)) + \one_{\{\rho = T_K\}} w_K(G_K - \theta_K)^+ | \cF_{\btau}]. \label{eq:A_TK}
	\end{align}	 
	
	Under the event $A^c$, we use the definition of $A$, the submartingale property of $v_{K,-}$, and $v_{K,-} \leq v^\delta_K$ everywhere, to derive
	\begin{align}
		& \one_{A^c} v^\delta_K(\btau, \xi) \nonumber \\
		& = \one_{A^c} v_{K,-}(\btau, \xi) \nonumber \\
		& \leq \one_{A^c} \E[ \one_{\{\rho < T_K\}} v_{K,-}(\rho, X(\rho; \btau, \xi, \theta_K, \emptyset)) + \one_{\{\rho = T_K\}} w_K(G_K - \theta_K)^+ | \cF_{\btau}] \nonumber \\
		& \leq \one_{A^c} \E[ \one_{\{\rho < T_K\}} v^\delta_K(\rho, X(\rho; \btau, \xi, \theta_K, \emptyset)) + \one_{\{\rho = T_K\}} w_K(G_K - \theta_K)^+ | \cF_{\btau}]. \label{eq:Ac_TK}
	\end{align}
	Putting \eqref{eq:A_TK} and \eqref{eq:Ac_TK} together, we obtain Condition (4) as desired. 
	
	Hence, $(v_{1,-}, \ldots, v_{K-1, -}, v^\delta_K)$ is a stochastic subsolution under Definition \ref{def:sto_sub}. However, $$v^\delta_{K}(T_K, \bar{x}) = v_{K,-}(T_K, \bar{x}) + \delta > v_{K,-}(T_K, \bar{x}), $$ which contradicts with the definition of $v_{K,-}$ as a supremum.
\end{proof}

\begin{lemma}\label{lem:k_vissup}
	The lower stochastic envelope $v_{-}$ satisfies the viscosity supersolution property \eqref{k_vissup} at $T_k$, $k = 1, \ldots, K-1$, under Definition \ref{def:vis_super}.
\end{lemma}

\begin{proof}
	Note that we also have $v_{k, -, *} = v_{k,-}$, $v_{k+1, -, *} = v_{k+1, -}$, and $\cM[v_{k,-}]_* = \cM[v_{k,-}]$. Assume on the contrary that there exists $\bar{x} := (\bar{x}_0, \bar{x}_1) \in \cS$, such that 
	\begin{equation*}
		\begin{aligned}
			\max \Big\{ & v_{k,-}(T_k, \bar{x}) - \inf_{0 \leq \theta_k \leq \bar{x}_0} \left[ w_k (G_k - \theta_k)^+ + v_{k+1,-}(T_k, \bar{x}_0 - \theta_k, \bar{x}_1) \right] , \\
			& v_{k,-}(T_k, \bar{x}) - \cM[v_{k,-}](T_k, \bar{x}) \Big\} < 0.
		\end{aligned}
	\end{equation*}
	
	Since $v_{k+1,-}$ is LSC and the correspondence $(x_0, x_1) \twoheadmapsto \{ \theta_k | 0 \leq \theta_k \leq x_0 \}$ is upper hemicontinuous, we obtain that
	\begin{equation*}
		(x_0, x_1) \mapsto \inf_{0 \leq \theta_k \leq x_0} \left[ w_k (G_k - \theta_k)^+ + v_{k+1,-}(T_k, x_0 - \theta_k, x_1) \right]
	\end{equation*}
	is LSC by \citet[Lemma 17.30]{aliprantis2006infinite}.
	
	For $\varepsilon > 0$ small enough, with a slight abuse of notation, we define several sets for later use:
	\begin{equation*}
		\begin{aligned}
			B(\bar{x}, \varepsilon) & := \{ x | x \in \barS \text{ and } |x - \bar{x}| < \varepsilon \}, \\
			\cD(T_k, \bar{x}, \varepsilon) & := (T_k - \varepsilon, T_k] \times B(\bar{x}, \varepsilon), \\
			E(\varepsilon) &:= \overline{\cD(T_k, \bar{x}, \varepsilon)} \setminus \cD(T_k, \bar{x}, \varepsilon/2).
		\end{aligned}
	\end{equation*}
	
	By the LSC property, there exists $\varepsilon > 0$ small enough, such that 
	\begin{equation}\label{k:sub_eps}
		\begin{aligned}
			& v_{k,-}(T_k, \bar{x}) + \varepsilon \leq \inf_{0 \leq \theta_k \leq x_0} \left[ w_k (G_k - \theta_k)^+ + v_{k+1,-}(T_k, x_0 - \theta_k, x_1) \right], \\
			& v_{k,-}(T_k, \bar{x}) + \varepsilon \leq \cM[v_{k,-}](t, x),
		\end{aligned}
	\end{equation}
	when $(t, x) \in \overline{\cD(T_k, \bar{x}, \varepsilon)}$.

	Similarly, with a large enough $p > 0$ and small enough $\eta > 0$, we can define
	\begin{equation}\label{def:psi}
		\psi^{\varepsilon, \eta, p}(t, x) := v_{k,-}(T_k, \bar{x}) - \frac{|x - \bar{x}|^2}{\eta} - p(T_k - t).
	\end{equation} 
	It satisfies the following properties:
	\begin{itemize}
		\item For $(t, x) \in \overline{\cD(T_k, \bar{x}, \varepsilon)}$,
		\begin{equation}\label{k:L_psi}
			\cL[\psi^{\varepsilon, \eta, p}] (t, x) < 0. 
		\end{equation}
		\item When $(t, x) \in E(\varepsilon)$,
		\begin{align}\label{k:Eeps}
			\psi^{\varepsilon, \eta, p}(t, x) < v_{k,-}(T_k, \bar{x}) - \frac{\varepsilon^2}{4 \eta} < \inf_{(t, x) \in E(\varepsilon)} v_{k,-}(t, x) - \varepsilon \leq v_{k,-}(t, x) - \varepsilon.
		\end{align}
		\item 	Besides, \eqref{k:sub_eps} and \eqref{def:psi} imply that
		\begin{equation}\label{k:psi}
			\begin{aligned}
				\psi^{\varepsilon, \eta, p}(t, x) & \leq \inf_{0 \leq \theta_k \leq x_0} \left[ w_k (G_k - \theta_k)^+ + v_{k+1,-}(T_k, x_0 - \theta_k, x_1) \right]  - \varepsilon, \\
				\psi^{\varepsilon, \eta, p}(t, x) & \leq \cM[v_{k,-}](t, x) - \varepsilon,
			\end{aligned}
		\end{equation}
		when $(t, x) \in \overline{\cD(T_k, \bar{x}, \varepsilon)}$.
	\end{itemize}
	
	Let $0 < \delta < \varepsilon$ be small enough and define
	\begin{equation}\label{k:hdelta}
		v^{\delta}_{k} (t, x) = \left\{ 
		\begin{array}{cl}
			v_{k,-} (t, x) \vee ( \psi^{\varepsilon, \eta, p}(t, x) + \delta) & \text{on } \overline{\cD(T_k, \bar{x}, \varepsilon)}, \\
			v_{k,-} (t, x), & \text{otherwise}.
		\end{array}
		\right.
	\end{equation}
	
	We show that $(v_{1,-}, \ldots, v_{k-1, -},  v^\delta_k, v_{k+1,-}, \ldots, v_{K,-})$ is a stochastic subsolution under Definition \ref{def:sto_sub}. Conditions (1), (2), (3) can be verified similarly as before.

	For Condition (4), since $v^\delta_k \geq v_{k,-}$ everywhere, we only need to prove it for $\btau \in [T_{k-1}, T_k]$. Consider any random initial condition $(\btau, \xi)$ with $\xi \in \cF_{\btau}$ and $\p(\xi \in \barS) = 1$ and any $(\btau, \xi)$-admissible withdrawals $\theta_{k:K}$. Define the event 
	\begin{equation*}
		A := \{(\btau, \xi) \in \cD(T_k, \bar{x}, \varepsilon/2) \} \cap \{ \psi^{\varepsilon, \eta, p}(\btau, \xi) + \delta > v_{k,-}(\btau, \xi)\}.
	\end{equation*}
	Then $A \in \cF_{\btau}$. Let 
	\begin{align*}
		\tau^1 := \inf \big\{ t \in [\btau, T_k] \, \big| &  (t,  X(t; \btau, \xi, \theta_{k:K}, \emptyset)) \notin \cD(T_k, \bar{x}, \varepsilon/2) \big\} \wedge T_k
	\end{align*}
	be the exit time. Since it is possible that $\tau^1 = T_k$, we introduce  
	\begin{align*}
		\xi^{1-} := (\xi^{1-}_0, \xi^{1-}_1) :=  X(\tau^1-; \btau, \xi, \theta_{k:K}, \emptyset)
	\end{align*}
	as the position that excludes a possible jump at $\tau^1$. 
	
	Let $\rho \in [\btau, T]$ be another stopping time. For notational simplicity, denote $$\psi^{\delta}(\btau, \xi) := \psi^{\varepsilon, \eta, p}(\btau, \xi) + \delta. $$
	
	Under the event $A$,
	\begin{align}
		\one_{A} v^\delta_k(\btau, \xi) & = \one_A \psi^{\delta}(\btau, \xi) \nonumber \\
		& \leq \E [\one_A \psi^{\delta} (\tau^1 \wedge \rho, X((\tau^1 \wedge \rho) -;\btau, \xi, \theta_{k:K}, \emptyset)) | \cF_{\btau}] \nonumber \\
		& = \E [\one_{A \cap \{\rho < \tau^1 \}} \psi^{\delta} (\rho, X(\rho;\btau, \xi, \theta_{k:K}, \emptyset)) | \cF_{\btau}] + \E [\one_{A \cap \{ \rho \geq \tau^1\}} \psi^{\delta} (\tau^1, \xi^{1-}) | \cF_{\btau}]. \label{k:A_psi}
	\end{align}
	The inequality is from applying It\^o's formula to $\psi^{\delta}(t, X)$ from $\btau$ to $(\tau^1 \wedge \rho) -$, together with \eqref{k:L_psi}. 
	
	When the event $A \cap \{\rho < \tau^1 \}$ happens, we have $(\rho, X(\rho;\btau, \xi, \theta_{k:K}, \emptyset)) \in \cD(T_K, \bar{x}, \varepsilon/2)$. By the definition of $v^\delta_k$, 
	\begin{equation}\label{k:A1}
		\one_{A \cap \{\rho < \tau^1 \}} \psi^{\delta} (\rho, X(\rho;\btau, \xi, \theta_{k:K}, \emptyset)) \leq \one_{A \cap \{\rho < \tau^1 \}} v^{\delta}_k (\rho, X(\rho;\btau, \xi, \theta_{k:K}, \emptyset)).
	\end{equation}
	
	For the second term in \eqref{k:A_psi}, we introduce two events:
	\begin{align*}
		Q_1 & := \{ \tau^1 < T_k \} \cup \{ \tau^1 = T_k \text{ and } \xi^{1-} \notin B(\bar{x}, \varepsilon/2) \}, \\
		Q_2 & := \{ \tau^1 = T_k \text{ and } \xi^{1-} \in B(\bar{x}, \varepsilon/2) \}.
		%	 \cap \{ \psi^\delta(\tau^1, \xi^{1-}) > v_{k,-}(\tau^1, \xi^{1-}) \}, \\
		%	Q_3 & := \{ \tau^1 = T_k \text{ and } \xi^{1-} \in \B(\bar{x}, \varepsilon/2) \} \cap \{ \psi^\delta(\tau^1, \xi^{1-}) \leq v_{k,-}(\tau^1, \xi^{1-}) \}.
	\end{align*}
	
	For $Q_1$, we have
	\begin{align*}
		& \one_{A \cap \{\rho \geq \tau^1\} \cap Q_1} \psi^\delta (\tau^1, \xi^{1-}) \\
		& \leq \one_{A \cap \{\rho \geq \tau^1\} \cap Q_1} v_{k,-}(\tau^1, \xi^{1-}) \\
		& \leq \one_{A \cap \{\rho \geq \tau^1\} \cap Q_1} \E[  \cH([\tau^1, \rho], v_{k:K,-}, X(\cdot; \tau^1, \xi^{1-}, \theta_{k:K}, \emptyset))  | \cF_{\tau^1}] \\
		& \leq \one_{A \cap \{\rho \geq \tau^1\} \cap Q_1} \E[  \cH([\tau^1, \rho], (v^\delta_k, v_{k+1:K,-}), X(\cdot; \tau^1, \xi^{1-}, \theta_{k:K}, \emptyset))  | \cF_{\tau^1}] \\
		& = \one_{A \cap \{\rho \geq \tau^1\} \cap Q_1} \E[  \cH([\tau^1, \rho], (v^\delta_k, v_{k+1:K,-}), X(\cdot; \btau, \xi, \theta_{k:K}, \emptyset))  | \cF_{\tau^1}].
	\end{align*}
	The first inequality follows from \eqref{k:Eeps} and $(\tau^1, \xi^{1-}) \in E(\varepsilon)$. The second inequality uses the submartingale property of $v_{k,-}$, with the random initial condition $(\tau^1, \xi^{1-})$. The third inequality is due to $v_{k,-} \leq v^\delta_k$ everywhere. The last equality is from the definition of $X(\cdot; \btau, \xi, \theta_{k:K}, \emptyset)$.
	
	For $Q_2$, we obtain
	\begin{align*}
		& \one_{A \cap \{\rho \geq \tau^1\} \cap Q_2} \psi^\delta (\tau^1, \xi^{1-}) \\
		& = \one_{A \cap \{\rho \geq \tau^1\} \cap Q_2} \psi^\delta (T_k, \xi^{1-}) \\
		& \leq \one_{A \cap \{\rho \geq \tau^1\} \cap Q_2} \E[  w_k(G_k - \theta_k)^+ + v_{k+1,-}(T_k, \xi^{1-}_0 - \theta_k, \xi^{1-}_1) | \cF_{\tau^1}] \\
		& = \one_{A \cap \{\rho \geq \tau^1\} \cap Q_2} \E[  w_k(G_k - \theta_k)^+ + v_{k+1,-}(T_k, \xi^{1}) | \cF_{\tau^1}] \\
		& \leq \one_{A \cap \{\rho \geq \tau^1\} \cap Q_2} \E[  w_k(G_k - \theta_k)^+ + \cH([T_k, \rho], v_{k+1:K,-}, X(\cdot; T_k, \xi^{1}, \theta_{k+1:K}, \emptyset)) | \cF_{\tau^1}] \\
		& = \one_{A \cap \{\rho \geq \tau^1\} \cap Q_2} \E[  w_k(G_k - \theta_k)^+ + \cH([T_k, \rho], v_{k+1:K,-}, X(\cdot; \btau, \xi, \theta_{k:K}, \emptyset)) | \cF_{\tau^1}].
	\end{align*}
	The first equality uses $\tau^1 = T_k$. The first inequality follows from \eqref{k:psi} and the fact that $(T_k, \xi^{1-}) \in \overline{\cD(T_k, \bar{x}, \varepsilon)}$. The second equality holds due to the definition of $\xi^{1-}$ and $\xi^1$. The last two lines use the submartingale property of $v_{k+1,-}$, with the random initial condition $(T_k, \xi^1)$ and the definition of $X(\cdot; \btau, \xi, \theta_{k:K}, \emptyset)$.

	Under the event $A^c$, it is direct to show
	\begin{align*}
		\one_{A^c} v^\delta_k(\btau, \xi)  & = \one_{A^c} v_{k,-}(\btau, \xi) \\
		& \leq \one_{A^c} \E[ \cH([\btau, \rho], v_{k:K,-}, X(\cdot; \btau, \xi, \theta_{k:K}, \emptyset)) | \cF_{\btau}] \\
		& \leq \one_{A^c} \E[ \cH([\btau, \rho], (v^\delta_k, v_{k+1:K,-}), X(\cdot; \btau, \xi, \theta_{k:K}, \emptyset)) | \cF_{\btau}].
	\end{align*}
	Putting these inequalities together, we obtain the following Condition (4) as desired:
	\begin{align*}
		v^\delta_k(\btau, \xi) \leq \E[ \cH([\btau, \rho], (v^\delta_k, v_{k+1:K,-}), X(\cdot; \btau, \xi, \theta_{k:K}, \emptyset)) | \cF_{\btau}].
	\end{align*}

	Hence, $(v_{1,-}, \ldots, v_{k-1,-},  v^\delta_k, v_{k+1,-}, \ldots, v_{K,-})$ is a stochastic subsolution under Definition \ref{def:sto_sub}. However, $v^\delta_{k}(T_k, \bar{x}) = v_{k,-}(T_k, \bar{x}) + \delta > v_{k,-}(T_k, \bar{x})$, which contradicts with the definition of $v_{k,-}$ as a supremum.
\end{proof}

\section{Proofs of the comparison principle}
\begin{proof}[Proof of Proposition \ref{prop:compareTk}]
	Choose $q \in (p_0, 1)$ in $F^1_k(t, x)$. Moreover, we replace the constant 2 in $C_k$ by a sufficiently large constant specified later. For any $\eta > 1$, define
	\begin{equation*}
		u_\eta(T_k, x) := \frac{\eta+1}{\eta} u(T_k, x) + \frac{1}{\eta} F^1_k(T_k, x), \quad 	v_\eta(T_k, x) := \frac{\eta-1}{\eta} v(T_k, x) - \frac{1}{\eta} F^1_k(T_k, x).
	\end{equation*} 
	The idea is to show $u_\eta(T_k, x) - v_\eta(T_k, x) \leq 0$ for all $\eta > 1$ and $x \in \barS$, which implies $u(T_k, x) - v(T_k, x) \leq 0$ when $\eta \rightarrow \infty$.
	
	Assume on the contrary that, there exist $x^* \in \barS$ and $\eta > 1$ such that 
	\begin{equation*}
		u_\eta(T_k, x^*) - v_\eta(T_k, x^*) > 0.
	\end{equation*}
	Then 
	\begin{equation*}
		C_\eta := \sup_{x \in \barS} \big\{ u_\eta(T_k, x) - v_\eta(T_k, x) \big\} > 0.
	\end{equation*}
	
	For each $n \geq 0$, define
	\begin{equation*}
		\Phi_n(x, x') := u_\eta(T_k, x) - v_\eta(T_k, x') - \frac{n}{2} |x - x'|^2, \quad x, x' \in \barS.
	\end{equation*}
	We note that
	\begin{equation}\label{ineq:Phi}
		\begin{aligned}
			0 & < u_\eta(T_k, x^*) - v_\eta(T_k, x^*) \leq \sup_{x \in \barS} \big\{ u_\eta(T_k, x) - v_\eta(T_k, x) \big\} \\
			& \leq \sup_{x, x' \in \barS} \Phi_{n+1}(x, x') \leq \sup_{x, x' \in \barS} \Phi_{n}(x, x') \leq \sup_{x, x' \in \barS} \Phi_{0}(x, x').
		\end{aligned}
	\end{equation}
	Under the growth condition \eqref{eq:growth} and $q > p_0$ in $F^1_k(t, x)$, we have $\Phi_n(x, x') \rightarrow - \infty$ when $|(x, x')| \rightarrow + \infty$ in $\barS \times \barS$. Together with the USC property of $u_\eta - v_\eta$, then $\sup_{x, x' \in \barS} \Phi_{n}(x, x')$ is attained at some $(x_n , x'_n)$. The inequality \eqref{ineq:Phi} indicates that $\{(x_n, x'_n)\}^\infty_{n=1}$ is in the following set:
	\begin{equation}\label{set:superlevel}
		\big\{  (x, x') \in \barS \times \barS \; \big| \; u_\eta(T_k, x) - v_\eta(T_k, x') \geq 0 \big\}.
	\end{equation}
	The USC property of $u_\eta(T_k, x) - v_\eta(T_k, x')$ shows that the set \eqref{set:superlevel} is closed. Since $u_\eta(T_k, x) - v_\eta(T_k, x') \rightarrow - \infty$ when $|(x, x')| \rightarrow + \infty$, the set  \eqref{set:superlevel} is bounded. Therefore, the set \eqref{set:superlevel} is compact. Up to a subsequence, we can assume that $\{(x_n, x'_n)\}^\infty_{n=1}$ is convergent. Then \eqref{ineq:Phi} yields
	\begin{equation*}
		0  < u_\eta(T_k, x^*) - v_\eta(T_k, x^*) \leq u_\eta(T_k, x_n) - v_\eta(T_k, x'_n) - \frac{n}{2} |x_n - x'_n|^2,
	\end{equation*}
	which means that
	\begin{equation*}
		u_\eta(T_k, x_n) - v_\eta(T_k, x'_n) - \{u_\eta(T_k, x^*) - v_\eta(T_k, x^*)\} \geq \frac{n}{2} |x_n - x'_n|^2.
	\end{equation*}
	When $n \rightarrow \infty$, the left-hand side is bounded because of the USC property. Then we must have
	\begin{equation*}
		\lim_{n \rightarrow \infty} |x_n - x'_n|^2 = 0.
	\end{equation*}
	Hence, there exists $\bar{x} \in \barS$ and
	\begin{equation}
		\lim_{n \rightarrow \infty} x_n = \lim_{n \rightarrow \infty} x'_n = \bar{x}.
	\end{equation}
	By definition, we have
	\begin{equation*}
		\sup_{x, x' \in \barS} \Phi_n(x, x') = u_\eta(T_k, x_n) - v_\eta(T_k, x'_n) - \frac{n}{2} |x_n - x'_n|^2.
	\end{equation*}
	Then
	\begin{align*}
		0 & \leq \limsup_{n \rightarrow \infty} \frac{n}{2} |x_n - x'_n|^2 = \limsup_{n \rightarrow \infty} \Big\{  u_\eta(T_k, x_n) - v_\eta(T_k, x'_n) -   \sup_{x, x' \in \barS} \Phi_n(x, x') \Big\} \\
		& \leq \limsup_{n \rightarrow \infty} \Big\{ u_\eta(T_k, x_n) - v_\eta(T_k, x'_n) \Big\} + \limsup_{n \rightarrow \infty} \Big\{ -   \sup_{x, x' \in \barS} \Phi_n(x, x') \Big\} \\
		& \leq u_\eta(T_k, \bar{x}) - v_\eta(T_k, \bar{x}) - \sup_{x \in \barS} \big\{ u_\eta(T_k, x) - v_\eta(T_k, x) \big\} \\
		& \leq 0.
	\end{align*}
	Here, we use the USC property and \eqref{ineq:Phi} in the second to last inequality. Hence, all the inequalities should be equalities:
	\begin{equation}\label{supbar}
		\lim_{n \rightarrow \infty} \frac{n}{2} |x_n - x'_n|^2 = 0 \quad \text{ and }  \quad u_\eta(T_k, \bar{x}) - v_\eta(T_k, \bar{x}) = \sup_{x \in \barS} \big\{ u_\eta(T_k, x) - v_\eta(T_k, x) \big\}.
	\end{equation}
	It also implies that, up to another subsequence (still indexed with $n$),
	\begin{equation}\label{eq:uxn_conv}
		\begin{aligned}
			u_\eta(T_k, \bar{x}) & = \limsup_{n \rightarrow \infty} u_\eta(T_k, x_n) = \lim_{n \rightarrow \infty} u_\eta(T_k, x_n), \\ 
			v_\eta(T_k, \bar{x}) & = \liminf_{n \rightarrow \infty} v_\eta(T_k, x'_n) = \lim_{n \rightarrow \infty} v_\eta(T_k, x'_n).
		\end{aligned}
	\end{equation}
	
	We claim that $\bar{x} \neq 0$. In fact,
	\begin{align*}
		u_\eta(T_k, 0) - v_\eta(T_k, 0) = u(T_k, 0) - v(T_k, 0) + \frac{1}{\eta} \Big( u(T_k, 0) + v(T_k, 0) + 2 F^1_k(T_k, 0) \Big).
	\end{align*} 
	The assumption \eqref{cond:0} ensures that, when the constant $2$ in $C_k$ from $F^1_k(t, x)$ is replaced by a sufficiently large constant, we have
	\begin{align*}
		u(T_k, 0) - v(T_k, 0) & \leq 0, \\
		u(T_k, 0) + v(T_k, 0) + 2 F^1_k(T_k, 0) & \leq 4 \sum^K_{i = k} w_i G_i - 2 C_k < 0.
	\end{align*}
	Then $u_\eta(T_k, 0) - v_\eta(T_k, 0) < 0$, which implies that $\bar{x} \neq 0$. Hence, we can assume that $x_n \neq 0$ and $x'_n \neq 0$ when $n$ is large enough.
	
	By \citet[Proposition 4.2]{belak2019utility} and \eqref{eq:subTk} in Lemma \ref{lem:classical_sub}, we have
	\begin{equation}\label{sub_kappa}
		\begin{aligned}
			\max \Big\{ & u_\eta(T_k, x_n) - \inf_{0 \leq \theta_k \leq x_{n, 0}} \left[ w_k (G_k - \theta_k)^+ + f(T_k, x_{n, 0} - \theta_k, x_{n, 1}) \right], \\
			& u_\eta(T_k, x_n) - \cM[u_\eta]^* (T_k, x_n) \Big\} \leq - \frac{\bar{\kappa}}{\eta},
		\end{aligned}
	\end{equation}
	and
	\begin{equation}\label{sup_kappa}
		\begin{aligned}
			\max \Big\{ & v_\eta(T_k, x'_n) - \inf_{0 \leq \theta_k \leq x'_{n, 0}} \left[ w_k (G_k - \theta_k)^+ + f(T_k, x'_{n, 0} - \theta_k, x'_{n, 1}) \right], \\
			& v_\eta(T_k, x'_n) - \cM[v_\eta]_* (T_k, x'_n) \Big\} \geq \frac{\bar{\kappa}}{\eta}.
		\end{aligned}
	\end{equation}
	Here, $\bar{\kappa} := \inf_{n} \min\{\kappa^b_k(x_n), \kappa^b_k(x'_n) \} > 0$, where $\kappa^b_k(\cdot)$ is defined in \eqref{eq:subTk}. 
	
	Suppose $v_\eta(T_k, x'_n) - \inf_{0 \leq \theta_k \leq x'_{n, 0}} \left[ w_k (G_k - \theta_k)^+ + f(T_k, x'_{n, 0} - \theta_k, x'_{n, 1}) \right] \geq \bar{\kappa}/\eta$ does not hold for infinitely many $n$. Then there exists $N$ large enough, such that
	\begin{equation}\label{eq3}
		v_\eta(T_k, x'_n) - \cM[v_\eta]_* (T_k, x'_n) \geq \frac{\bar{\kappa}}{\eta}, \quad n \geq N.
	\end{equation}
	We proceed to obtain a contradiction. In the following steps, the threshold $N$ may vary line by line. First, by the definition of $\cM[\cdot]$, \eqref{eq3} implies that $x'_n \notin \cS_\emptyset$. Since $\cS_\emptyset$ is open, it further implies that $\bar{x} \notin \cS_\emptyset$. 
	
	By the convergence result in \eqref{eq:uxn_conv},
	\begin{equation}\label{eq2} 
		u_\eta(T_k, \bar{x}) - v_\eta(T_k, \bar{x}) \leq u_\eta(T_k, x_n) - v_\eta(T_k, x'_n) + \frac{\bar{\kappa}}{4\eta}, \quad n \geq N.
	\end{equation}
	The LSC property of $\cM[v_\eta]_*$ on $\barS$ leads to
	\begin{equation}\label{eq4}
		\cM[v_\eta]_*(T_k, x'_n) \geq \cM[v_\eta]_*(T_k, \bar{x}) - \frac{\bar{\kappa}}{4\eta}, \quad n \geq N.
	\end{equation}
	Besides, since $v_\eta$ is LSC, Lemma \ref{lem:Mproperty} proves that
	\begin{equation}\label{eq5} 
		\cM[v_\eta]_*(T_k, x'_n) = \cM[v_\eta](T_k, x'_n) \quad \text{ and } \quad \cM[v_\eta]_*(T_k, \bar{x}) = \cM[v_\eta](T_k, \bar{x}).
	\end{equation}
	As $\bar{x} \notin \cS_\emptyset$, the LSC property of $v_\eta$ ensures the existence of an optimizer $\Delta \in D(\bar{x})$, such that 
	\begin{equation}\label{eq6}
		\cM[v_\eta](T_k, \bar{x}) = v_\eta(T_k, \Gamma(\bar{x}, \Delta)).
	\end{equation}
	
	Putting these results together, we have
	\begin{align*}
		u_\eta(T_k, \bar{x}) - v_\eta(T_k, \bar{x}) \leq & u_\eta(T_k, x_n) - v_\eta(T_k, x'_n) + \frac{\bar{\kappa}}{4\eta}  & \text{(by \eqref{eq2})} \\
		\leq & u_\eta(T_k, x_n) - \cM[v_\eta]_* (T_k, x'_n) - \frac{\bar{\kappa}}{\eta} + \frac{\bar{\kappa}}{4\eta} & \text{(by \eqref{eq3})} \\
		\leq & u_\eta(T_k, x_n) - \cM[v_\eta]_*(T_k, \bar{x}) + \frac{\bar{\kappa}}{4\eta} - \frac{\bar{\kappa}}{\eta} + \frac{\bar{\kappa}}{4\eta} & \text{(by \eqref{eq4})} \\
		= & u_\eta(T_k, x_n) - \cM[v_\eta](T_k, \bar{x}) - \frac{\bar{\kappa}}{2\eta} & \text{(by \eqref{eq5})} \\
		= & u_\eta(T_k, x_n) - v_\eta(T_k, \Gamma(\bar{x}, \Delta)) - \frac{\bar{\kappa}}{2\eta}, \quad n \geq N. & \text{(by \eqref{eq6})} 
	\end{align*}
	
	Next, we show that $\bar{x} \notin \overline{\cS_\emptyset}\setminus \cS_\emptyset$. Indeed, if not, then $\Gamma(\bar{x}, \Delta) = 0$ and
	\begin{align*}
		& u_\eta(T_k, x_n) - v_\eta(T_k, \Gamma(\bar{x}, \Delta)) \\
		& \quad = u_\eta(T_k, x_n) - v_\eta(T_k, 0) \\
		& \quad = u(T_k, x_n) - v(T_k, 0) + \frac{1}{\eta} \big\{ u(T_k, x_n) + v(T_k, 0) + F^1_k(T_k, x_n)  + F^1_k(T_k, 0) \big\}.
	\end{align*} 
	The assumptions \eqref{cond:0} and \eqref{eq:growth} guarantee that $u(T_k, x_n) \leq v(T_k, 0)$. Moreover, 
	\begin{equation*}
		u(T_k, x_n) + v(T_k, 0) + F^1_k(T_k, x_n)  + F^1_k(T_k, 0) \leq 4 \sum^K_{i = k} w_i G_i - 2 C_k < 0.
	\end{equation*}
	Then it leads to $u_\eta(T_k, x_n) - v_\eta(T_k, \Gamma(\bar{x}, \Delta)) < 0$, which contradicts with the previous inequality.
	
	We simplify $u_\eta(T_k, x_n)$ as follows:
	\begin{itemize}
		\item \eqref{sub_kappa} shows that 
		\begin{equation}\label{eq7}
			u_\eta(T_k, x_n) - \cM[u_\eta]^* (T_k, x_n) \leq - \bar{\kappa}/\eta.
		\end{equation}
		\item Since $\bar{x} \notin \overline{\cS_\emptyset}$, Lemma \ref{lem:Mproperty} proves that
		\begin{equation}
			\cM[u_\eta]^* (T_k, \bar{x}) = \cM[u_\eta](T_k, \bar{x}).
		\end{equation}
		Moreover, since $\Gamma(\bar{x}, \Delta)$ is feasible, 
		\begin{equation}\label{eq9}
			\cM[u_\eta](T_k, \bar{x}) \leq u_\eta(T_k, \Gamma(\bar{x}, \Delta)).
		\end{equation} 
		\item The USC property of $\cM[u_\eta]^*$ yields
		\begin{equation}\label{eq8}
			\cM[u_\eta]^*(T_k, x_n) \leq \cM[u_\eta]^*(T_k, \bar{x}) + \frac{\bar{\kappa}}{2 \eta}, \quad n \geq N.
		\end{equation}
	\end{itemize}
	Hence,
	\begin{align*}
		0 < & u_\eta(T_k, \bar{x}) - v_\eta(T_k, \bar{x})   & \text{(by \eqref{supbar}) }\\
		\leq & u_\eta(T_k, x_n) - v_\eta(T_k, \Gamma(\bar{x}, \Delta)) - \frac{\bar{\kappa}}{2\eta} \\
		\leq & \cM[u_\eta]^* (T_k, x_n) - \frac{\bar{\kappa}}{\eta} - v_\eta(T_k, \Gamma(\bar{x}, \Delta)) - \frac{\bar{\kappa}}{2\eta} & \text{(by \eqref{eq7})} \\
		\leq & \cM[u_\eta]^*(T_k, \bar{x}) + \frac{\bar{\kappa}}{2 \eta} - \frac{\bar{\kappa}}{\eta} - v_\eta(T_k, \Gamma(\bar{x}, \Delta)) - \frac{\bar{\kappa}}{2\eta} & \text{(by \eqref{eq8})} \\
		\leq & u_\eta(T_k, \Gamma(\bar{x}, \Delta)) - v_\eta(T_k, \Gamma(\bar{x}, \Delta)) - \frac{\bar{\kappa}}{\eta}  & \text{(by \eqref{eq9})} \\
		\leq & u_\eta(T_k, \bar{x}) - v_\eta(T_k, \bar{x})  - \frac{\bar{\kappa}}{\eta},
	\end{align*}
	which is a contradiction. Therefore, we must have $$ v_\eta(T_k, x'_n) - \inf_{0 \leq \theta_k \leq x'_{n, 0}} \left[ w_k (G_k - \theta_k)^+ + f(T_k, x'_{n, 0} - \theta_k, x'_{n, 1}) \right] \geq \bar{\kappa}/\eta$$ for infinitely many $n$. Up to another subsequence, \eqref{sub_kappa} leads to
	\begin{equation}\label{kappa0kappa}
		\begin{aligned}
			& v_\eta(T_k, x'_n) - \inf_{0 \leq \theta_k \leq x'_{n, 0}} \left[ w_k (G_k - \theta_k)^+ + f(T_k, x'_{n, 0} - \theta_k, x'_{n, 1}) \right] \\
			& \geq \frac{\bar{\kappa}}{\eta} > 0 > -\frac{\bar{\kappa}}{\eta} \geq u_\eta(T_k, x_n) - \inf_{0 \leq \theta_k \leq x_{n, 0}} \left[ w_k (G_k - \theta_k)^+ + f(T_k, x_{n, 0} - \theta_k, x_{n, 1}) \right].
		\end{aligned}
	\end{equation}
	
	Since $f$ is continuous and bounded, 
	\begin{align*}
		x = (x_0, x_1) \mapsto \inf_{0 \leq \theta_k \leq x_{0}} \left[ w_k (G_k - \theta_k)^+ + f(T_k, x_{0} - \theta_k, x_{1}) \right]
	\end{align*}
	is a continuous function. 
	
	Letting $n \rightarrow \infty$ in \eqref{kappa0kappa}, we obtain $v_\eta(T_k, \bar{x}) > u_\eta(T_k, \bar{x})$, which is also a contradiction. Then the claim follows as desired.
\end{proof}

\begin{proof}[Proof of Proposition \ref{prop:compareTK}]
	Thanks to the strict classical subsolution property of $F^1_K(T_K, x)$ at $x \in \cS$, we can obtain
	\begin{equation*}
		\begin{aligned}
			& v_\eta(T_K, x'_n) - w_K \left[ G_K - x'_{n, 0} - (x'_{n, 1} - C(-x'_{n,1}))^+ \right]^+ \\
			& \geq \frac{\bar{\kappa}}{\eta} > 0 > -\frac{\bar{\kappa}}{\eta} \geq u_\eta(T_K, x_n) -  w_K \left[ G_K - x_{n, 0} - (x_{n, 1} - C(-x_{n,1}))^+ \right]^+,
		\end{aligned}
	\end{equation*}
	with the same proof procedure in Proposition \ref{prop:compareTk}.
	
	Letting $n \rightarrow \infty$, we have a contradiction as $v_\eta(T_K, \bar{x}) - u_\eta(T_K, \bar{x}) > 0$.
\end{proof}

\begin{proof}[Proof of Proposition \ref{prop:compare}]
	The claim follows directly from modifying the proof of Proposition \ref{prop:compareTk} and applying Ishii's lemma, which is similar to \citet[Theorem 5.4]{belak2022optimal}. We also note that the strict classical subsolution property of $F^1_k(t, x)$ in \eqref{eq:sub_t} is used to apply \citet[Proposition 4.2]{belak2019utility}. 
\end{proof}

\section{Proofs of optimal strategies}
\begin{proof}[Proof of Lemma \ref{lem:borel}]
	{\bf Step 1}. For any $(t, x) \in [T_{k-1}, T_k] \times \barS$, $k = 1, \ldots, K$, Condition (3) in Definition \ref{def:sto_sub} leads to
	\begin{align*}
		h_k(t, x) \leq \cM[h_k](t, x) \leq \cM[h^*_k](t, x) \leq  \cM[h^*_k]^*(t, x).
	\end{align*} 
	Since $\cM[h^*_k]^*$ is USC, taking $\limsup$ shows that $h^*_k(t, x) \leq \cM[h^*_k]^*(t, x)$, as required by the viscosity subsolution property.
	
	{\bf Step 2}. Fix $x \in \cS$ and $T_k$, $k = 1, \ldots, K-1$. Consider a sequence $(s_n, y_n) \rightarrow (T_k, x)$ where $s_n \leq T_k$, such that
	\begin{equation*}
		\lim_{n \rightarrow \infty} h_k(s_n, y_n) = h^*_k(T_k, x).
	\end{equation*} 
	Recall that $x_0$ is the wealth in the bank account. For any constant $\theta_k \in [0, x_0]$, define the (random) withdrawal for goal $k$ as 
	\begin{equation*}
		\Theta_n : = \min\{ \theta_k, X_0(T_k; s_n, y_n, \emptyset, \emptyset)\},
	\end{equation*}
	which is $\cF_{T_k}$-measurable. The submartingale property (4) of $h_k$ yields
	\begin{align*}
		& h^*_k(T_k, x) \\
		& =  \lim_{n \rightarrow \infty} h_k(s_n, y_n) \\
		& \leq \limsup_{n \rightarrow \infty} \E \Big[ w_k(G_k - \Theta_n)^+ + h_{k+1}(T_k, X_0(T_k; s_n, y_n, \emptyset, \emptyset)- \Theta_n, X_1(T_k; s_n, y_n, \emptyset, \emptyset)) \Big] \\
		& \leq \limsup_{n \rightarrow \infty} \E \Big[ w_k(G_k - \Theta_n)^+ + h^*_{k+1}(T_k, X_0(T_k; s_n, y_n, \emptyset, \emptyset)- \Theta_n, X_1(T_k; s_n, y_n, \emptyset, \emptyset)) \Big] \\
		& \leq  \E \Big[ \limsup_{n \rightarrow \infty} \Big( w_k(G_k - \Theta_n)^+ + h^*_{k+1}(T_k, X_0(T_k; s_n, y_n, \emptyset, \emptyset)- \Theta_n, X_1(T_k; s_n, y_n, \emptyset, \emptyset)) \Big) \Big] \\
		& \leq \E[ w_k(G_k - \theta_k)^+ +  h^*_{k+1}(T_k, x_0 - \theta_k, x_1) ] \\
		& = w_k(G_k - \theta_k)^+ +  h^*_{k+1}(T_k, x_0 - \theta_k, x_1).
	\end{align*}
	Here, the second line uses the submartingale property (4) from $s_n$ to $T_k$. Note that $\Theta_n$ is admissible. The third line follows from $h_{k+1} \leq h^*_{k+1}$. The fourth line is from Fatou's lemma and the fact that $h_{k+1}$ is bounded from above. The fifth line holds because $X$ has continuous paths, $h^*_{k+1}$ is USC, and $\lim_{n \rightarrow \infty} \Theta_n = \theta_k$. The last line holds since these terms are deterministic. As $\theta_k \in [0, x_0]$ is arbitrary, we obtain
	\begin{equation*}
		h^*_k(T_k, x) \leq \inf_{0 \leq \theta_k \leq x_0} \Big( w_k(G_k - \theta_k)^+ +  h^*_{k+1}(T_k, x_0 - \theta_k, x_1) \Big).
	\end{equation*}
	The case for $T_K$ follows similarly by replacing $\Theta_n$ with the liquidation value.
	
	{\bf Step 3}. Finally, fix $(t, x) \in [T_{k-1}, T_k) \times \cS$, $k=1, \ldots, K$. Consider $(s_n, y_n) \subset [T_{k-1}, T_k) \times \cS$, such that $(s_n, y_n) \rightarrow (t, x)$ when $n \rightarrow \infty$ and 
	\begin{equation}\label{eq:seqh*}
		\lim_{n \rightarrow \infty} h_k(s_n, y_n) = h^*_k (t, x).
	\end{equation}
	
	Define a test function $\varphi \in C^{1, 2}([T_{k-1}, T_k) \times \cS)$, such that $(t, x)$ is a maximum point of $h^*_k - \varphi$, with
	\begin{equation*}
		h^*_k(t, x) = \varphi (t, x) \quad \text{ and } \quad h_k(s, y) \leq h^*_k(s, y) \leq \varphi(s, y) \text{ when } (s, y) \in  [T_{k-1}, T_k) \times \cS.
	\end{equation*}
	Set $\gamma_n := \varphi(s_n, y_n) - h_k(s_n, y_n)$. As $\varphi$ is continuous and \eqref{eq:seqh*} holds, we have 
	\begin{equation*}
		0 \leq \gamma_n \rightarrow 0, \quad \text{ when } n \rightarrow \infty.
	\end{equation*}
	We introduce another sequence $\{\delta_n\}_{n}$ of strictly positive real numbers, satisfying
	\begin{equation*}
		\lim_{n \rightarrow \infty} \delta_n = 0 \quad \text{and} \quad \lim_{n \rightarrow \infty} \frac{\gamma_n}{\delta_n} = 0.
	\end{equation*}
	
	Let $\varepsilon > 0$ and define
	\begin{equation*}
		\rho_n := \inf \{ t \in [s_n, T_k] : | X(t; s_n, y_n, \emptyset, \emptyset) - y_n | \geq \varepsilon \} \wedge (s_n + \delta_n) \wedge T_k.
	\end{equation*}
	For $n$ large enough, we have $s_n + \delta_n < T_k$ and $\rho_n < T_k$. We apply the submartingale property of $h_k$, the fact that $h_k \leq \varphi$, and It\^o's formula to obtain 
	\begin{align*}
		h_k(s_n, y_n) \leq & \E [ h_k(\rho_n, X(\rho_n ; s_n, y_n, \emptyset, \emptyset)) ] \\
		\leq & \E [ \varphi(\rho_n, X(\rho_n ; s_n, y_n, \emptyset, \emptyset)) ] \\
		= & \varphi(s_n, y_n) - \E \Big[ \int^{\rho_n}_{s_n} \cL[\varphi] (u, X(u ; s_n, y_n, \emptyset, \emptyset)) du  \Big].
	\end{align*}	
	Rearranging the terms and dividing by $\delta_n$, we have
	\begin{equation*}
		\frac{1}{\delta_n} \E \Big[ \int^{\rho_n}_{s_n} \cL[\varphi] (u, X(u ; x_n, y_n, \emptyset, \emptyset)) du  \Big] - \frac{\gamma_n}{\delta_n} \leq 0.
	\end{equation*}
	Sending $n \rightarrow \infty$, the dominated convergence theorem and mean value theorem show that
	\begin{equation*}
		\cL[\varphi](t, x) \leq 0, \quad (t, x) \in [T_{k-1}, T_k) \times \cS.
	\end{equation*}
\end{proof}

\begin{proof}[Proof of Lemma \ref{lem:Tk_submart}]
	If $\btau = T_k$, \eqref{eq2:Tk_submart} is trivial. Then we only need to prove
	\begin{equation*}
		V_k(\btau, \xi) \one_{\{\btau < T_k \}} \leq \E \big[ V_k(\rho, X(\rho; \btau, \xi, \emptyset, \emptyset)) \one_{\{\btau < T_k \}} \big| \cF_{\btau} \big].
	\end{equation*}
	
	Define 
	\begin{equation*}
		\eta_n = \min\{  \rho, \max\{ T_k - 1/n, \btau \} \}, \quad n \geq N.
	\end{equation*}
	Here, constant $N$ is large enough, such that $T_k - 1/N > 0$. Note that $\eta_n$ is a stopping time. Moreover, $\btau \leq \eta_n \leq \rho$. If $\btau < T_k$, then $\eta_n < T_k$. Instead, if $\btau=T_k$, then $\eta_n = T_k$. Also, $\lim_{n \rightarrow \infty} \eta_n = \rho$. 
	
	Since $V_k$ is a stochastic subsolution and $\btau \leq \eta_n < T_k$ when $\btau < T_k$, \eqref{eq:submartingale} in Definition \ref{def:sto_sub} leads to
	\begin{equation*}
		V_k(\btau, \xi) \one_{\{\btau < T_k \}} \leq \E \big[ V_k(\eta_n, X(\eta_n; \btau, \xi, \emptyset, \emptyset)) \one_{\{\btau < T_k \}} \big| \cF_{\btau} \big].
	\end{equation*}
	Since $V_k$ is bounded and continuous and $X(\cdot; \btau, \xi, \emptyset, \emptyset)$ has continuous paths, dominated convergence theorem shows that
	\begin{align*}
		V_k(\btau, \xi) \one_{\{\btau < T_k \}} \leq & \lim_{n \rightarrow \infty} \E \big[ V_k(\eta_n, X(\eta_n; \btau, \xi, \emptyset, \emptyset)) \one_{\{\btau < T_k \}} \big| \cF_{\btau} \big] \\
		= & \E \big[ V_k(\rho, X(\rho; \btau, \xi, \emptyset, \emptyset)) \one_{\{\btau < T_k \}} \big| \cF_{\btau} \big]. 
	\end{align*}
	Hence, the claim \eqref{eq2:Tk_submart} holds.
\end{proof}

\begin{proof}[Proof of Theorem \ref{thm:optimalstrategy}]
	Clearly, $(\theta^*_{k:K}, \Lambda^*)$ is admissible by construction. We only need to prove the optimality.
	
	Denote constant $\lambda \in (0, 1)$ and $W_g >\sum^{K}_{i=1} w_i G_i$. Consider the perturbed continuation and intervention regions defined as follows:
	\begin{align}
		\cC_{i, \lambda} & := \left\{ (t, x) \in [T_{i-1}, T_i] \times \barS : V_i(t, x) + W_g(1 - \lambda)/\lambda < \cM[V_i] (t, x) \right\}, \\
		\cI_{i, \lambda} & := \left\{ (t, x) \in [T_{i-1}, T_i] \times \barS : V_i(t, x) + W_g(1 - \lambda)/\lambda \geq \cM[V_i] (t, x) \right\}.
	\end{align}
	By Lemma \ref{lem:Mproperty}, $\cM[V_i]$ is LSC on  $[T_{i-1}, T_i] \times \barS$. Then $\cC_{i, \lambda}$ is open and $\cI_{i, \lambda}$ is closed, respectively. We note that $\cC_{i, \lambda}$ can be empty when $\lambda$ is close to zero. $\cI_{i, \lambda}$ is decreasing in $\lambda$ and $\cI_i$ in \eqref{eq:I_i} satisfies
	\begin{equation*}
		\cI_i = \bigcap_{\lambda \in (0, 1)} \cI_{i, \lambda}.
	\end{equation*}

	{\bf Step 1}. Given $\lambda \in (0, 1)$ and $(s, y) \in [T_{k-1}, T_k] \times \barS$, we define a stopping time as
	\begin{equation*}
		\rho^{\lambda, k, s, y} := \inf \{ u \in [s, T_k] : (u, X(u; s, y, \emptyset, \emptyset)) \in \cI_{k, \lambda} \} \wedge T_k, 
	\end{equation*}
	and two functions as
	\begin{align*}
		h_k(s, y) & := \E \big[V_k(\rho^{\lambda, k, s, y},  X(\rho^{\lambda, k, s, y}; s, y, \emptyset, \emptyset)) \big], \\
		h_{k, \lambda} (s, y) & := \lambda V_k(s, y) + (1 - \lambda) h_k (s, y).
	\end{align*}
	Since $0 \leq V_k < W_g$, we have $0 \leq h_k < W_g$ and $0 \leq h_{k, \lambda} < W_g$. 
	
	{\bf Step 2}. When $k \neq K$, we verify that $h^*_{k, \lambda}$ is a USC viscosity subsolution on $[T_{k-1}, T_k] \times \barS$, that is, it satisfies Conditions (1), (2), and (4) in Definition \ref{def:vis_sub}. Since the value function $V_k$ is a viscosity supersolution, the comparison principle in Propositions \ref{prop:compare} and \ref{prop:compareTk} yields that $h_{k, \lambda} \leq h^*_{k, \lambda} \leq V_k$ on $[T_{k-1}, T_k] \times \barS$.
	
	The idea is to show that Lemma \ref{lem:borel} can be applied here. 
	\begin{itemize}
		\item Lemma \ref{lem:Tk_submart} yields the submartingale property, where $\rho^{\lambda, k, s, y} = T_k$ is allowed:
		\begin{equation*}
			V_k(s, y) \leq \E \big[V_k(\rho^{\lambda, k, s, y},  X(\rho^{\lambda, k, s, y}; s, y, \emptyset, \emptyset)) \big] \leq h_k(s, y), \quad (s, y) \in [T_{k-1}, T_k] \times \barS.
		\end{equation*}
		It implies that $V_k \leq h_{k, \lambda}$.
		\item The growth condition (2) in Definition \ref{def:sto_sub} of stochastic subsolutions holds since $0 \leq h_{k, \lambda} < W_g$.
		\item Condition (3) about the non-decreasing property in transactions can be shown as follows. First,
		\begin{align*}
			\cM[h_{k, \lambda}](s, y) & \geq \lambda \cM[V_k](s, y) + (1 - \lambda) \cM[h_k] (s, y) \\
			& \geq \lambda \cM[V_k] (s, y) + (1 - \lambda) \cM[V_k](s, y) \\
			& = \cM[V_k](s, y).
		\end{align*}
		Here, the second inequality uses $h_k \geq V_k$. 
		
		If $(s, y) \in \cI_{k, \lambda}$, then the stopping time $\rho^{\lambda, k, s, y} = s$, which leads to $h_k(s, y) = V_k(s, y)$ and $h_{k, \lambda}(s, y) = V_k(s, y)$. Hence,
		\begin{align*}
			\cM[h_{k, \lambda}](s, y) \geq \cM[V_k](s, y) \geq V_k(s, y) = h_{k, \lambda}(s, y), \quad \text{ for } (s, y) \in \cI_{k, \lambda}.
		\end{align*}
		The second inequality holds since $V_k$ satisfies Condition (3).
		
		Instead, if $(s, y) \in \cC_{k, \lambda}$, then  $V_k(s, y) + W_g(1 - \lambda)/\lambda < \cM[V_k] (s, y)$. It yields
		\begin{align*}
			\cM[h_{k, \lambda}](s, y) & \geq \lambda \cM[V_k](s, y) + (1 - \lambda) \cM[h_k] (s, y) \\
			& \geq \lambda V_k(s, y) + W_g (1 - \lambda) + (1 - \lambda) \cM[h_k] (s, y) \\
			& \geq \lambda V_k(s, y) + (1 - \lambda) h_k (s, y) \\
			& = h_{k, \lambda}(s, y).
		\end{align*}
		The third inequality holds since $W_g > h_k$ and $(1 - \lambda) \cM[h_k] (s, y) \geq 0$.
		
		\item We show that $(h_{k, \lambda}, V_{k+1:K})$ satisfies the submartingale condition (4) in Definition \ref{def:sto_sub}. Since $h_{k, \lambda}$ is a linear combination of $V_k$ and $h_k$, we only need to show that $(h_{k}, V_{k+1:K})$ satisfies this condition. Consider a random initial condition $(\btau, \xi)$ with $\btau \in [T_{k-1}, T_k]$. Fix a stopping time $\bar{\rho} \in [\btau, T]$ and $(\btau, \xi)$-admissible withdrawals $\theta_{k:K}$. For notational simplicity, we introduce the uncontrolled wealth process stopped at $\bar{\rho} \wedge T_k$ and $T_k$ as follows:
		\begin{align*}
			\bar{\eta} := X(\bar{\rho} \wedge T_k; \btau, \xi, \emptyset, \emptyset), \qquad \eta_{T_k} := X(T_k; \btau, \xi, \emptyset, \emptyset).
		\end{align*}
		Replacing $(s, y)$ in $\rho^{\lambda, k, s, y}$ with random initial conditions, we define
		\begin{align*}
			\rho_1 & := \rho^{\lambda, k, \btau, \xi}, \quad \qquad \eta_1 := X(\rho_1; \btau, \xi, \emptyset, \emptyset), \\
			\rho_2 & := \rho^{\lambda, k, \bar{\rho} \wedge T_k, \bar{\eta}}, \qquad \eta_2 := X(\rho_2; \rho_1, \eta_1, \emptyset, \emptyset).        %= \inf \{ u \in [\bar{\rho} \wedge T_k, T_k] : (u, X(u; \bar{\rho} \wedge T_k, \bar{\eta}, \emptyset, \emptyset)) \in \cI_{k, \lambda} \} \wedge T_k
		\end{align*}
		We note that $\rho_1 \leq \rho_2$ since $\btau \leq \bar{\rho} \wedge T_k$.
		
		The submartingale property can be shown as follows:
		\begin{align*}
			h_k(\btau, \xi) = & \E\big[V_k(\rho_1, \eta_1) \big| \cF_{\btau} \big] \\
			\leq & \E \big[\one_{\{ \bar{\rho} < T_k\}} V_k(\rho_2, \eta_2) + \one_{\{\bar{\rho} \geq T_k \}} V_k(T_k, X(T_k; \rho_1, \eta_1, \emptyset, \emptyset)) \big| \cF_{\btau} \big] \\
			= & \E \big[\one_{\{ \bar{\rho} < T_k\}} V_k(\rho_2, X(\rho_2; \bar{\rho}, \bar{\eta}, \emptyset, \emptyset)) + \one_{\{\bar{\rho} \geq T_k \}} V_k(T_k, X(T_k; \btau, \xi, \emptyset, \emptyset)) \big| \cF_{\btau} \big] \\
			= & \E \big[\one_{\{ \bar{\rho} < T_k\}} h_k(\bar{\rho}, \bar{\eta}) + \one_{\{\bar{\rho} \geq T_k \}} V_k(T_k, X(T_k; \btau, \xi, \emptyset, \emptyset)) \big| \cF_{\btau} \big] \\
			\leq & \E \big[\one_{\{ \bar{\rho} < T_k\}} h_k(\bar{\rho}, \bar{\eta}) + \one_{\{\bar{\rho} \geq T_k \}} \cH([T_k, \bar{\rho}], V_{k+1:K}, X(\cdot; T_k, \eta_{T_k}, \theta_{k:K}, \emptyset)) \big| \cF_{\btau} \big] \\
			= & \E\big[ \cH([\btau, \bar{\rho}], (h_k, V_{k+1:K}), X(\cdot; \btau, \xi, \theta_{k:K}, \emptyset)) \big| \cF_{\btau} \big].
		\end{align*}
		The first line is from the strong Markov property. The second line uses Lemma \ref{lem:Tk_submart} from $\rho_1$ to $\rho_2$. The third line uses the pathwise uniqueness:
		\begin{align*}
			& \eta_2 = X(\rho_2; \rho_1, \eta_1, \emptyset, \emptyset) = X(\rho_2; \bar{\rho}, \bar{\eta}, \emptyset, \emptyset) \quad \text{ when } \bar{\rho} < T_k, \\
			& X(T_k; \rho_1, \eta_1, \emptyset, \emptyset) = X(T_k; \bar{\tau}, \xi, \emptyset, \emptyset).
		\end{align*}
		Th fourth line uses the strong Markov property:
		\begin{align*}
			h_k(\bar{\rho} \wedge T_k, \bar{\eta}) & = \E \big[ V_k(\rho_2, X(\rho_2; \bar{\rho} \wedge T_k, \bar{\eta}, \emptyset, \emptyset))  \big| \cF_{\bar{\rho} \wedge T_k}\big],
		\end{align*}
		and the tower property.  The fifth line is from the submartingale property of $V_k$ from $T_k$ to $\bar{\rho}$. The last line uses the pathwise uniqueness and the definition of $\cH$.		
	\end{itemize}
	
	Moreover, the boundary condition at $0$ is satisfied as 
	\begin{equation*}
		h^*_{k, \lambda}(t, 0) \leq \sum^K_{i = k} w_i G_i, \; t \in [T_{k-1}, T_k].
	\end{equation*}
	
	Hence, the conditions to apply comparison principle are satisfied. We have $h_{k, \lambda} \leq h^*_{k, \lambda} \leq V_k$ on $[T_{k-1}, T_k] \times \barS$.
	
	{\bf Step 3}. Fix $n \in \{ 0, 1, 2, \ldots \}$. For notational simplicity, we write
	\begin{equation*}
		(\tau, \xi) := (\tau^{*, k}_n, \xi^{*, k}_n), \quad \rho^{\lambda} := \rho^{\lambda, k, \tau, \xi}.
	\end{equation*}
	
	The strong Markov property leads to
	\begin{align*}
		h_{k}(\tau, \xi) = & \E\big[V_k(\rho^{\lambda, k, s, y}, X(\rho^{\lambda, k, s, y}; s, y, \emptyset, \emptyset)) \big] \big|_{(s, y) = (\tau, \xi)} \\
		= & \E\big[V_k(\rho^{\lambda}, X(\rho^{\lambda}; \tau, \xi, \emptyset, \emptyset)) \big| \cF_\tau \big] \qquad \text{ when } \tau \leq T_k.
	\end{align*}
	With $V_k \geq h_{k, \lambda}$, we have
	\begin{equation*}
		V_k(\tau, \xi) \geq h_{k, \lambda} (\tau, \xi) = \lambda V_k(\tau, \xi) + (1 - \lambda) \E\big[V_k(\rho^{\lambda}, X(\rho^{\lambda}; \tau, \xi, \emptyset, \emptyset)) \big| \cF_\tau \big] \quad \text{ when } \tau \leq T_k.
	\end{equation*}
	It yields
	\begin{equation*}
		V_k(\tau, \xi) \geq \E\big[V_k(\rho^{\lambda}, X(\rho^{\lambda}; \tau, \xi, \emptyset, \emptyset)) \big| \cF_\tau \big] \quad \text{ when } \tau \leq T_k.
	\end{equation*}
	Lemma \ref{lem:Tk_submart} gives another side of inequality. Hence, we have 
	\begin{equation}\label{eq:rho_lambda}
		V_k(\tau, \xi) = \E\big[V_k(\rho^{\lambda}, X(\rho^{\lambda}; \tau, \xi, \emptyset, \emptyset)) \big| \cF_\tau \big] \quad \text{ when } \tau \leq T_k.
	\end{equation}
	
	{\bf Step 4}. By definition, $\rho^\lambda \leq \tau^{*, k}_{n+1} \wedge T_k$. Moreover, $\rho^\lambda$ is nondecreasing in $\lambda$. Then the limit $\rho := \lim_{\lambda \uparrow 1} \rho^\lambda$ exists and $\rho \leq \tau^{*, k}_{n+1} \wedge T_k$. Define two events
	\begin{equation*}
		B_1 := \{ \tau \leq T_k \} \cap \{ \tau^{*, k}_{n+1} \leq T_k \} = \{ \tau^{*, k}_{n+1} \leq T_k  \}, \quad 	B_2 := \{ \tau \leq T_k \} \cap \{ \tau^{*, k}_{n+1} > T_k \}.
	\end{equation*}
	% Then $B_1 \cup B_2 = \{ \tau \leq T_k \}$. Instead, if $B_2$ happens, when have $\rho^\lambda = T_k$. 
	We obtain
	\begin{align*}
		\cM[V_k](\rho, X(\rho; \tau, \xi, \emptyset, \emptyset)) \geq & V_k(\rho, X(\rho; \tau, \xi, \emptyset, \emptyset))  \\
		= & \lim_{\lambda \uparrow 1} V_k(\rho^\lambda, X(\rho^\lambda; \tau, \xi, \emptyset, \emptyset)) \\
		\geq & \liminf_{\lambda \uparrow 1} \Big( \cM[V_k] (\rho^\lambda, X(\rho^\lambda; \tau, \xi, \emptyset, \emptyset)) - W_g(1-\lambda)/\lambda \Big) \\
		\geq & \cM[V_k](\rho, X(\rho; \tau, \xi, \emptyset, \emptyset)) \qquad \text{ on } B_1.
	\end{align*}
	Here, the first line is due to that $V_k$ satisfies Condition (3) in Definition \ref{sto_sub}. The second line holds since $V_k$ is continuous and $X$ has continuous paths. The third line uses the definition of $\cI_{k, \lambda}$ and the fact that $(\rho^\lambda, X(\rho^\lambda; \tau, \xi, \emptyset, \emptyset)) \in \cI_{k, \lambda}$ when $B_1$ happens. The last line relies on the LSC property of $\cM[V_k]$. It implies that all inequalities are equalities and $\rho = \tau^{*, k}_{n+1}$ on $B_1$. Therefore,
	\begin{align*}
		V_k(\tau, \xi) = & \lim_{\lambda \uparrow 1} \E \big [V_k(\rho^\lambda, X(\rho^\lambda; \tau, \xi, \emptyset, \emptyset)) \big| \cF_{\tau} \big] \\
		\geq & \liminf_{\lambda \uparrow 1} \E \Big[ \cM[V_k] (\rho^\lambda, X(\rho^\lambda; \tau, \xi, \emptyset, \emptyset)) - W_g(1-\lambda)/\lambda \Big| \cF_{\tau} \Big] \\
		\geq & \E \Big[ \cM[V_k] (\tau^{*, k}_{n+1}, X(\tau^{*, k}_{n+1}; \tau, \xi, \emptyset, \emptyset)) \Big| \cF_{\tau} \Big] \\
		\geq & \E \Big[ V_k (\tau^{*, k}_{n+1}, X(\tau^{*, k}_{n+1}; \tau, \xi, \emptyset, \emptyset)) \Big| \cF_{\tau} \Big] \\
		\geq & 	V_k(\tau, \xi) \qquad \text{ on } B_1.
	\end{align*}
	The first line uses \eqref{eq:rho_lambda}. The second line is again from the fact that $(\rho^\lambda, X(\rho^\lambda; \tau, \xi, \emptyset, \emptyset)) \in \cI_{k, \lambda}$ when $B_1$ happens. The third line follows from Fatou's lemma and the LSC property of $\cM[V_k]$. The fourth line holds because $V_k$ satisfies the non-decreasing property (3) in Definition \ref{def:sto_sub}. The last line uses Lemma \ref{lem:Tk_submart}. Then the inequalities are all equalities. By the definition of $\xi^{*, k}_{n+1}$, we have
	\begin{align*}
		V_k(\tau^{*, k}_n, \xi^{*, k}_n) = & \E \Big[ \cM[V_k] (\tau^{*, k}_{n+1}, X(\tau^{*, k}_{n+1}; \tau, \xi, \emptyset, \emptyset)) \Big| \cF_{\tau^{*, k}_n} \Big] \\
		= & \E \Big[ V_k (\tau^{*, k}_{n+1},  \xi^{*, k}_{n+1}) \Big| \cF_{\tau^{*, k}_n} \Big] \qquad \text{ on } B_1. 
	\end{align*}

	Instead, on $B_2$, $\rho = T_k$. By dominated convergence theorem, we have 
	\begin{align*}
		V_k(\tau, \xi) = & \lim_{\lambda \uparrow 1} \E \big [V_k(\rho^\lambda, X(\rho^\lambda; \tau, \xi, \emptyset, \emptyset)) \big| \cF_{\tau} \big] \\
		= & \E \big [V_k(T_k, X(T_k; \tau, \xi, \emptyset, \emptyset)) \big| \cF_{\tau} \big] \quad \text{ on } B_2.
	\end{align*}
	
	Putting them together, we obtain
	\begin{equation}\label{eq:iter}
		\begin{aligned}
			V_k(\tau^{*, k}_n, \xi^{*, k}_n) \one_{\{\tau^{*, k}_{n} \leq T_k \}} = \E \Big[ & V_k (\tau^{*, k}_{n+1},  \xi^{*, k}_{n+1}) \one_{\{\tau^{*, k}_{n+1} \leq T_k \}} \\
			& + V_k(T_k, X(T_k; \tau^{*, k}_n, \xi^{*, k}_n, \emptyset, \emptyset)) \one_{\{\tau^{*, k}_{n} \leq T_k \} \cap \{\tau^{*, k}_{n+1} > T_k \}} \Big| \cF_{\tau^{*, k}_n} \Big]. 
		\end{aligned} 
	\end{equation}
	Iteratively applying this equality on $n = 0, 1, \ldots$, we have
	\begin{align*}
		V_k(t, x) =&  \lim_{n \rightarrow \infty} \E \Big[ V_k(\tau^{*, k}_n, \xi^{*, k}_n) \one_{\{\tau^{*, k}_n \leq T_k \}}  + V_k(T_k, X(T_k; \tau^{*, k}_0, \xi^{*, k}_0, \emptyset, \Lambda^*_k)) \one_{\{\tau^{*, k}_n > T_k \}} \Big] \\
		= & \E \Big[ V_k(T_k, X(T_k; \tau^{*, k}_0, \xi^{*, k}_0, \emptyset, \Lambda^*_k)) \Big] \\
		= &  \E \Big[ w_i (G_i - \theta^*_k)^+ +  V_{k+1}(T_k, X(T_k; \tau^{*, k}_0, \xi^{*, k}_0, \theta^*_k, \Lambda^*_k)) \Big].
	\end{align*}
	Here, the first line uses \eqref{eq:iter} and the definition of $\Lambda^*$. The second line follows from the dominated convergence theorem and the fact that $\p(\lim_{n \rightarrow \infty} \tau^{*, k}_n > T_k) = 1$. The last line relies on the definition of $\theta^*_k$ and the following fact: Since $V_k$ is a viscosity solution of \eqref{Tk_bc} at $T_k$ with the boundary condition \eqref{eq:cond_0} at $x=0$, we have
	\begin{equation*}
		V_k(T_k, x) = \min_{n \in \{0, 1, 2, \ldots\}} \cM^n[U_k](x), \quad x \in \barS,
	\end{equation*}
	where
	\begin{equation*}
		\begin{aligned}
			U_k(x) := & \inf_{0 \leq \theta_k \leq x_0} \left[ w_k (G_k - \theta_k)^+ + V_{k+1}(T_k, x_0 - \theta_k, x_1) \right], \quad x \in \barS.
		\end{aligned}
	\end{equation*}

	{\bf Step 5}. We repeat Step 1 to 4 above, until $k=K$. Only the terminal condition at $T$ is different and requires slight modifications. Then we finally have the desired result:
	\begin{align*}
		V_k(t, x) = &  \E \Big[ \sum^{K}_{i=k} w_i (G_i - \theta^*_k)^+ \Big].
	\end{align*}
\end{proof}

\end{document}